\documentclass{thesis}

\title{The Graph Isomorphism Problem:\\Local Certificates for Giant Action}
\author{Tim Frederik Seppelt}
\date{3.~Juli~2019}

\begin{document}

\begin{titlepage}
\linespread{1.5}
\vspace*{\fill}
\centering
\begin{LARGE}
\textsc{\thetitle}
\end{LARGE}
\vspace{2cm}

\vspace{1.5cm}

\begin{Large}
Bachelorarbeit in Mathematik \\
eingereicht an der Fakultät für Mathematik und Informatik\\
der Georg-August-Universität Göttingen\\
am \thedate\\
\end{Large}

\vspace{1cm}
Überarbeitete Fassung vom 23.~September~2019
\vspace{1cm}

\small{von}\\

\large{\theauthor}
\vspace{1cm}

\small{Erstgutachter:}\\

\large{Prof.\@ Dr.\@ Harald Andrés Helfgott}

\vspace{1cm}

\small{Zweitgutachter:}\\

\large{Prof.\@ Dr.\@ Stephan Waack}
\vspace*{\fill}
\end{titlepage}

\tableofcontents

\chapter{Introduction}
\label{sec:intro}
Graphs are fundamental combinatorial objects. Defined as a tuple $(V, E)$ of a set of vertices $V$ and a set of edges $E \subseteq V \times V$ connecting the vertices, they can be used to formalize structures in many areas of mathematics, computer science and other fields. For example, a graph could represent a molecule with atoms as vertices and covalent bonds as edges, or a social network with individuals as vertices and friendships as edges.

Given two graphs $X_1 = (V, E_1)$ and $X_2 = (V, E_2)$ on the same set of vertices $V$ of size $n$ one may ask whether these two graphs are essentially the same, i.e.\@ whether there exists a bijection $\pi: V \to V$ such that two vertices $v, w \in V$ are adjacent in $X_1$ if and only if their images $\pi(v)$ and $\pi(w)$ are adjacent in $X_2$. Being capable of answering this question efficiently would help us to search for chemical compounds in a database or to recognize patterns in social networks. Hence, it is desirable to design a fast algorithm which either computes such a bijection $\pi$ or confirms that no such $\pi$ exists -- in which case the two graphs are said to be not isomorphic. The \emph{Graph Isomorphism Problem} (GI) is the problem of computationally determining $\pi$ or confirming that no such $\pi$ exists.

Trivially, one could try to scan through all possible bijections and check whether they preserve adjacency. Since the total number of possible bijections is $n!$, this simplest strategy satisfies by Stirling's formula
\begin{equation} \label{eq:intro1}
T(n) \leq C n! \asymp \sqrt{n} \left( \frac{n}{e} \right)^n
\end{equation}
for a constant $C > 0$ and with $T(n)$ denoting the number of elementary operations required to obtain an answer given two graphs with $n$ vertices. This enormous growth behavior renders the trivial approach practically useless and theoretically uninteresting.

An open question in complexity theory is whether the Graph Isomorphism Problem can be decided in polynomial time, that is,
\begin{equation} \label{eq:intro2}
T(n) \leq n^C
\end{equation}
for a constant $C > 0$. In 2015, László Babai \cite{babai} proposed an algorithm which decides GI in quasi-polynomial time, that is,
\begin{equation} \label{eq:intro3}
T(n) \leq \exp \left((\log n)^C \right)
\end{equation}
for a constant $C > 0$. Babai's algorithm is based on the work of Eugene Luks \cite{luks}, who proved in 1981 that the isomorphicity of graphs with bounded degree can be tested in polynomial time. Pushing down the bound from exponential (\cref{eq:intro1}) to quasi-polynomial (\cref{eq:intro3}), was a huge leap. Nevertheless a polynomial bound (\cref{eq:intro2}) remains out of reach. Therefore, GI continues to be an interesting problem for exploring the connection between the two main complexity classes P and NP.

This thesis provides an explanation of Babai's algorithm with a particular focus on the case that cannot be dealt with by Luks' method. Group theoretic and combinatorial arguments are used to tackle this situation, which prevented significant progress for more than thirty years. The thesis extends the explanations provided by Harald Andrés Helfgott \cite{helfgott0}\footnote{For linguistic reasons, we will henceforth not refer to the French original \cite{helfgott0} but to its English translation \cite{helfgott}.} and by Babai \cite{babai,babai2}\footnote{The more recent \cite{babai2} represents an extended yet incomplete revision of the main article \cite{babai}. We will mostly refer to \cite{babai} except for one minor part of \cref{sec:agg-cert}.} himself. After introducing the relevant objects and some of their properties in \cref{sec:preliminaries}, we will describe important subroutines and general strategies which will be used repeatedly throughout the algorithm in \cref{sec:bblocks}. \Cref{sec:overview} gives an overview of the algorithm laying the ground work for \cref{sec:localcert} which treats Luks' barrier case with local certificates. In \cref{sec:complexity} we will justify the algorithms' quasi-polynomial complexity. We will reproduce Helfgott's stronger result that $C=3$ suffices in \cref{eq:intro3}. 

Even though the Graph Isomorphism Problem is of vast theoretical interest, the applications of Babai's algorithm in -- for instance -- chemistry or social sciences are limited. For practical use cases, faster heuristic methods suffice to test whether graphs are isomorphic or to find patterns in graphs. \cite{mckay} provides an overview of these strategies and their implementation in tools such as \emph{nauty} and \emph{Traces}.

\chapter{Preliminaries}
\label{sec:preliminaries}

In its most basic version the Graph Isomorphism Problem takes as an input two graphs and returns an adjacency-preserving bijection of the set of vertices or the confirmation that the two graphs are not isomorphic. In this section we want to describe a more general setting. In order to do so, we need some basic vocabulary. The starting point are permutation groups which arise from the bijections between the two graphs. We will then formalize the Graph Isomorphism Problem and finally state the definitions and properties of some combinatorial objects.

\section{Permutation Groups and their actions}
\label{sec:permgroup}
We state some notions for permutation groups and their actions. See \cite{dixon} for a comprehensive introduction.

\begin{defn}
Let $\Omega$ be a finite set. We denote the symmetric group acting with its canonical action on $\Omega$ as $\Sym(\Omega)$ and the alternating group respectively as $\Alt(\Omega)$. When simply $\Omega = \{1,\dots, n\}$, we denote these groups as $\Sym_n$ and $\Alt_n$. For $\omega \in \Omega$ and $\sigma \in \Sym(\Omega)$ we write $\omega^\sigma$ for $\sigma(\omega)$. We write $A^g \coloneqq \{a^g \mid a \in A\}$ for $A \subseteq \Omega$, $g \in G$ and $A^H \coloneqq \left\{ a^h \ \middle|\ a \in A,\ h \in H\right\}$ for $H \subseteq G$.
\end{defn}

We will write $(a\ b)$ for the transposition sending $a \mapsto b$ and $b \mapsto a$ and analogously $(a_1\ a_2\ \cdots \ a_m)$ for longer cycles. For groups $H$ and $G$ we write $H \leq G$ if $H$ is a subgroup of $G$. The index of a subgroup will be denoted as $\gindex{G}{H}$. We will generally use the superscript notation for maps, i.e.\@ for $\psi: A \to B$ we write $a^\psi$ instead of $\psi(a)$ for $a \in A$. For a set $B \subseteq \Sym(\Omega)$ we write $\generate{B}$ for the subgroup of $\Sym(\Omega)$ generated by the elements in $B$. For a map $\psi: A \to B$ and a subset $A' \subseteq A$ we write $\restrict{\psi}{A'}$ for the restriction of $\psi$ to the map $A' \to B$, $a' \mapsto \psi(a')$.

\begin{defn}
A \emph{permutation group} is a subgroup $G \leq \Sym(\Omega)$ of the symmetric group. $\Omega$ is said to be the \emph{domain} of $G$.
\end{defn}

The two largest permutation groups are $\Sym(\Omega)$ and $\Alt(\Omega)$. Babai calls them the \emph{giants}. In disguise they appear in \cref{sec:localcert} when we discuss the core group-theoretic algorithm.

Central to the entire thesis are strings since we will reduce the Graph Isomorphism Problem shortly to the problem of determining string isomorphisms.

\begin{defn}[Strings]
A string $\str x$ is a map $\Omega \to \Sigma$ from a finite set $\Omega$ of \emph{positions} to a finite set $\Sigma$ of \emph{letters}, the \emph{alphabet}. Let $\Sigma^\Omega$ denote the set of all strings.
\end{defn} 

In examples we will write strings simply as chains of characters, e.g.\@ $\str x = \mathsf{raspberry}$ for $\Omega = \{1,\dots, 9\}$ and the lower case English alphabet $\Sigma = \{\mathsf a, \dots, \mathsf z\}$.

The action of $\Sym(\Omega)$ on $\Omega$ induces an action on $\Sigma^\Omega$. We define $\str x^\sigma(i) \coloneqq \str x\left(i^{\sigma^{-1}}\right)$ for all $i \in \Omega$, $\str x \in \Sigma^\Omega$, $\sigma \in \Sym(\Omega)$. This twist is necessary in order to have for all $i \in \Omega$ the convenient relation
\begin{equation} \label{eq:perm-twist}
	\str x^{\sigma \tau}(i) = \str x\left( i^{(\sigma \tau)^{-1}} \right) = \str x \left( \left( i^{\sigma^{-1}} \right)^{\tau^{-1}} \right) = \str x^{\sigma}\left(i^{\tau^{-1}} \right) = \left( \str x^\sigma \right)^\tau (i).
\end{equation}

Moreover, the action of $\Sym(\Omega)$ on $\Omega$ induces an action on $\binom{\Omega}{k} = \{A \subseteq \Omega \mid \abs{A} = k\}$ for $k \geq 1$, the set of $k$-sized subsets of $\Omega$. Without lost of generality we can assume that $k \leq \abs{\Omega}/2$. 

\begin{defn}[Johnson groups, {\cite[§1.2.1]{babai}}]
	If the groups $\Alt(\Omega)$ or $\Sym(\Omega)$ act on $\binom{\Omega}{k}$ with their induced actions, they are called \emph{Johnson groups}. We will write $\Alt^{(k)}(\Omega)$ or $\Sym^{(k)}(\Omega)$, respectively. For $\Omega = \{1, \dots, m\}$, we write $\Sym_m^{(k)}$ and $\Alt_m^{(k)}$, respectively. 
\end{defn}

Johnson groups are the barrier that prevented the Graph Isomorphism Problem to be decided in polynomial time for over thirty years. They are the automorphism groups of Johnson schemes, cf.\@ \cref{lemma:johnsonaut}.

\begin{notation}[Stabilizers]
A subset $\Delta \subseteq \Omega$ is said to be invariant under $G$ if $\Delta^g = \Delta$ for all $g \in G$. We denote the \emph{setwise stabilizer} of $\Delta$, i.e.\@ the set of all $g \in G$ such that $\Delta^g = \Delta$, as $G_\Delta$. The \emph{pointwise stabilizer} will be denoted as $G_{(\Delta)}$. It is the set of all $g \in G$ satisfying $\delta^g = \delta$ for all $\delta \in \Delta$.
\end{notation}

For a permutation group $G \leq \Sym(\Omega)$, a subset $\Delta \subseteq \Omega$ and a group $H \leq G_{\Delta}$ we will write $\restrict{H}{\Delta}$ for the group of permutations restricted to the domain $\Delta$ and taken from $H$. Hence, $\restrict{H}{\Delta} \leq \Sym(\Delta)$.

\begin{defn}[Transitivity, Primitivity] \label{def:orbits}
Let $G \leq \Sym(\Omega)$ act on $\Omega$. The \emph{orbit} of an $\omega \in \Omega$ is the set $\omega^G$. The orbits partition the domain $\Omega$. If $G$ acting on $\Omega$ admits only one orbit, namely $\Omega$, the action is said to be \emph{transitive} and $G$ is called \emph{transitive} whenever the action is evident from the context.
The action of $G$ is said to be \emph{$k$-transitive} for $1 \leq k \leq \abs \Omega$ if the induced action on the set of distinct $k$-tuples drawn from $\Omega$ is transitive.

Suppose that $G$ is transitive. A non-empty subset $\Delta \subseteq \Omega$ is called a \emph{block} of $G$ if for each $g \in G$ either $\Delta^g = \Delta$ or $\Delta^g \cap \Delta = \emptyset$. For a block $\Delta$ the set $\left\{ \Delta^g\ \middle|\ g \in G \right\}$ is called a \emph{system of blocks}. Such a system partitions $\Omega$. A block is said to be \emph{non-trivial} if it is neither a singleton nor the entire domain. If $G$ acting transitively on $\Omega$ does not admit non-trivial blocks, the action is said to be \emph{primitive} and $G$ is called \emph{primitive}.

A transitive group action induces an action on any system of blocks $\mathcal{B}$. This can be rephrased by inferring a homomorphism $\psi: G \to \Sym(\mathcal{B})$. A system of blocks $\mathcal{B}$ is said to be \emph{minimal} if the induced action on it is primitive, i.e.\@ $G^\psi$ acts primitively on $\mathcal{B}$. The \emph{stabilizer} of a system of blocks $\mathcal B = \{\Delta_1, \dots, \Delta_m\}$ is the subgroup $N \leq G$ of permutations $g \in G$ with the property $\Delta_i^g = \Delta_i$ for all $1 \leq i \leq m$. We have, $N = \ker \psi$.
\end{defn}
\begin{example} \label{ex:alttrans}
Let $n \in \mathbb{N}$. The symmetric group $\Sym_n$ is $n$-transitive and hence $k$-transitive for all $1 \leq k \leq n$. The alternating group $\Alt_n$ is $(n-2)$-transitive: For given tuples $x \coloneqq  (x_1, \dots, x_{n-2})$ and $y \coloneqq (y_1,\dots, y_{n-2})$ we can construct an arbitrary permutation $\sigma \in \Sym_n$ mapping $x$ to $y$. Let $x', x''$ denote the two values not contained in $x$ and $y', y''$ the two values not contained in $y$. Then either $\sigma(x') = y'$ and $\sigma(x'') = y''$ or $\sigma(x') = y''$ and $\sigma(x'') = y'$ by bijectivity. The two possibilities differ by a single transposition, namely $(y'\ y'')$. Hence, either $\sigma$ or $\sigma (y'\ y'')$ is even and maps $x$ to $y$. We have constructed an element from $\Alt_n$ mapping $x$ to $y$.
\end{example}
\begin{example} \label{ex:dihedral1}
An example for a non-primitive group is the symmetry group of the square known as the dihedral group $D_8 \leq \Sym_4$. We number the vertices of the square clockwise as $1,2,3,4$. The elements of $G$ send $\Delta \coloneqq \{1,3\}$ either to itself or to its complement. Thus, $\mathcal{B} \coloneqq \left\{ \{1,3\}, \{2,4\} \right\}$ is a system of blocks. We have a map $D_8 \to \Sym(\mathcal{B})$ with the property that its kernel is generated by the reflections across the square's diagonals. The induced action is primitive, hence $\mathcal{B}$ is minimal.
\end{example}
The term minimal must be used carefully. It is the number of blocks that is minimal whereas their size is maximal in the sense that the action on any courser system of blocks would admit non-trivial blocks. We make a straightforward observation:

\begin{lemma} Let $A, B \subseteq \Omega$ be two non-disjoint blocks under the action of $G \leq \Sym(\Omega)$. Then $A \cap B$ is a block under the same action. \label{lemma:blockcap}
\end{lemma}
\begin{proof}
We distinguish two cases. First assume that $g \in G$ satisfies $A^g = A$ and $B^g = B$. Then clearly, $(A\cap B)^g = A^g \cap B^g = A \cap B$ since $g$ is bijective. Contrarily, let without loss of generality $A^g \cap A = \emptyset$. Then, $(A\cap B)^g \cap (A \cap B) = A^g \cap B^g \cap A \cap B = \emptyset$. Hence, $A\cap B \neq \emptyset$ is a block.
\end{proof}

\begin{defn} \label{def:permiso}
Two permutation groups $G \leq \Sym(\Omega)$ and $G' \leq \Sym(\Omega')$ are said to be \emph{permutation isomorphic} if their exists a group isomorphism $\varphi: G \to G'$ and a bijection $\iota: \Omega \to \Omega'$ satisfying for all $\omega \in \Omega$ and $g \in G$,
\[
	\iota\left(\omega^g \right) = \iota(\omega)^{\varphi(g)}.
\]
\end{defn}
Throughout the algorithm we will maintain two maps with similar properties. They originate in \cref{proc:iota}. We establish one final convention:

\begin{defn}
From a computational perspective, groups are given to us as sets of generators. We refer to such a set as a \emph{description} of the group. \emph{Describing} a group means to compute a description. A homomorphism $\varphi$ between two groups is stored a collection of tuples $(g, g^\varphi)$ where $g$ runs through a set of generators. Such a representation is said to be a \emph{description} of the homomorphism.
\end{defn}

The complexity estimates for many fundamental procedures, e.g.\@ in \cref{sec:orbitsblocks,sec:schreier-sims}, will depend on the size of the set of generators. Potentially, the set of generators could be as large as the permutation group itself. However, by \cite{babaiGenerators}, the size of any minimal set of generators for a subgroup of $\Sym_n$ is bounded by $2n$. We can hence assume that the sizes of the generating sets that we deal with are polynomially bounded in $n$, cf.\@ \cref{sec:measures-complex}.

\section{Graph and String Isomorphism Problem}

In \cref{sec:intro} we introduced a first basal version of the Graph Isomorphism Problem: Given two directed graphs $X_1 = (V, E_1)$ and $X_2 = (V, E_2)$ on a shared finite vertex set $V$ with edge sets $E_1, E_2 \subseteq V \times V$, we are asked to decide whether there exists an adjacency-preserving bijection $\pi: V \to V$, i.e.\@ a map satisfying $(v,w) \in E_1 \iff (\pi(v), \pi(w)) \in E_2$ for all $v,w \in V$. The result consists either of such a map $\pi$ or the verification that no such map exists. 

However, we will ask a more general question. Given the same input, we want to compute the set of all adjacency-preserving bijections $\pi$. If we can compute this set in quasi-polynomial time, cf.\@ \cref{eq:intro3}, we can clearly also answer the question from above in quasi-polynomial time.

Numerous data formats are used for encoding graphs. For instance, they can be stored as linked objects, edge lists or vertex lists. A widely used data format are adjacency matrices. The adjacency matrix of $G = (V,E)$ is a string $\str x: V \times V \to \{0,1\}$ with $\str x(v,w) = 1 \iff (v,w) \in E$. Thinking of graphs as such strings leads to a further reduction of the problem. Instead of asking for graph isomorphisms, one may ask for string isomorphism. 

\begin{defn} \label{def:stringiso}
Given two strings $\str x, \str y: \Omega \to \Sigma$ and a group $G \leq \Sym(\Omega)$, the \emph{String Isomorphism Problem} is the task of determining
\[
	\Iso_G(\str x, \str y) = \left\{\pi \in G \ \middle|\ \str x^\pi = \str y\right\},
\]
the set of \emph{string isomorphism} between $\str x$ and $\str y$ taken from $G$. We write $\Aut_G(\str x) = \Iso_G(\str x, \str x)$. If $\Iso_G(\str x, \str y) \neq \emptyset$, we can write
\[
	\Iso_G(\str x, \str y) = \Aut_G(\str x) \sigma
\]
for a $\sigma \in \Iso_G(\str x, \str y)$. We present the solution of the String Isomorphism Problem as a description of the group $\Aut_G(\str x)$ and a suitable $\sigma$.
\end{defn}

Imposing the condition that the bijections have to be taken from a certain group is very natural. For example, the problem of computing the adjacency-preserving bijections between the digraphs $X_1$ and $X_2$ reduces in polynomial time to the problem of finding $\Iso_G(\str x_1, \str x_2)$ where $\str x_i$ is the adjacency matrix of $X_i$ for $i=1,2$ and $G$ is the image of $\Sym(V)$ under the canonical homomorphism $\Sym(V) \to \Sym(V \times V)$. Intuitively, the group encodes which positions in the strings belong together as they represent edges sharing a common vertex. We have reduced the Graph Isomorphism Problem to the String Isomorphism Problem in polynomial time. If we can solve the latter in quasi-polynomial time, we can solve the former in quasi-polynomial time. The fact that a graph of $n$ vertices transforms into a string of length $n^2$ does only affect the implicit constants. Overtly, the String Isomorphism Problem is far more general than the initial Graph Isomorphism Problem as it deals with arbitrary alphabets with possibly more than two letters and arbitrary permutation groups.

For undirected graphs the reduction is similar. A graph $G = (V,E)$ with $E \subseteq \binom{V}{2}$ is encoded as an adjacency matrix $\str x: \binom{V}{2} \to \{0,1\}$ as above.

Now we are ready to state Babai's grand result whose justification will occupy us for the rest of this thesis.

\begin{theorem}[Babai, {\cite{babai}}]
	The String Isomorphism Problem can be solved in quasi-polynomial time in the length of the strings.
\end{theorem}

\section{Canonicity}
\label{sec:canonicity}
The notion of canonicity will be important throughout the entirety of the algorithm. In particular, we will be interested in computing canonical structures based on the two input strings $\str x, \str y: \Omega \to \Sigma$. Suppose that a group $G \leq \Sym(\Omega)$ acts on $\Omega$ and hence on the set of strings $\Sigma^\Omega$. Let $\mathcal X$ denote a set of structures, e.g.\@ a set of partitions of the domain $\Omega$. Suppose additionally that $G$ acts on $\mathcal X$, for example by permuting the sets of the partition. An assignment $C: \Sigma^\Omega \to \mathcal X$ is said to be \emph{canonical} with respect to $G$ if the action of $G$ commutes with $C$.

A basal example for canonical assignments are colorings according to letter multiplicities. For a string $\str x \in \Sigma^\Omega$ we color $\omega \in \Omega$ according to the multiplicity of $\str x(\omega)$ in $\str x$. For instance, $\str x = \mathsf{banana}$ would be mapped to the sequence $(1,3,2,3,2,3)$. Clearly, $G \leq \Sym(\Omega)$ acts on the set of such colorings by permuting the positions. The assignment commutes with the group actions.

A non-canonical assignment would be a coloring of $\omega \in \Omega$ according to whether $\str x(\omega)$ is surrounded by vowels. The example string $\str x$ would be mapped to $(0,1,0,1,0,0)$ where $1$ encodes yes. Applying the transposition $\tau = (1\ 2)$ results in the non-edible $\str x^{\tau} = \mathsf{abnana}$ which is mapped to $(0,0,0,0,1,0)$. Thus, this assignment does not satisfy $C(\str x^\tau) = \left(C(\str x)\right)^\tau$. $C$ is not canonical with respect to any supergroup of $\generate{(1\ 2)}$.

Canonical structures are an important tool for refuting isomorphicity. This happens based on the fact that the set $\Iso_G(\str x, \str y)$ can only be non-empty if $\Iso_G\left(C(\str x), C(\str y)\right)$ is non-empty. For example, two strings can only be isomorphic if there exists a permutation respecting their letter multiplicities. In any case, $\Iso_G(\str x, \str y) \subseteq \Iso_G\left(C(\str x), C(\str y)\right)$. This allows us to limit the set of possible isomorphisms. The challenge will be to compute canonical structures sufficiently fast.

\section{Measures of Complexity} \label{sec:measures-complex}
We will use Landau's and Vinogradov's notation for denoting the asymptotic growth behavior of functions under the tacit assumption that we take limits towards $+\infty$. For example, $f(n) = O(g(n))$ and $f(n) \ll g(n)$ are by definition equivalent to $\limsup_{n\to\infty} \abs{f(n)/g(n)} < \infty$ for $f,g: \mathbb{R} \to \mathbb{R}$. We write $f(n) \asymp g(n)$ if $f(n) = O(g(n))$ and $g(n) = O(f(n))$.
For measuring complexity, we introduce the following classes of functions.

\begin{defn} We call a function $f: \mathbb{R} \to \mathbb{R}$
\begin{enumerate}
	\item poly-logarithmic if $f(n) \ll (\log n)^{O(1)}$,
	\item polynomial if $f(n) \ll n^{O(1)}$,
	\item quasi-polynomial if $f(n) \ll \exp\left( (\log n)^{O(1)} \right)$.
\end{enumerate}
\end{defn}

The natural logarithm will be denoted by $\log$. Furthermore, $\log_b$ will denote the logarithm with respect to base $b$.

In \cref{sec:complexity} we will analyze the complexity of the entire algorithm. Our aim is to obtain a quasi-polynomial bound on the overall execution time, cf.\@ \cref{eq:intro3}. Central to the algorithm is the divide-and-conquer paradigm of recursively breaking down a problem into several subproblems. It will turn out that the tree spanned by the branching subproblems is of poly-logarithmic depth. Until we make this rigorous, we should keep in mind that we can in general afford polynomial non-recursive operations. This means that we do not care about the exponents in polynomial bounds. Even quasi-polynomial subroutines are affordable, although they occur rarely.

\section{Combinatorial structures}
\label{sec:combi}
Despite that this thesis focuses on Babai's group theoretic tools, we will still require some combinatorial structures and methods. The presentation will be very limited. See \cite[§§2.3--2.6]{babai} and \cite[§§2.3,~2.4]{helfgott} for a full overview.

In each of the following definitions we have a finite set $\Gamma$ on which we want to define a structure, and a set of colors $\mathcal C$, as well finite. We generalize the notion of colored graphs. In a vertex-colored graph every vertex is assigned a color. This induces on the one hand a partition of the vertex set in color classes and on the other hand a map from the set of vertices to the set of colors. We will use the same principle but for coloring vectors of vertices instead of single nodes. Note that we do not ask two adjacent vertices two have different colors as it is often done when studying graph colorings.

The graphs that are colored in this section are neither of the input graphs -- which have long ago been encoded as strings anyways. The following notions are meant to be used for abstract correspondences between the two input strings that restrict the set of possible isomorphisms as outlined in \cref{sec:canonicity}.

\begin{defn}[Relational structure, partition structure] \label{def:relstruct} Let $\Gamma$ be a finite set, $k \geq 1$. A \emph{$k$-ary relational structure} on $\Gamma$ is a tuple $\mathfrak{X} = \left(\Gamma, (R_i)_{i \in \mathcal C} \right)$ where $R_i \subseteq \Gamma^k$ for each $i \in \mathcal C$. We call the $R_i$ \emph{relations} on $\Gamma$. The \emph{rank} of $\mathfrak{X}$ is $\abs{\mathcal C}$.
If furthermore all $R_i$ are non-empty and partition $\Gamma^k$, we call $\mathfrak{X}$ a \emph{$k$-ary partition structure}. In this case, we describe $\mathfrak{X}$ by a function $c: \Gamma^k \to \mathcal{C}, x \mapsto i$ such that $x \in R_i$ and write $\mathfrak{X} = (\Gamma, c)$. $c(x)$ is said to be the \emph{color} of $x \in \Gamma^k$. 

For two $k$-ary relational structures $\mathfrak{X} = \left(\Gamma, (R_i)_{i \in \mathcal C} \right)$ and $\mathfrak{X}' = \left(\Gamma', (R'_i)_{i \in \mathcal C} \right)$ a bijection $f: \Gamma \to \Gamma'$ is said to be an \emph{isomorphism} between $\mathfrak{X}$ and $\mathfrak{X}'$ if for all $i \in \mathcal{C}$ and $(x_1,\dots,x_k) \in \Gamma^k$, $(x_1,\dots,x_k) \in R_i \iff (x_1^f,\dots,x_k^f) \in R'_i$. We write $\Iso(\mathfrak{X},\mathfrak{X}')$ for the set of isomorphisms and $\Aut(\mathfrak{X}) = \Iso(\mathfrak{X},\mathfrak{X}) \leq \Sym(\Gamma)$ for the set of automorphisms.
\end{defn}

We will impose further regularity conditions on the structures. Note that $\Sym_k$ acts naturally on $\Gamma^k$ by permuting the coordinates.

\begin{defn}[Configuration] A $k$-ary partition structure $\mathfrak{X} = (\Gamma, c) = \left(\Gamma, (R_i)_{i \in \mathcal C} \right)$ is called \emph{$k$-ary configuration} if the following conditions are satisfied: \nopagebreak \label{def:config}
	\begin{enumerate}
		\item If $c(x_1,\dots,x_k) = c(x'_1,\dots,x'_k)$, then $\left( x_i = x_j \iff x'_i = x'_j \right)$ for all $1 \leq i,j \leq k$. \label{item:config1}
		\item For all $\pi \in \Sym_k$ and all $i \in \mathcal{C}$ there exists $j \in \mathcal{C}$ such that $R_i^\pi = R_j$. \label{item:config2}
	\end{enumerate}
\end{defn}
The case $k=2$ offers some intuition for these conditions: We can think of a $2$-ary (or binary) configuration structure as an edge-colored complete digraph with vertices $\Gamma$. \Cref{item:config1} implies that when two nodes $y \neq z$, then $c(x,x) \neq c(y,z)$. Loops are colored differently than proper links. Since $\Sym_2 = \left\{(1),(1\ 2)\right\}$, \cref{item:config2} simplifies: We have for each color $i \in \mathcal{C}$ a color $j \in \mathcal{C}$ such that $R_i = R_j^- \coloneqq \{(y,x) \mid (x,y) \in R_j\}$. This means that $c(x,y)$ determines $c(y,x)$. If an out-edge is colored with the first color, then the reverse in-edge is colored with the second. We call the digraph $X_i = (\Gamma, R_i)$ the \emph{color-$i$ constituent graph} for $\mathfrak{X}$.

\begin{defn}[Coherent Configuration] A $k$-ary configuration $\mathfrak{X} = (\Gamma, c)$ is called \emph{$k$-ary coherent configuration} if the following additional condition holds: \label{def:coherent-config}
	\begin{enumerate} \setcounter{enumi}{2} \nopagebreak
		\item There exists a function $\gamma: \mathcal{C}^k \times \mathcal{C} \to \mathbb{Z}_{\geq 0}$ such that for all $v \in \mathcal{C}^k$ and $j \in \mathcal{C}$ and for any $x \in \Gamma^k$ such that $c(x) = j$, \label{item:config3} 
		\[
			\gamma(v,j) = \abs{\{ z \in \Gamma \mid c(x^i(z)) = v_i\ \forall 1 \leq i \leq k\}},		
		\]
		where $x^i(z) = (x_1,\dots, x_{i-1},z,x_{i+1},\dots,x_k)$. 
	\end{enumerate}
	A coherent configuration is said to be \emph{classical} if $k=2$. If a classical coherent configuration has only two colors, one for the \emph{diagonal} $\diag \Gamma = \{(x,x) \mid x \in \Gamma\}$ and one for its complement, then it is said to be \emph{trivial} or a \emph{clique}.
\end{defn}

We will again look at the case $k=2$ in order to provide some intuition. See \cref{fig:orbital-config} for an example. The existence of $\gamma$ guarantees that for each choice of three colors $i,j,l \in \mathcal{C}$ the number $\abs{\{z \in \Gamma \mid c(x,z) = j \land c(z,y) = l\}}$ is independent of $x,y \in \Gamma$ whenever $c(x,y) = i$. This means that for all pairs of nodes that are linked with a red edge, say, there is the same number of $2$-paths connecting them and containing first a blue and secondly a yellow edge. If $k > 2$, this generalizes to the number of colored walks following a given sequence of colors. 

An important example for classical coherent configurations occurs in the context of group actions:
\begin{lemma}[Orbital configuration] Let $G \leq \Sym(\Gamma)$ for a finite set $\Gamma$. Let $\{R_1, \dots, R_n\}$ denote the set of orbits of the induced action on $\Gamma \times \Gamma$, i.e.\@ the \emph{orbitals}. Then $\mathfrak{X} = (\Gamma; R_1, \dots, R_n)$ defines a classical coherent configuration, the \emph{orbital configuration}. \label{lemma:orbital-config}
\end{lemma}
\begin{proof}
Clearly, $\mathfrak{X}$ is a 2-ary partition structure. For \cref{item:config1} take an arbitrary orbital $R$ and let $(x_1,x_2), (x'_1,x'_2)\in R$. Thus $(x_1,x_2)^g = (x'_1,x'_2)$ for some $g \in G$. Clearly, $x_1 = x_2 \iff x_1^g = x_2^g \iff x'_1 = x'_2$. For \cref{item:config2} it is to show that $R^- = \{(y,x) \mid (x,y) \in R\}$ is an orbital. We have that $R = \{(x_0,y_0)^g \mid g \in G\}$ for some $(x_0, y_0) \in \Gamma^2$. Trivially, $R^-$ is the orbital containing $(y_0,x_0)$. The function $\gamma$ in \cref{item:config3} must be independent of $x$. We therefore choose three colors, here orbitals, $1 \leq v_1, v_2, j \leq n$ and $(x_1,x_2), (y_1,y_2) \in R_j$ arbitrarily. Let $g \in G$ such that $(x_1,x_2) = (y_1,y_2)^g$. Then,
\begin{align*}
\abs{\left\{z \in \Gamma\ \middle|\ (z,x_2) \in R_{v_1}, (x_1,z) \in R_{v_2}\right\}} 
& = \abs{\left\{z \in \Gamma\ \middle|\ (z,y_2^g) \in R_{v_1}, (y_1^g,z) \in R_{v_2}\right\}} \\
& = \abs{\left\{z \in \Gamma\ \middle|\ (z,y_2) \in R_{v_1}^g, (y_1,z) \in R_{v_2}^g\right\}} \\
& = \abs{\left\{z \in \Gamma\ \middle|\ (z,y_2) \in R_{v_1}, (y_1,z) \in R_{v_2}\right\}} 
\end{align*}
where we replaced $z$ by its image under the bijection $g^{-1}$ and used that the $R_{v_i}$ are as orbitals $G$-invariant. Concluding that $\gamma$ in \cref{item:config3} is well-defined completes the proof.
\end{proof}

\begin{example}
We compute the orbital configuration for the dihedral group $D_8$, that is the symmetry group of the square. We have $D_8 = \generate{(1\ 2\ 3\ 4), (2\ 4)}$. Clearly, the diagonal is one orbital. Applying the generators shows that the orbital configuration for $D_8$ acting on the set of four elements is of rank 3. We can think of the orbital configuration as a colored complete digraph. For clarity, we draw the edges of the three colors in three separate graphs, cf.\@ \cref{fig:orbital-config}.

Reverting to \cref{item:config3}, we note that for example $\gamma(\text{turquoise}, \text{ultramarine}, \text{violet}) = 0$ because there are no turquoise-ultramarine paths from the source to the sink of a violet edge. This number is independent of the chosen violet edge.
\end{example}
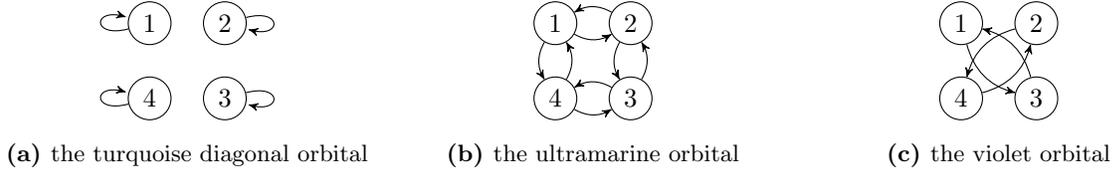
\begin{figure}
\centering
\begin{subfigure}{.3\textwidth}
\centering
\begin{tikzpicture}[node distance=1cm, auto, ->,>=stealth']
\tikzstyle{node} = [draw, circle, inner sep=.1cm]

\node [node] (a1) {$1$};
\node [node, right of=a1] (a2) {$2$};
\node [node, below of=a2] (a3) {$3$};
\node [node, below of=a1] (a4) {$4$};

\path (a1) edge [->, draw, loop left] (a1);
\path (a4) edge [->, draw, loop left] (a4);
\path (a2) edge [->, draw, loop right] (a2);
\path (a3) edge [->, draw, loop right] (a3);
\end{tikzpicture}
\caption{the turquoise diagonal orbital}
\end{subfigure}
\begin{subfigure}{.3\textwidth}
\centering
\begin{tikzpicture}[node distance=1cm, auto, ->,>=stealth']
\tikzstyle{node} = [draw, circle, inner sep=.1cm]

\node [node] (a1) {$1$};
\node [node, right of=a1] (a2) {$2$};
\node [node, below of=a2] (a3) {$3$};
\node [node, below of=a1] (a4) {$4$};

\path (a1) edge [->, draw, bend right] (a2);
\path (a2) edge [->, draw, bend right] (a3);
\path (a3) edge [->, draw, bend right] (a4);
\path (a4) edge [->, draw, bend right] (a1);
\path (a2) edge [->, draw, bend right] (a1);
\path (a3) edge [->, draw, bend right] (a2);
\path (a4) edge [->, draw, bend right] (a3);
\path (a1) edge [->, draw, bend right] (a4);
\end{tikzpicture}
\caption{the ultramarine orbital}
\end{subfigure}
\begin{subfigure}{.3\textwidth}
\centering
\begin{tikzpicture}[node distance=1cm, auto, ->,>=stealth']
\tikzstyle{node} = [draw, circle, inner sep=.1cm]

\node [node] (a1) {$1$};
\node [node, right of=a1] (a2) {$2$};
\node [node, below of=a2] (a3) {$3$};
\node [node, below of=a1] (a4) {$4$};

\path (a1) edge [->, draw, bend right] (a3);
\path (a3) edge [->, draw, bend right] (a1);
\path (a2) edge [->, draw, bend right] (a4);
\path (a4) edge [->, draw, bend right] (a2);
\end{tikzpicture}
\caption{the violet orbital}
\end{subfigure}
\caption{Orbital configuration of $D_8$}
\label{fig:orbital-config}
\end{figure}

\begin{defn} A $k$-ary coherent configuration $\mathfrak{X} = (\Gamma, c)$ is said to be \emph{homogeneous} if the $(x,\dots,x) \in \Gamma^k$ have the same color for all $x \in \Gamma$. If $k=2$, $\mathfrak{X}$ is called \emph{primitive} if it is homogeneous and all \emph{constituent graphs} $X_i = \left(\Gamma, \{(x,y) \in \Gamma^2 \mid c(x,y) = i\} \right)$ except the diagonal $X_{c(x,x)}$ are connected. It is said to be \emph{uniprimitive} if it is primitive and non-trivial. \label{def:coherent-config-prop}
\end{defn}

As an example, the orbital configuration of a transitive group is homogeneous. \Cref{fig:orbital-config} shows a non-primitive coherent configuration. The violet constituent graph is not connected. When working with local certificates, we will need another elementary result. First we introduce some elementary graph theoretic notions. The concerning digraph $X=(V,E)$ induces a binary partition structure $\mathfrak{X} = (V; E, V \times V \setminus E)$. We can thus regard the automorphism group of $X$ as the automophism group of $\mathfrak{X}$. 

\begin{defn} Let $X = (V,E)$ be a digraph. Then $X$ is said to be \emph{trivial} if $\Aut(X) = \Sym(V)$. That is the case if and only if $E \in \{\emptyset, V \times V, \diag V, V \times V \setminus \diag V\}$ where $\diag V = \{(v,v) \mid v \in V\}$, the \emph{diagonal}. $X$ is said to be \emph{irreflexive} if $E \cap \diag V = \emptyset$. It is \emph{biregular} if there exists $d \in \mathbb{N}$ such that $d = \deg^+(v) = \deg^-(v)$ for all $v \in V$ where $\deg^+(v)$ and $\deg^-(v)$ denote the in- and out-degree of the vertex $v$ respectively.
\end{defn}

\begin{lemma}[Degree awareness, {\cite[Observation~2.5.3]{babai}}] \label{lemma:biregular} Let $\mathfrak{X} = (\Gamma, c)$ be a classical coherent homogeneous configuration. Then every constituent graph $X = (\Gamma, R)$ is biregular. 
\end{lemma}
\begin{proof}
Let $v \in \Gamma$. Let amber be the color of the edges in $X$. Let, by \cref{item:config2}, burgundy be the color of the reverse edges. Since $\mathfrak{X}$ is homogeneous all loops have the same color. Call this color carmine. Then, by \cref{item:config3},
\begin{align*}
	\deg^+(v) & = \abs{\{z \in \Gamma \mid (v,z) \in R \}}
	= \abs{\{z \in \Gamma \mid c(v,z) = \text{amber}, c(z,v) = \text{burgundy} \}}  \\
	&= \gamma(\text{amber},\text{burgundy}, \text{carmine}).
\end{align*}
Thus $\deg^+(v)$ is in fact independent of $v$. The out-degree and analogously the in-degree of all vertices must be the same. From elementary graph theory we know that $\abs R = \sum_{v \in \Gamma} \deg^+(v) = \sum_{v \in \Gamma} \deg^-(v)$. Consequently, there exists $d \in \mathbb{N}$, such that $d = \deg^+(v) = \deg^-(v)$ for all $v \in \Gamma$. Hence, $X$ is biregular.
\end{proof}

In order to be able to state the Design Lemma, we introduce two more classes of coherent configurations following \cite[§2.5.5]{babai}. They are clearly homogeneous.
\begin{defn}
	Let $\mathfrak{X} = (\Gamma, c)$ be a classical coherent configuration. Then $\mathfrak{X}$ is said to be an \emph{association scheme} if $c(x,y) = c(y,x)$ for every $x,y \in \Gamma$.
\end{defn}  
\begin{defn} Let $t \geq 2$ and $k \geq 2t+1$. Let $\Omega$ be a set with $k$ elements and let $\Gamma = \binom{\Omega}{t}$. The \emph{Johnson scheme} $\mathfrak{J}(k,t) = (\Gamma; R_0, \dots, R_t)$ is an association scheme with the relations $(T_1, T_2) \in R_i \iff \abs{T_1 \setminus T_2} = i$ for $0 \leq i \leq t$. \label{def:johnson-scheme}
\end{defn}

The following result will be useful when we want to determine automorphisms which respect to Johnson schemes:

\begin{lemma}[Automorphisms of Johnson schemes, {\cite[Proposition~2.5.23]{babai}}] \label{lemma:johnsonaut}
	If $t \geq 2$ and $k \geq 2t+1$, then $\Aut(\mathfrak{J}(k,t)) = \Sym^{(t)}(\Omega)$ where $\Omega$ is the set of $k$ elements in the preceding definition.
\end{lemma}

\section{Twins and Symmetry Defects}
For combinatorial arguments we need more notions to describe correspondences and symmetries. A recurring theme are twins.
\begin{defn}
Let $G \leq \Sym(\Omega)$. Two elements $x, y \in \Omega$ are called \emph{twins}\footnote{Babai distinguishes strong und weak twins. These technical details shall not bother us. What we call \emph{twins} are \emph{strong twins} in Babai's nomenclature.} with respect to $G$ if $x=y$ or the transposition $(x\ y)$ is contained in $G$.
The notion of being a twin induces an equivalence relation on $\Omega$. The partitions inferred from this relation are said to be the \emph{twin classes} of $\Omega$.
\end{defn}
Clearly, the relation is symmetric and reflexive. Transitivity is implied by the fact that $(y\ z)(x\ y)(y\ z) = (x\ z)$.
\begin{defn}
Let $G \leq \Sym(\Omega)$. Let $T \subseteq \Omega$ be a smallest subset of $\Omega$ such that $\Omega \setminus T$ is symmetrical, i.e. all distinct pairs drawn from $\Omega \setminus T$ are twins with respect to $G$. We call $\abs T$ the \emph{symmetry defect} and $\abs T/\abs \Omega$ the \emph{relative symmetry defect} of $G$.
\end{defn}
In other words, $\Omega \setminus T$ is the biggest twin class in $\Omega$ by cardinality. For example, if we take \[G \coloneqq \Sym(\{1,\dots,4\}) \times \Sym(\{5,\dots,12\}) \leq \Sym(\{1,\dots,12\}), \] then $T = \{1,\dots, 4\}$ and the relative symmetry defect is $1/3$.
The preceding definitions for general groups naturally extend to relational structures.
\begin{defn} \label{def:symrelstruct}
	Let $\mathfrak{X}$ be a relational structure on $\Gamma$. Two elements $x, y \in \Gamma$ are called \emph{twins} with respect to $\mathfrak{X}$ if they are twins with respect to $\Aut(\mathfrak{X})$. Analogously, the \emph{(relative) symmetry defect} of $\mathfrak{X}$ is the (relative) symmetry defect of $\Aut(\mathfrak{X})$.  
\end{defn}

As an example, we state one of Babai's corollaries which will be relevant in \cref{sec:agg-cert}. Such results will allow us to perform efficient reduction. Remember that we have identified digraphs and binary partition structures in \cref{sec:combi}.

\begin{lemma}[Symmetry defect of digraphs, {\cite[Corollary~2.4.13]{babai}}] \label{cor:2413}
	Let $X=(V,E)$ be a non-trivial irreflexive biregular digraph. Suppose $\abs V \geq 4$. Then the relative symmetry defect of $X$ is $\geq 1/2$.
\end{lemma}

\section{Colored partitions}
We summarize some definitions from \cite[§5]{babai} which will be necessary to formulate the Design Lemma.
\begin{defn} A \emph{colored partition} $\Pi$ of a set $\Omega$ is a coloring of the elements of $\Omega$ along with a partition of each color class, i.e.\@ each set of elements of a given color. Let $C_1, \dots, C_r$ denote the color classes and $\{B_{ij} \mid 1 \leq j \leq k_i\}$ the blocks of $C_i$ where $k_i$ denotes the number of blocks in $C_i$. Let $0 < \alpha \leq 1$.  Then $\Pi$ is said to be a \emph{colored $\alpha$-partition} if it satisfies the following conditions: \label{def:colored-partition}
\begin{enumerate}
	\item For each color class $C_i$ such that $\abs{C_i} \geq 2$ all blocks have size $\abs{B_{ij}} \geq 2$.
	\item All blocks $B_{ij}$ satisfy $\abs{B_{ij}} \leq \alpha \abs{\Omega}$.
	\item All blocks within the same color class are of same size. \label{item:apart3}
\end{enumerate} \label{def:apart}
\end{defn}
The original definition does not contain \cref{item:apart3}. However, we can canonically refine any structure which satisfies only the other two conditions such that it satisfies the last. We encode the size of the block in the color. Automorphism and isomorphisms of colored partitions must preserve the color classes but can permute the blocks within them. Of course, also the elements within a block can be permuted.

\chapter{Algorithmic Building Blocks}
\label{sec:bblocks}
In this chapter we will present basal algorithmic strategies which will be used throughout the entirety of this thesis. First, we will explain how to determine orbits and blocks of group actions. Schreier-Sims' method for computing fundamental group theoretic objects such as generators, subgroups, kernels and cosets etc.\@ will be the subject of the second section. Thirdly, a brief introduction to Luks' techniques for reduction and recurrence will be provided. At the end of the chapter we will discuss other parts of Babai's algorithm such as the Extended Design Lemma which will only appear as a black box.

\section{Orbits and Blocks}
\label{sec:orbitsblocks}

Orbits and blocks are fundamental objects for describing group actions, cf.\@ \cref{def:orbits}. At numerous points, we will need to map out the orbit structure or require systems of blocks. From a computational perspective, group actions are described by the images of the points in the domain under the action of the generators. See \cite[§§4.1,~4.3]{holt} for background information and more efficient algorithms.

We can determine the orbits of $G \leq \Sym(\Omega)$ given by a set of generators $A$ by computing the connected components of the graph with vertices $\Omega$ that contains an edge linking $\omega_1, \omega_2 \in \Omega$ whenever there exists an $a \in A$ sending $\omega_1$ to $\omega_2$. Constructing the graph takes $O(\abs A \abs \Omega)$. For every vertex we apply a breadth-first or depth-first search in $O\left(\abs \Omega^2\right)$ obtaining the connected components which are precisely the orbits of the action. This takes in total $O\left(\abs A \abs \Omega + \abs\Omega^3 \right)$. By a remark in \cref{sec:permgroup}, $\abs A$ can be assumed to be polynomially bounded in $\abs \Omega$. Hence, computing orbits takes polynomial time.

The orbit of a single element can be rapidly extracted from this result. Moreover, we can determine in polynomial time whether an action is transitive.

Suppose now that $G$ is transitive. We want to determine whether its action is furthermore primitive. Following \cite[Proposition~4.4]{sims}, we analyze for given distinct $a,b \in \Omega$ the graph $\mathcal{G}_{a,b} \coloneqq \left( \Omega, \left\{ \{a,b\}^g\ \middle|\ g \in G \right\} \right)$ with vertices $\Omega$ and edges taken from the orbit of $\{a,b\}$ under the induced action of $G$ on $\binom{\Omega}{2}$. Again using breadth- or depth-first search we can determine the connected components of $\mathcal{G}_{a,b}$ in polynomial time. The connected components cannot be singletons since by transitivity every $\Omega \ni c = a^g$ is adjacent to $b^g$ for a suitable $g \in G$. 

Each connected component $C$ forms a block: Let $g \in G$ be fixed. Suppose that  $C^g \cap C \neq \emptyset$. Then there exist $c, d \in C$ such that $c = d^g$. Thus, in $C$ we have a path $d = a^{h_1} \rightarrow b^{h_1} \rightarrow \dots \rightarrow a^{h_k} \rightarrow b^{h_k} = c = d^g$ for certain $h_i \in G$. Let $e \in C$ be arbitrary.  This means that there exist $h'_i \in G$ such that,
\[
	e = a^{h'_1} \rightarrow  \dots \rightarrow  b^{h'_{k'}} = d = a^{h_1} \rightarrow \dots \rightarrow b^{h_k} = d^g = b^{h'_{k'}g}  \rightarrow \dots \rightarrow a^{h'_1 g} = e^g.
\]
Hence, $e^g \in C$ and $C^g \subseteq C$. Moreover, if $a,b \in C$, then $C$ is the smallest block containing $a$ and $b$. Suppose that there exists a smaller block $C'$ containing $a$ and $b$, i.e. $\abs{C'} < \abs{C}$. Then by \cref{lemma:blockcap}, $C' \cap C \subsetneq C$ is as well a block. Since $C$ is connected there exists an edge $\{c,c'\}$ in $C$ connecting $C' \cap C$ with its complement, i.e.\@ $c \in C \setminus (C' \cap C)$ and $c' \in C' \cap C$. By construction, there exists $g \in G$ such that $\{a,b\}^g = \{c,c'\}$. But $C' \cap C$ is supposed to be a block implying that $c \in C' \cap C$, a contradiction.

We observe that $G$ is primitive if and only if there exists $a \in \Omega$ such that $\mathcal{G}_{a,b}$ is connected for all $b \in \Omega$: 
Let $a,b \in \Omega$ such that $\mathcal{G}_{a,b}$ is disconnected. Then the connected components of this graph form non-trivial blocks. Thus, $G$ is not primitive. Conversely, assume that $\Gamma$ is a non-trivial block of $G$. Choose $a \in \Omega$ with the property stated in the claim. Then $\mathcal{G}_{a,b}$ is connected for arbitrary $b \in \Gamma$. Choose $c \in \Gamma \setminus \{b\}$. $\Gamma$ is a block containing $c$ and $b$. But since $\mathcal{G}_{a,b}$ is connected, its only connected component $\Omega$ is the smallest block containing $c$ and $b$. $\Gamma$ must equal $\Omega$, again a contradiction.

We have justified a method for testing primitivity in polynomial time. We pick an arbitrary $a \in \Omega$ and scan through all $b \in \Omega$ checking whether $\mathcal{G}_{a,b}$ is disconnected. If one of the graphs is disconnected, we know that $G$ is imprimitive, otherwise $G$ is primitive. This idea for verifying primitivity was efficiently implemented by Atkinson, cf.\@ \cite[§4.3]{holt}. However, any polynomial algorithm suffices, as always, for our needs. \Cref{fig:orbital-config} shows an example for the correspondence of connected components and blocks.

Moreover, we can compute a minimal system of blocks in polynomial time supposing that $G$ is imprimitive. For suitable $a,b \in \Omega$ we compute the smallest block $B$ containing both elements as described above. We then analyze the action of $G$ on the systems of blocks given by $\{B^g \mid g \in G\}$. If it is primitive, we have found the desired minimal system. In the contrarian case, we compute a new system of blocks for this action. The new blocks are courser than their predecessors guaranteeing that this process terminates after $\leq \abs \Omega$ iterations.

It remains to be seen how to compute the stabilizer of a system of blocks $\mathcal{B}$ in polynomial time. Reverting to \cref{def:orbits} and anticipating \cref{sec:schreier-sims}, this is simple. We consider the homomorphism $G \to \Sym(\mathcal{B})$ induced by the action on the blocks. Using Schreier-Sims we can compute its kernel which is precisely the desired stabilizer.

\section{Schreier-Sims and derived algorithms}
\label{sec:schreier-sims}

At many points throughout the algorithm we have to compute group theoretic objects such as stabilizers, kernels and preimages. Schreier-Sims is the fundamental toolkit that allows us to complete these tasks in polynomial time. Note that there exist many faster and more sophisticated strategies for solving these problems deterministically or employing randomization, cf.\@ \cite{seress}. Since we are interested in an overall quasi-polynomial bound, we can afford applying subprocedures with any polynomial complexity. Following \cite[§1.2]{luks}, we state a basic version of Schreier-Sims which is definitely rapid enough for our needs.

\begin{proc}[Schreier-Sims, {\cite[§2.1.1]{helfgott}}, {\cite[§1.2]{luks}}]
	\label{proc:schreier-sims}
	\Input a set $\Omega = \{x_1,\dots,x_n\}$, a set $A$, such that $A$ generates $G \leq \Sym(\Omega)$.\\
	\Output sets $C_i$ of representatives of $G_i/G_{i+1}$, such that $\bigcup_{i\leq j <n-1} C_j$ generates the chain of pointwise stabilizer $G_i = G_{(x_1,\dots, x_i)}$ for $0 \leq i < n-1$, i.e.\@
	\[ G = G_0 \geq G_1 \geq \dots \geq G_{n-1} = \{1\}. \]
	\Complexity $O(n^5+n\abs{A})$ or $O(n^5)$ by the following remark
\end{proc}

We will roughly present the algorithm in order to justify the claim on the complexity and the derived procedures stated below. Central to the procedure is the function \emph{filter} which determines the generating set $C_i$ that a group element belongs to.

\begin{procedure}
	{Filter}
	{a group element $g \in G$, current $(C_i)$}
	{modified $(C_i)$}
	{filter}
	\item \textbf{for} $i=0, \dots, n-2$:
	\item \cindent \textbf{if} $\gamma^{-1} \alpha \in G_{i+1}$ for some $\gamma \in C_{i}$: \label{line:filter-cond}
	\item \cindent \cindent $\alpha \leftarrow \gamma^{-1}\alpha$
	\item \cindent \cindent add $\alpha$ to $C_i$, remember $x_{i+1}^\alpha$
	\item \cindent \cindent \textbf{return} $(C_i)$, \textit{enlarged set of representatives}
	\item \cindent \textbf{end}
	\item \textbf{end}
	\item \textbf{return} $(C_i)$, \textit{nothing changed} 
\end{procedure}

The condition in line~\ref{line:filter-cond} is equivalent to $x_{i+1}^\alpha = x_{i+1}^\gamma$. This can be surely tested in constant time. We store the elements of $C_i$ in a dictionary indexed by $x_{i+1}^\alpha$. This allows us to find a $\gamma \in C_i$ which satisfies the condition in $O(1)$. In this way, line~\ref{line:filter-cond} contains just a dictionary look-up and the entire procedure runs in $O(n)$. We can now state the instructions for \cref{proc:schreier-sims}:
\begin{code}
\item initialize $C_i \leftarrow \{1\}$ for all $i = 0, \dots, n-1$
\item filter $A$, the set of generators for $G$
\item filter $C_iC_j$ for $i \geq j$
\item \textbf{return} $(C_i)$
\end{code}

For all $0 \leq i < n-1$, it holds that $\abs{C_i} = \gindex{G_i}{G_{i+1}} \leq n-i \leq n$, as the coset representatives of $G_{i+1}$ in $G_i$ are limited by the possible (non-stabilized) images of $x_{i+1}$. Therefore, we obtain the desired result in $O(n^5+n \abs A)$.

Schreier-Sims allows us to refine the set of generators $A$. As $G_0 = G$ is generated by $C = \bigcup_{j < n-1} C_j$, we have $\abs{C} \leq \sum_{i=0}^{n-2} (n-i)$. Thus, we can find a set of generators for $G$ of size $O(n^2)$ ensuring that the new generators are products of the former generators. Therefore, we may assume from now on that all groups acting on sets of size $n$ are generated by $O(n^2)$ elements.

Several basal tasks can be fulfilled in polynomial time using Schreier-Sims. Let $G \leq \Sym( \Omega)$ and $\abs \Omega = n$.

First of all, we can compute $\abs{G}$ and verify whether $g \in \Sym(\Omega)$ satisfies $g \in G$. We have that $\abs G = \prod_{i=0}^{n-2} \abs{C_i}$ and $g \in G$ if and only if $\abs{\generate{G,g}} = \abs G$.
Secondly, we can describe the pointwise stabilizer $G_{(\Delta)}$ of a set $\Delta \subseteq \Omega$ by ordering $\Omega$ such that the elements of $\Delta$ come first, calling Schreier-Sims and returning $\bigcup_{ \abs{\Delta} \leq j < \abs{\Omega}-1} C_j$. This is further explained in \cite[§1.2]{luks} and \cite[§5.1.1]{seress}.

Of interest will also be whether the group $G$ contains $\Alt(\Omega)$, or even $\Sym(\Omega)$. Let without loss of generality $\Omega = \{1,\dots,n\}$. For this, \cite[§10.2]{seress} does not only describe excelling algorithms but also presents useful generators of $\Sym_n$ and $\Alt_n$, $n\geq 2$, which together with the preceding membership test suffice to verify $\Alt_n \leq G$ polynomially:
\begin{align}
\Sym_n &= \generate{(1\ 2\ \cdots\ n), (1\ 2)} \\
\Alt_n &= \begin{cases}
		    \generate{(3\ 4\ \cdots\ n), (1\ 2\ 3)}, & n \text{ odd} \\
		    \generate{(1\ 2)(3\ 4\ \cdots\ n), (1\ 2\ 3)}, & n \text{ even} \\
          \end{cases} \label{eq:generators-alt}
\end{align}

Following \cite[Exercise~2.1c]{helfgott}, we are furthermore able to describe the preimage $\varphi^{-1}(H)$ of a subgroup $H \leq \Sym(\Omega')$ under a given homomorphism $\varphi: G \to \Sym(\Omega')$ if we assume that $\abs{\Omega'} \ll \abs{\Omega}^{O(1)}$. Similarly, we can compute the preimage of a single element $\sigma \in \Sym(\Omega')$ which is either empty or a coset of $\ker \varphi$, cf. \cite[§5.1.2]{seress}. Homorphisms are given to us, in general, as a set of tuples $\left(g, g^\varphi \right)$ where $g$ runs through a set of generators of the domain group.

When looking at a subgroup $H \leq G$ which has bounded index $\gindex{G}{H} \ll \abs{\Omega}^{O(1)}$ and admits polynomial time membership testing, we are able to describe $H$ in terms of generators of $G$. Furthermore we can compute coset representatives for $H$ in $G$. We do so by using the alternative chain of stabilizers
\[
G \geq H = H_0 \geq H_1 \geq \dots \geq H_{n-1} = \{1\}.
\]
\Cref{proc:filter} needs to be modified for being able to return a set of representatives $C_{-1}$ for $G/H$. For $i=-1$, line~\ref{line:filter-cond} changes to $\gamma^{-1} \alpha \in H$ which can be verified in polynomial time by assumption. Here we find a suitable $\gamma \in C_{-1}$ by iteratively checking $\abs{C_{i-1}} \leq \gindex{G}{H}$ elements. The group $H$ is then described by $\bigcup_{0\leq i < n-1}C_i$ and $C_{-1}$ is a set of coset representatives for $H$ in $G$. Let $\mu(n) \ll n^{O(1)}$ denote the complexity of the membership test. Then the complexity of this modified version of \cref{proc:schreier-sims} is $O\left(\left(\abs A + n^4 + \gindex{G}{H}n^2 \right)\left(\mu(n) \gindex{G}{H} + n \right) \right)$ which is by assumption polynomial.

In some situations we cannot demand that $\gindex{G}{H}$ is polynomially bounded. For example, it may be the case that $\gindex{G}{H}$ grows quasi-polynomially in $n$, e.g. $\gindex{G}{H} \ll n^{O(\log n)}$. By the complexity estimate from above, a call to Schreier-Sims costs then $ n^{O(\log n)}$ assuming that $\abs A$ is not too large. This will be affordable. 

\section{Luks' method}

The following strategies are based on \cite{luks} who proved that the Graph Isomorphism Problem for graphs with bounded degree can be decided in polynomial time. The algorithm that he proposed includes two kinds of reductions which are used many times throughout Babai's algorithm. See \cite[§2.2]{helfgott} and \cite[§3.1]{babai} for more details.

The rather simple strategy is called \emph{weak Luks reduction} and is used whenever we want to pass to a subgroup. The subgroup's index determines the number of subproblems that we have to deal with after the reduction. Clearly, we will require some bounds on that index to ensure efficiency. \emph{Strong Luks reduction} is the more sophisticated method. It can be used to recur on a partition which is invariant under the group's action. The most basal example is the recurrence on orbits. For both strategies we require the notion of partial isomorphisms which formalizes the idea of considering parts of the input only.

\begin{defn}[Partial isomorphisms] Let $\str x, \str y: \Omega \to \Sigma$ be two strings, $G \leq \Sym(\Omega)$ and $\Delta \subseteq \Omega$ a subset called the \emph{window}, then
\[
	\Iso_G^\Delta(\str x, \str y) = \left\{ \sigma \in G\ \middle|\ \str x(x) = \str y(x^\sigma)\ \forall x \in \Delta \right\}
\] 
denotes the set of \emph{partial isomorphisms} with respect to $\Delta$. Correspondingly, $\Aut_G^\Delta(\str x) = \Iso_G^\Delta(\str x, \str x)$ denotes the set of \emph{partial automorphisms}.
\end{defn}

Clearly, $\Iso_G^\Omega (\str x, \str y) = \Iso_G(\str x, \str y)$, cf.\@ \cref{def:stringiso}. The following lemma is crucial in both situations.

\begin{lemma}
\label{lemma:isoaut}
Let $\str x, \str y: \Omega \to \Sigma$ be two strings, $K, K_1, K_2 \subseteq \Sym(\Omega), \sigma \in \Sym(\Omega)$ and $\Delta \subseteq \Omega$ the window. Then the following holds:
\begin{enumerate}
\item \emph{(Shift identity)} $\Iso^\Delta_{K\sigma}(\str x, \str y) = \Iso^\Delta_K\left( \str x, \str{y}^{\sigma^{-1}} \right) \sigma$,
\item \label{lemma:isoaut-item2} $\Iso^\Delta_{K_1 \cup K_2}(\str x, \str y) = \Iso_{K_1}^\Delta (\str x, \str y) \cup \Iso_{K_2}^\Delta (\str x, \str y)$,
\item \label{lemma:isoaut-item3} If furthermore $G \leq \Sym(\Omega)$ is a subgroup which leaves $\Delta$ invariant, then $\Aut^\Delta_G(\str x) \leq G$. For all $\sigma \in \Sym(\Omega)$, $\Iso_{G\sigma}^\Delta(\str x, \str y)$ is either empty or of the form $\Aut_G^\Delta(\str x)\tau$ for any $\tau \in \Iso_{G\sigma}^\Delta(\str x, \str y)$.
\item \emph{(Chain rule, \cite[Proposition~3.1.7]{babai})} %
%Let $\Delta_1, \Delta_2 \subseteq \Omega$, $\Delta_1$ invariant under $G$. Then for $G' = \Aut_G(\str x)$ and $\sigma, \tau$ such that $\Iso^{\Delta_1}_{G\tau}(\str x, \str y) = G'\tau$,
%$$\Iso^{\Delta_1 \cup \Delta_2}_{G\sigma}(\str x, \str y) = \Iso^{\Delta_2}_{G'\tau}(\str x, \str y) = \Iso^{\Delta_2}_{G'} \left( \str x, \str{y}^{\tau^{-1}} \right)\tau.$$
Let $\Delta_1, \Delta_2 \subseteq \Omega$ invariant as sets under $G \leq \Sym(\Omega)$. Then for a subgroup $G_1 \leq G$ and $\tau \in G$ such that $\Iso^{\Delta_1}_G(\str x,\str y) = G_1 \tau$, it holds that
\[ \Iso^{\Delta_1 \cup \Delta_2}_G(\str x, \str y) = \Iso^{\Delta_2}_{G_1\tau}(\str x, \str y) = \Iso^{\Delta_2}_{G_1} \left( \str x, \str{y}^{\tau^{-1}} \right) \tau. \]
\end{enumerate}
\end{lemma}

\begin{proof}
Clearly, \cref{lemma:isoaut-item2,lemma:isoaut-item3} hold. It is worth looking at the proofs of the other claims in order to internalize the implications of \cref{eq:perm-twist}. The shift identity follows basically from the definition:
\[
\Iso^\Delta_{K\sigma}(\str x, \str y) 
 = \left\{ \tau \in K\ \middle|\ \str x(x) = \str y(x^{\tau\sigma})\ \forall x \in \Delta \right\} \sigma 
 = \left\{ \tau \in K\ \middle|\ \str x(x) = \str y^{\sigma^{-1}}(x^{\tau})\ \forall x \in \Delta \right\} \sigma 
 = \Iso^\Delta_K\left( \str x, \str{y}^{\sigma^{-1}} \right) \sigma.
\]

For the chain rule note that $\Iso^{\Delta_1 \cup \Delta_2}_G(\str x, \str y) = \Iso^{\Delta_1}_G(\str x, \str y) \cap \Iso^{\Delta_2}_G(\str x, \str y)$ which implies the first equation. The second is an application of the shift identity.
\end{proof}

Now we are ready to state the two types of Luks reductions.

\begin{procedure}
	{Weak Luks Reduction}
	{Descriptions of groups $H \leq G \leq \Sym(\Omega)$, two strings $\str x, \str y:\Omega \to \Sigma$, a window $\Delta \subseteq \Omega$}
	{$\Iso_G^\Delta(\str x, \str y)$ by reduction to $\gindex{G}{H}$ instances of $\Iso_H^\Delta\left(\str x, \str y^\sigma\right)$ for certain $\sigma \in G$}
	{weak-luks}
	\item compute a set $S$ right coset representatives of $H$ in $G$ calling \cref{proc:schreier-sims}
	\item \textbf{for} $\sigma_i \in S$:
	\item \cindent collect the $\Iso_H^\Delta \left(\str x, \str y^{\sigma_i^{-1}}\right) \sigma_i$ which are of the form $F \tau_i$ for $F = \Aut_H^\Delta(\str x)$ or empty; if so, set $\tau_i \leftarrow 1$
	\item \textbf{end}
	\item \textbf{return} the combined coset $\generate{F \cup \left\{\tau_i\tau_1^{-1} \ \middle|\  1 < i \leq \abs S \right\}} \tau_1$ \label{line:weak-luks-combine}
\end{procedure}
That the procedure works correctly follows directly from \cref{lemma:isoaut-item2} of \cref{lemma:isoaut} given the coset decomposition $G = \bigcup_i H \sigma_i$. Line~\ref{line:weak-luks-combine} is justified by the fact that 
\begin{equation} \label{eq:coset-union}
\bigcup_{1 \leq i \leq \abs S} F \tau_i = \generate{F \cup \left\{\tau_i\tau_1^{-1}\ \middle|\ 1 < i \leq \abs S\right\}} \tau_1.
\end{equation}
The left-hand side is contained in the right-hand side since $F\tau_i = F\tau_i\tau_1^{-1}\tau_1$. Conversely, the generators on the right are contained in the left-hand side when multiplied with $\tau_1$.

The time complexity crucially depends on $\gindex{G}{H}$. Reverting to \cref{sec:schreier-sims}, the computation of the coset representatives takes polynomial time whenever $H$ admits polynomial membership testing and additionally $\gindex{G}{H} \ll \abs{\Omega}^{O(1)}$. If such a bound is exceeded and $\gindex{G}{H}$ grows quasi-polynomially in $\abs \Omega$, then the execution time grows quasi-polynomially as well. Since we recur to $\gindex{G}{H}$ subproblems, the additive costs do not outweigh the multiplicative costs.

The situation in which we apply strong Luks reduction is slightly more subtle. The action of the group $G$ on $\Omega$ admits an invariant subset $\Delta$ which is partitioned into blocks $\{B_1, \dots, B_m\}$ themselves invariant under $G$. Formally, for all $g \in G$ and $1 \leq i \leq m$ there is $1 \leq j \leq m$ such that $B_i^g = B_j$. Thus, the action of $G$ induces an action on the blocks. We therefore have a homomorphism $\psi: G \to \Sym_m$ where the image acts on the blocks. The kernel of this map is precisely the stabilizer of the blocks, i.e.\@ $\sigma \in \ker \psi$ iff $B_i^\sigma = B_i$ for all $i=1, \dots, m$, cf.\@ \cref{def:orbits}.

\begin{proc}[Strong Luks Reduction]
	\label{proc:strong-luks}
	\Input a group $G \leq \Sym(\Omega)$, a $G$-invariant subset $\Delta \subseteq \Omega$, a $G$-invariant partition $\{B_1, \dots, B_m\}$ of $\Delta$, two strings $\str x, \str y: \Omega \to \Sigma$ \newline
	\Output $\Iso_G^\Delta(\str x, \str y)$ by reduction to $m \gindex{G}{\ker \psi}$ instances of $\Iso_{M_i}^{B_i}\left(\str x, \str y^{\sigma_i} \right)$ for certain $\psi: G \to \Sym_m$ as described above, $M_i \leq \ker \psi$, $\sigma_i \in G$
\end{proc}

We present the steps of the procedure in full prose: First we compute the kernel of $\psi$ and call it $H$. This can be done by using Schreier-Sims in polynomial time, cf.\@ \cref{sec:schreier-sims}, since $m \leq \abs \Omega$. Using weak Luks reduction with input groups $H \leq G$ we simplify the problem to $\gindex{G}{\ker \psi}$ instances of subproblems of the form $\Iso_H^\Delta(\str x, \str y^{\sigma_j})$. Every orbit of $H$ acting on $\Delta$ is now contained in one of the blocks $B_i$. We can look for partial isomorphism on each of the blocks separately. We iteratively let $B \in \{B_1, \dots, B_m\}$ and analyze the block following \cite[Observation~3.1.10]{babai}:

That $\Iso_{\restrict{H}{B}}\left(\restrict{\str x}{B}, \restrict{\str y^{\sigma_j}}{B}\right)$ is empty if and only if $\Iso_H^B(\str x, \str y^{\sigma_j})$ is empty, will allow us to pass to shorter strings. While restricting $H$ to $\restrict{H}{B}$, we remember how the elements of $H$ act on the rest of the domain. Maintaining this additional information does not make the subsequent computations more complicated because we do not need to remember the entire action of $H$. For every element of $\restrict{H}{B}$ we just remember one possible extension to $H$ resulting in negligible additional data.

 In this fashion we compute the subproblem $\Iso_{\restrict{H}{B}}\left(\restrict{\str x}{B}, \restrict{\str y^{\sigma_j}}{B}\right)$ and obtain either an empty result or a set of generators $R \subseteq H$ and a $\tau \in H$, such that $R$ generates $\Aut_{\restrict{H}{B}}(\str x)$ and $\restrict{\tau}{B} \in \Iso_{\restrict{H}{B}}\left(\restrict{\str x}{B}, \restrict{\str y^{\sigma_j}}{B}\right)$. Schreier-Sims can then be used to compute generators for $H_{(B)}$ which together with $R$ generate $H' = \Aut_H^B(\str x)$. Thus, $\Iso^B_H(\str x, \str y^{\sigma_j}) = \Aut_H^B(\str x)\tau$. 
By the chain rule, cf.\@ \cref{lemma:isoaut},  $\Iso^{\Delta}_H(\str x, \str y^{\sigma_j}) = \Iso^{\Delta\setminus B}_{H'}\left(\str x, (\str y^{\sigma_j})^{\tau^{-1}}\right) \tau$. We proceed with the next block until there are no more blocks left.

It can be readily verified that we have divided the problem into $m\gindex{G}{\ker \psi}$ subproblems with strings of length $\leq \max_i \abs{B_i}$. The additive costs are polynomial in $\abs \Omega$ when leaving the computation of coset representatives for $\ker \psi$ in $G$ apart. If $\gindex{G}{\ker \psi}$ is bounded polynomially in $\abs \Omega$, then the entire reduction costs only a polynomial amount of time.

A straightforward application of strong Luks reduction is the recurrence on orbits. Suppose that $G$ acts intransitively on $\Omega$. Then the orbits of this action form a partition of $\Omega$ into blocks. The induced action on the blocks is trivial as no permutation moves an element from one block to another. Thus, $\ker \psi = G$. Strong Luks reduction yields as many subproblems as $G$ has orbits on $\Omega$. The short strings are as long as the orbits and in total as long as the original domain. Of course, this reduction was nothing else than the application of the chain rule, cf.\@ \cref{lemma:isoaut}, wrapped in a fancy subprocedure. Importantly, we did not have to compute coset representatives in this case. The costs of this reduction on orbits are therefore negligible.

Most frequently, we will apply strong Luks reduction when $G$ acts transitively but not primitively on $\Omega$. Then we can compute a minimal system of blocks in polynomial time, cf.\@ \cref{sec:orbitsblocks}. Suppose we have found $m$ blocks. They equipartition $\Omega$. Consequently, we can reduce to $m\gindex{G}{N}$ subproblems of length $\abs \Omega/m$ where $N$ denotes the stabilizer of these blocks.

\section{Identification of groups and schemes}
In this section we want to map out isomorphisms between permutation groups which are known to be abstractly permutation isomorphic, cf.\@ \cref{def:permiso}. Let $\Omega$ be a set of size $n$ and let $m \in \mathbb{N}$ be an integer. We are given a permutation group $G \leq \Sym(\Omega)$ satisfying $G \cong \Alt_m$. Two permutation isomorphic groups must admit a bijection between their domains. Thus, if $m > n$, $\Alt_m$ must act as $\Alt^{(k)}_m$ for a given $1 < k \leq m/2$ such that $n = \binom{m}{k}$. 

We will impose an additional condition of the form $k^2 \ll m$ in order to simplify the procedure. However, this does not really come with a loss of generality. If contrarily $m \ll k^2$, we have that, 
\[
n^{\log n} 
\gg \binom{m}{\floor{\sqrt{m}}}^{\log \binom{m}{\floor{\sqrt{m}}}} 
\gg \frac{m!}{\floor{\sqrt m}! (m -\floor{\sqrt m})!} \left( \frac{m}{\sqrt{m}} \right)^{\sqrt{m} (\sqrt m\log m -1)}
\geq m! \sqrt{m}^{m \log m - 2\sqrt{m} - 2m}
\gg m!,
\]
because in general, $\binom{\nu}{\kappa} \geq \left(\frac{\nu}{\kappa}\right)^\kappa$ and $\log \binom{m}{\floor{\sqrt m}} \geq \floor{\sqrt m} \log \frac{m}{\floor{\sqrt m}} \asymp \sqrt m \log m$ for arbitrary $\nu, \kappa, m \in \mathbb{N}$. Hence, in this case $\abs G = m!/2 \ll n^{\log n}$ and we can afford to brutally iterate through the entire group and compute the desired information.
If, for example, $G$ is a quotient $H/N$ as in \cref{sec:first-steps}, the bound $\gindex{H}{N} \ll n^{\log n}$ makes a weak Luks reduction, cf.\@ \cref{proc:weak-luks}, to the smaller group affordable. In \cref{sec:align} when we will need the procedure for identifying two Johnson schemes, the bound allows us to freely scan through all possibilities.

\begin{proc} \label{proc:iota}
\Input set $\Omega$, $n = \abs\Omega$, integers $m \in \mathbb{N}$, $1 \leq k \leq m/2$ such that $n = \binom{m}{k}$ and $m > k(k+1) +1$, a group $G \leq \Sym(\Omega)$ satisfying $G \cong \Alt_m$. \\
\Output a set $\Gamma$ of size $m$, a bijection $\iota: \Omega \to \binom{\Gamma}{k}$ and an isomorphism $\varphi: G\to\Alt(\Gamma)$ such that
\begin{equation} \label{eq:permiso}
	\iota\left(\omega^g\right) = \iota(\omega)^{\varphi(g)} \quad \forall \omega \in \Omega,\ g \in G.
\end{equation}
\Complexity polynomial in $n$.
\end{proc}

Providing an isomorphism means, as always, to give images of the generators of the domain group, here $G$. The explanation follows \cite[§2.8]{helfgott} who quotes \cite{bls}. We will define $\iota$ and $\varphi$ at the very end of the procedure. When referring to $\iota_0$ and $\varphi_0$ until then, we think of general maps satisfying \cref{eq:permiso}. Similarly, $\Gamma_0$ remains an abstract set of $m$ elements until we define $\Gamma$ at the end of the procedure. We infer the correspondence between $G$ and $\Alt_m^{(k)}$ from the orbital structures of the two group actions.

For two sets $S_1, S_2 \in \binom{\Gamma_0}{k}$ and arbitrary $\sigma \in \Alt_m^{(k)}$, clearly $\abs{S_1 \cap S_2} = \abs{S_1^\sigma \cap S_2^\sigma}$. Therefore two pairs $(S_1,S_2),(T_1,T_2) \in \binom{\Gamma_0}{k} \times \binom{\Gamma_0}{k}$ can only be in the same orbital of $\Alt_m^{(k)}$ if $\abs{S_1\cap S_2} = \abs{T_1\cap T_2}$. Conversely, $\Alt_m$ acts $(m-2)$-transitively, cf.\@ \cref{ex:alttrans}. Hence, given two such pairs satisfying $\abs{S_1 \cup S_2} = \abs{T_1 \cup T_2} \leq 2k$ we can find an element in $\Alt_m^{(k)}$ mapping them onto each other whenever $2k \leq m-2$. This is clearly the case under the assumption that $m > k(k+1)+1$. Consequently, we have mapped out the orbital structure of $\Alt_m^{(k)}$: The orbitals are $\{R_0, \dots, R_k\}$ for $R_i = \left\{(S_1,S_2) \in \binom{\Gamma_0}{k} \times \binom{\Gamma_0}{k}\ \middle|\ \abs{S_1 \cap S_2} =i \right\}$. Ordering them by size allows us to identify the corresponding $G$-orbitals. 

We claim that $\abs{R_{i+1}} < \abs{R_i}$ for all $0 \leq i < k$.
Suppose that $m > k(k+1)$. Let $0 \leq i < k$ be arbitrary. Under these assumptions,
$(m-i)(i+1) \geq m-i > m-k > k(k+1) -k = k^2 > (k-i)^2.$
We proceed with the main estimation:
\[
\abs{R_i} = \binom{m}{i}\binom{m-i}{k-i}^2 = \frac{m!(m-i)!}{i!(k-i)!^2(m-k)!^2} = \abs{R_{i+1}} \frac{(m-i)(i+1)}{(k-i)^2} > \abs{R_{i+1}}.
\]

Hence, the largest orbital is $R_0$, the smallest is $R_k$.
We map out the orbital structure of $G$ acting on $\Omega$ in polynomial time, cf.\@ \cref{sec:orbitsblocks}. Let $\Xi \subseteq \Omega \times \Omega$ denote the smallest orbital of $G$ outside the diagonal $\diag \Omega$. Let $\Delta \subseteq \Omega \times \Omega$ be the largest orbital. By the argument from above, $\iota_0(\Xi) = R_{k-1}$ and $\iota_0(\Delta) = R_0$.

\begin{figure}
\centering
\begin{tikzpicture}
\node[inner sep=0pt, anchor=north west] (russell) at (0,0) {\includegraphics[width=6cm]{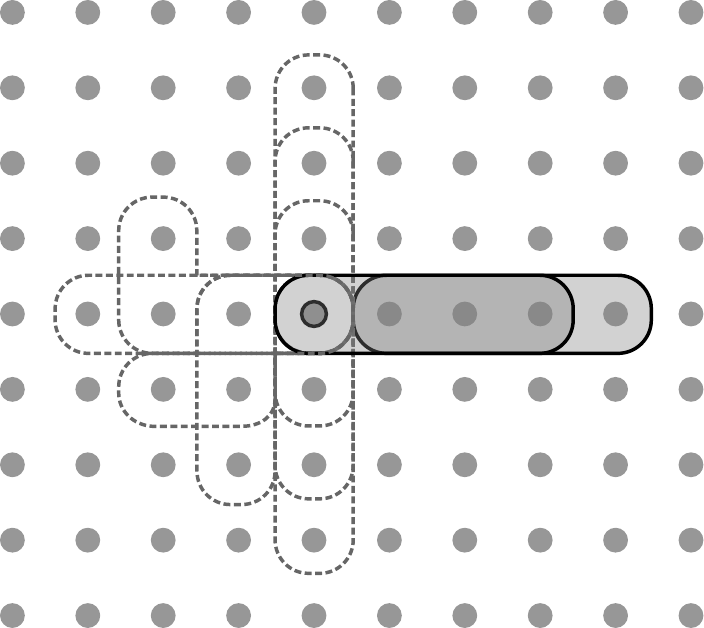}};
\node[inner sep=0pt, anchor=north west] at (.3,-.3) {$\Gamma_0$};
\node[inner sep=0pt, anchor=north west] at (.7,-4.75) {$\iota_0(B(x,y))$};
\node[inner sep=0pt, anchor=north west] at (3.3,-1.9) {$\iota_0(x)$};
\node[inner sep=0pt, anchor=north west] at (4.6,-3.1) {$\iota_0(y)$};
\end{tikzpicture}
\caption[Schematic example of the construction in \cref{proc:iota} with $k=4$]{Schematic example of the construction in \cref{proc:iota} with $k=4$. The distinct element $\delta(x,y)$ of $\Gamma_0$ that is contained in $\iota_0(x)$ but not in $\iota_0(y)$ is encircled. The dashed shapes represent some $\iota_0(z)$ for $z \in B(x,y)$. They are disjoint to $\iota_0(y)$ while having a non-empty intersection with $\iota_0(x)$. \label{fig:iota}}
\end{figure}
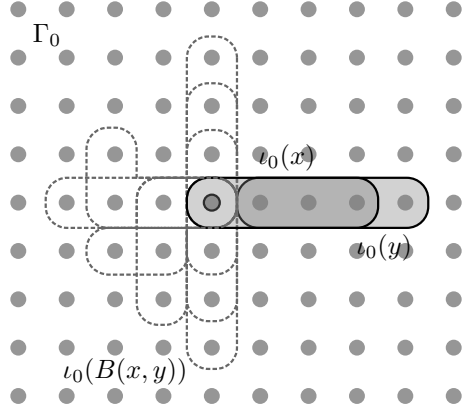

For $(x,y) \in \Xi$ we compute the sets
\begin{align*}
	B(x,y) &= \{z \in \Omega \mid (x,z) \not\in \Delta, (y,z) \in \Delta \}, \\
	C(x,y) &= \Omega \setminus \bigcup_{z \in B(x,y)} \{r \in \Omega \mid (z,r) \in \Delta\}.
\end{align*}
Note that the condition that defines $B(x,y)$ is equivalent to $\iota_0(x) \cap \iota_0(z) \neq \emptyset$ and $\iota_0(y) \cap \iota_0(z) = \emptyset$, since a pair is in $\Delta$ if and only if its $\iota_0$-image is disjoint. Then, we apply the definition:
\begin{align}
	\iota_0(C(x,y)) \nonumber
	&= \iota_0\left( \left\{ a \in \Omega\ \middle|\ (z,a) \not\in \Delta \quad \forall z \in \Omega \text{ such that } (x,z) \not\in \Delta \text{ and } (y,z) \in \Delta \right\} \right) \nonumber \\
	&= \left\{ A \in \binom{\Gamma_0}{k}\ \middle|\ A \cap Z \neq \emptyset \quad \forall Z \in \binom{\Gamma_0}{k}  \text{ such that } \iota_0(x) \cap Z \neq \emptyset \text{ and } \iota_0(y) \cap Z = \emptyset \right\} \nonumber \\
	&= \left\{ A \in \binom{\Gamma_0}{k}\ \middle|\ \delta(x,y) \in A\right\}, \label{eq:iota1}
\end{align}
where $\delta(x,y)$ is the element in the singleton $\iota_0(x) \setminus \iota_0(y)$, that is the unique element in $\iota_0(x)$ that is not simultaneously contained in $\iota_0(y)$.

\Cref{eq:iota1} only holds if $\binom{\Gamma_0}{k}$ is big enough. In particular, it must not happen that a set with non-empty intersection with all $\iota_0(z)$ for $z \in B(x,y)$ does not contain $\delta(x,y)$ but another shared point, cf.\@ \cref{fig:iota}. However, under the assumption $m>k(k+1)$ we can have $z_1,\dots,z_k \in B(x,y)$ such that $\iota_0(z_i) \cap \iota_0(z_j) = \{\delta(x,y)\}$ for all $i \neq j$. They use up $k(k-1)+1$ elements. Since $\iota_0(y)$ must also exist disjointly, we require $k(k-1)+1+k = k^2+1 \leq k(k+1) <m$ elements in $\Gamma_0$, in accordance with the assumption.

Set $\Gamma = \{C(x,y)\mid (x,y) \in \Xi\}$ without multiplicities. We will immediately see that indeed $\abs\Gamma = m$. $B(x,y)$ and $C(x,y)$ can be computed and compared in polynomial time. The action of $G$ on $\Xi$ induces a group action on $\Gamma$. The corresponding isomorphism $\varphi: G \to \Alt(\Gamma)$ is similarly computable in polynomial time. Thus, by definition, $C(x,y)^{\varphi(g)} = C(x^g, y^g)$ for all $g \in G$.

It follows from \cref{eq:iota1} by applying the bijection $\iota_0$ to the left-hand side, that for all $\omega \in \Omega$ and $(x,y) \in \Xi$:
\begin{equation}
	\omega \in C(x,y) \iff \delta(x,y) \in \iota_0(\omega). \label{eq:iota2}
\end{equation}

We have a bijection $j: \Gamma \to \Gamma_0, C(x,y) \mapsto \delta(x,y)$. Clearly, $j$ is onto. Well-definedness and injectivity follow directly from \cref{eq:iota2}.
We finally obtain the bijection $\iota: \Omega \to \binom{\Gamma}{k}$, $\omega \mapsto \{C(x,y) \in \Gamma \mid \omega \in C(x,y)\}$. Since $\varphi$ represents the naturally induced action of $G$ on $\Gamma$, \cref{eq:permiso} is satisfied. We have built a tangible set $\Gamma$ and maps $\iota$ and $\varphi$ which have the desired properties.

\section{Partition pullback}

Throughout the algorithm we maintain an auxiliary set $\Gamma$ which is linked to the permutation group of interests $G \leq \Sym(\Omega)$ and its permutation domain $\Omega$ by an epimorphism $\varphi: G \to \Alt(\Gamma)$ and a surjection $\iota: \Omega \to \binom{\Gamma}{k}$ for some $k \in \mathbb{N}$. $\iota$ has the additional property that the preimage of each singleton in $\Gamma$ is of the same size.

In this section, we study the case in which we are given a partition $\Delta_1, \dots, \Delta_t$ of $\Gamma$. We want to pullback this partition to a partition of $\Omega$ in order to treat the individual subsets separately. See \cite[§5.2]{babai} for background information.

\begin{proc}[Partition pullback] \label{proc:pullback}
\Input a partition $\Delta_1, \dots, \Delta_t$ of $\Gamma$, a surjection $\iota: \Omega \to \binom{\Gamma}{k}$ \\
\Output a canonical partition in color classes $\Omega_v$ for $v \in \mathcal{C}_t \coloneqq \left\{v \in \{0,\dots,k\}^t \ \middle|\ \sum v_i = k \right\}$ of $\Omega$ such that the following conditions hold:
\begin{enumerate}
	\item The color of $\omega \in \Omega$ is the vector $(\abs{\iota(\omega) \cap \Delta_i} \mid 1 \leq i \leq t)$.
	\item For all color classes $\abs {\Omega_v} \leq (2/3) \abs \Omega$ except one with $v = (0, \dots, 0, k, 0, \dots, 0)$.
\end{enumerate}
\Complexity polynomial in $\abs \Omega$
\end{proc}

In the case $t=2$, this simply means that we obtain a canonical partition $\Omega_0, \dots \Omega_k$ where $\omega \in \Omega_j$ if and only if $\iota(\omega)$ contains $j$ elements of $\Delta_1$ and hence $k-j$ elements of $\Delta_2$. For all $0 \leq j \leq k$ with the possible exception of either $j=0$ or $j=k$, the color classes satisfy $\abs{\Omega_j} \leq (2/3) \abs \Omega$.
The procedure depends on a binomial inequality proven by Babai with elementary means, cf.\@ \cite[Proposition~5.2.3]{babai}.
\begin{lemma} \label{lemma:binom-ineq}
Let $m_1, m_2, t_1, t_2$ be integers. Let $m \coloneqq m_1+m_2$ and $t \coloneqq t_1 +t_2$. Suppose that $t \leq m/2$ and $t_1, t_2 \geq 1$. Then
\[
	\binom{m_1}{t_1} \binom{m_2}{t_2} \leq \frac23 \binom{m}{t}.
\]
\end{lemma}

Having this result at hand, we begin with the procedure: We first color the elements of $\binom{\Gamma}{k}$. Let $\Gamma_v \coloneqq \left\{T \in \binom{\Gamma}{k}\ \middle|\ \abs{T \cap \Delta_i} = v_i\ \forall 1 \leq i \leq t \right\}$ for $v \in \mathcal{C}_t$. In other words, we color the $k$-sized subsets of $\Gamma$ according the number of elements they contain from each of the $\Delta_i$. We verify inductively that indeed $\abs{\Gamma_v} \leq (2/3)\abs{\Gamma}$ for each $v$ with one possible exception. If $t = 1$, there is only one such set and the condition is vacuous. We will study the case $t = 2$ at length since it appears in the application in case~2a of \cref{sec:agg-cert}.

Here, the color classes of $\binom{\Gamma}{k}$ can be indexed more intuitively as $\{\Gamma_0, \dots, \Gamma_k\}$ where $\gamma \in \Gamma_j$ if and only if $\gamma$ contains $k-j$ elements of $\Delta_1$ and thus $j$ elements of $\Delta_2 = \Gamma \setminus \Delta_1$. Without loss of generality, we can assume that $k \leq \abs \Gamma/2$. Hence, by \cref{lemma:binom-ineq}, for $0 < j < k$,
\[
\abs{\Gamma_j} 
= \abs{\left\{ \alpha \mathbin{\dot\cup} \beta\ \middle|\ \alpha \in \binom{\Delta_1}{k-j}, \beta \in \binom{\Delta_2}{j} \right\}}
= \binom{\abs{\Delta_1}}{k-j} \binom{\abs{\Delta_2}}{j}
\leq \frac23 \binom{\abs \Gamma}{k}.
\]
Suppose without loss of generality that $\abs{\Delta_1} \leq \abs \Gamma/2$. Then $\abs{\Gamma_0} = \binom{\abs{\Delta_1}}{k} \leq \binom{\abs\Gamma /2}{k} \leq \frac12 \binom{\abs \Gamma}{k}$. The only possible exception to the bound is thus $\Gamma_k$ corresponding to $v = (0,k)$.

We revert to the original notation and suppose that the claim holds for $t-1 \in \mathbb{N}$. We apply the hypothesis to the partition $\Delta_1, \dots, \Delta_{t-2}, \Delta'$ for $\Delta' \coloneqq \Delta_{t-1} \cup \Delta_t$ obtaining a partition $\Gamma'_v$ for $v \in \mathcal{C}_{t-1}$ satisfying the desired bound with one possible exception. The colors are aware of the number of elements from $\Gamma'$ but cannot tell how many elements are from $\Delta_{t-1}$ and $\Delta_t$. If the exception is not $\Gamma'_{(0, \dots, 0, k)} = \binom{\Delta'}{k}$, we can simply encode the numbers $\abs{T \cap \Delta_{t-1}}$ and $\abs{T \cap \Delta_{t}}$ for $T \in \binom{\Gamma}{k}$ in the colors of $\Gamma'$ not enlarging the color classes. Otherwise we apply the induction hypothesis to the partition $\Delta_{t-1}, \Delta$ obtaining a coloring $\Gamma''_v$ of $\Delta'$ for $v \in \mathcal{C}_2$.  These color classes satisfy $\abs{\Gamma''_v} \leq (2/3) \abs{\Delta'} \leq (2/3) \abs\Gamma$ with one possible exception. We refine the coloring $\Gamma'$ as above using $\Gamma''$ for the sets from $\binom{\Delta'}{k}$. All color classes except one have admissible size and the only exception is the class corresponding to either $\binom{\Delta_{t-1}}{k}$ or $\binom{\Delta_t}{k}$. This finishes the induction.

It remains to infer the coloring of $\Gamma$. For $v \in \mathcal{C}_t$ we set $\Omega_v = \iota^{-1}(\Gamma_v)$. By the properties of $\iota$, $\abs{\Omega_v} = \abs{\Omega} \binom{\abs \Gamma}{k}^{-1} \abs{\Gamma_v} \leq (2/3) \abs \Omega$ for all $v$ except one.
For computing the color classes we need to evaluate $\iota$ for each $\omega \in \Omega$ and compute the cardinalities of the intersection with the $\Delta_i$. This takes polynomial time.

\section{Design Lemma}
The two main mostly combinatorial tools of Babai's algorithm are the Design Lemma \cite[§6]{babai} and the Split-or-Johnson routine \cite[§7]{babai}. Their powers are combined in the following Extended Design Lemma. It is used to descend from a $k$-ary relational structure with large symmetry defect either to a canonical $\alpha$-partition or to a canonically embedded Johnson scheme, cf.\@ \cref{def:relstruct,def:symrelstruct,def:apart,def:johnson-scheme}. That a Johnson scheme $\mathfrak{J}(m,t)$ is non-trivial means in particular that $t \geq 2$.

\begin{proc}[Extended Design Lemma, {\cite[Theorem~7.3.3]{babai}}] \label{proc:design}
\Input a threshold parameter $3/4 \leq \alpha < 1$, a $k$-ary relational structure $\mathfrak{X} = (\Omega, \mathcal{R})$ with relative symmetry defect $> 1-\alpha$, such that $2 \leq k \leq n/4$ for $n = \abs \Omega$.\\
\Output Either
\begin{enumerate}
\item a canonical colored $\alpha$-partition of $\Omega$, or
\item a canonically embedded non-trivial Johnson scheme on a subset $W \subseteq \Omega$ of size $\abs W \geq \alpha n$.
\end{enumerate}
\Complexity multiplicative and additive costs of $n^{O(k + \log n)}$
\end{proc}

The Extended Design Lemma with its various parts is one of the core elements of Babai's algorithm. That its discussion here will be rather brief should not disguise that its justification in \cite{babai} takes more than twenty pages of elaborate arguments. The included case of uniprimitive coherent configurations (UPCC) was the part where Helfgott in 2017 found a mistake invalidating the overall quasi-polynomial bound. The issue has long been resolved by Babai, cf.\@ \cite{helfgottBlog}.

We will give a short overview of the Extended Design Lemma's internal mechanisms and justify the complexity claim following \cite[§5,~p.~41]{helfgott}. The first step is to transform the $k$-ary relational structure $\mathfrak{X}$ into a $k$-ary coherent configuration $\mathfrak{Y} = (\Omega, c : \Omega^k \to \mathcal{C})$, cf.\@ \cref{def:coherent-config}. The algorithm of Weisfeiler-Leman is the method of choice, cf.\@ \cite[§2.5]{helfgott}. In time $k^2 n^{2k+1} \log k \ll n^{O(k)}$ it iteratively refines the color classes of $\mathfrak{X}$ until the resulting structure satisfies the desired regularity conditions. Importantly, the assignment of the coherent configuration to the input structure is canonical. Hence, the algorithm does not incur any multiplicative costs. Note that we simplified the explanation hiding that Weisfeiler-Leman actually takes configurations as an input, cf.\@ \cref{def:config}. In fact, the given relational structure is first refined to a partition structure and then to a configuration. This process is formalized in \cite[§2.3]{helfgott}. The additional costs can be hidden in the implicit constants of the estimates for Weisfeier-Leman. 

After having obtained the coherent configuration $\mathfrak{Y}$, we start a brute force search for tuples of distinct elements $(x_1, \dots, x_l) \in \Omega^l$ for $l \leq k-1$. For each of these tuples we check in $n^{O(1)}$ the following two conditions:
\begin{enumerate}
\item There exists no color $i \in \mathcal{C}$ such that there are $\geq \alpha \abs \Omega$ values of $\gamma \in \Omega$ with $c'(\gamma) = i$ in the partition structure $\mathfrak{Z}'$ defined as \label{item:design1}
\[
	\mathfrak{Z}' \coloneqq \left( \Omega,\ c': \Omega \to \mathcal{C}, \omega \mapsto c(x_1, \dots, x_l, \omega, \dots, \omega)\right).
\]
\item It holds that $l \leq k-2$ and there exists a color $i \in \mathcal{C}$ from $\mathfrak{Z}'$ such that $c'(\gamma) = i$ for $\geq \alpha \abs \Omega$ values of $\gamma \in \Omega$. Let furthermore $C \subseteq \Omega$ denote the color class of $i$. The structure \label{item:design2}
\[
\mathfrak{Z}'' \coloneqq \left(C,\ c'': C \times C \to \mathcal{C}, (\omega_1, \omega_2) \mapsto c(x_1,\dots, x_l, \omega_1, \omega_2, \dots, \omega_2) \right)
\] is not a clique, i.e.\@ it admits at least one non-trivial color class.
\end{enumerate}
That we can find a tuple satisfying at least one of the above conditions is guaranteed by the Design Lemma, cf.\@ \cite[Proposition~5.1]{helfgott}, under the assumption that $\mathfrak{X}$ and hence $\mathfrak{Y}$ have large symmetry defect. Thus, finding a suitable tuple takes $n^{O(k)}$. We would like to utilize the structure $\mathfrak{Z}'$ or respectively $\mathfrak{Z}''$ which is associated to the tuple. Unfortunately, the choice of the tuple is not canonical. We have to individualize the tuple in order to treat the structure afterwards as a canonical feature, cf.\@ \cref{sec:agg-cert} for a detailed explanation of this strategy. The total number of tuples equals the incurred multiplicative costs of $n^{O(k)}$.

We will now work with $\mathfrak{Z}'$ and $\mathfrak{Z}''$. In \cref{item:design1} we have found nothing else than a coloring of $\Omega$ with no color class larger than $\alpha \abs \Omega$. Taking the color classes as blocks we obtain a colored $\alpha$-partition as desired and terminate. If contrarily \cref{item:design2} holds, we use $\mathfrak{Z}''$ as an input for the Split-Or-Johnson routine, cf.\@ \cite[Theorem~3.5]{helfgott}. We observe that $\mathfrak{Z}''$ inherits the property of being a coherent configuration from $\mathfrak{Y}$, cf.\@ \cite[Exercises~2.11,~2.13]{helfgott}. Overtly, $\mathfrak{Z}''$ is classical and by construction homogeneous. In the case that $\mathfrak{Z''}$ is not primitive we obtain by \cite[Exercise~2.16]{helfgott} a colored $1/2$-partition as desired, cf.\@ \cref{def:coherent-config-prop}. We have now ensured that $\mathfrak{Z}''$ is a uniprimitive classical coherent configuration. Hence, we are ready to invoke Split-Or-Johnson. In time $n^{O(1)}$ we find either a colored $\alpha$-partition or a Johnson scheme on a subset $W \subseteq \Omega$ of size $\abs W \geq \alpha \abs \Omega$. Again, these findings are not canonical. The conducted individualization accounts for multiplicative costs of $n^{O(\log n)}$.

\chapter{Overview of the Algorithm}
\label{sec:overview}
\begin{figure}
\centering
\resizebox{\textwidth}{!}{
\begin{tikzpicture}[node distance = 1cm, auto]
\tikzstyle{decision} = [diamond, draw, fill=gray!20, 
    text width=4.5em, text badly centered, node distance=3cm, inner sep=0pt]
\tikzstyle{block} = [rectangle, draw, fill=gray!20, 
    text width=5em, text centered, rounded corners, node distance=3cm, minimum height=4em]
\tikzstyle{line} = [draw, -latex']
\tikzstyle{cloud} = [draw, ellipse,fill=gray!20, node distance=3cm, text width=5em, text badly centered,
    minimum height=2em]
    
\node [block, text width=6cm] (init) {\textbf{Main Procedure} \\ {Input:} $G \leq \Sym(\Omega)$, $\str x, \str y: \Omega \to \Sigma$ \\ {Output:} $\Iso_G(\str x, \str y) = \Aut_G(\str x) \sigma$ or $\emptyset$};
\node [decision, below of=init, yshift=.5cm] (d1) {$G$ small};
\node [block, right of=d1, xshift=3cm] (r7) {Brute Force};
\node [decision, below of=d1] (d2) {$G$ transitive};
\node [block, right of=d2, xshift=3cm] (r8) {Reduction};
\node [decision, below of=d2] (d3) {$G/N$ small};
\node [block, right of=d3, xshift=3cm] (r9) {Reduction};
\node [decision, below of=d3] (d4) {$G$ primitive};
\node [decision, right of=d4] (d6) {$k=1$};
\node [decision, below of=d6] (d5) {\small Large Symmetry};
\node [block, right of=d5] (r10) {Reduction};
\node [block, left of=d4] (cert) {Local Certificates};
\node [block, left of=cert] (r1) {Reduction};
\node [cloud, below of=d4] (r2) {Canonical structure};
\node [block, below of=r2] (design) {Extended Design Lemma};
\node [cloud, below of=design,xshift=3cm] (r4) {Johnson scheme};
\node [cloud, below of=design,xshift=-3cm] (r3) {Colored partition};
\node [block, left of=r3] (r5) {Reduction};
\node [block, right of=r4] (r6) {Reduction};
\node [block, right of=d6] (r11) {Trivial Case};

% \node [left of=d1, text width=7cm, xshift=-2cm, draw, anchor=north west] {After having finished the first steps, we maintain appart from the input data, a set $\Gamma$ an epimorphism $$\varphi: G \to \Alt(\Gamma) $$ and a surjection $$\iota: \Omega \to \binom{\Gamma}{k}$$};

\path [line] (init) -- (d1);
\path [line] (d1) -- node {no} (d2);
\path [line] (d2) -- node {yes} (d3);
\path [line] (d3) -- node {no} (d4);
\path [line] (d4) -- node {yes} (d6);
\path [line] (d4) -- node {no} (cert);
\path [line] (d5) -- node {no} (r2);
\path [line] (d6) -- node {yes} (r11);
\path [line] (d6) -- node {no} (d5);
\path [line] (design) -- (r3);
\path [line] (design) -- (r4);
\path [line] (cert) -- node {c.~2a} (r1);
\path [line] (cert) -- node [yshift=-.2cm] {cases~2b,~3} (r2);
\path [line] (cert) -- node {case~1} (r3);
\path [line] (r2) -- (design);
\path [line] (r4) -- (r6);
\path [line] (r3) -- (r5);
\path [line] (d1) -- node {yes} (r7);
\path [line] (d2) -- node {no} (r8);
\path [line] (d3) -- node {yes} (r9);
\path [line] (d5) -- node {yes} (r10);

% \path [line] (r7) -| +(2,0) |- (init);
\path [line] (r8) -| +(2,0) |- (init);
\path [line] (r9) -| +(2,0) |- (init);
\path [line] (r10) -| +(2,0) |- (init);
\path [line] (r6) -| +(2,0) |- (init);
\path [line] (r5) -| +(-2,0) |- (init);
\path [line] (r1) -| +(-2,0) |- (init);
\end{tikzpicture}
}
\caption[Overview of the main procedure]{Overview of the main procedure. Some trivial branches have been omitted. \label{fig:overview}}
\end{figure}
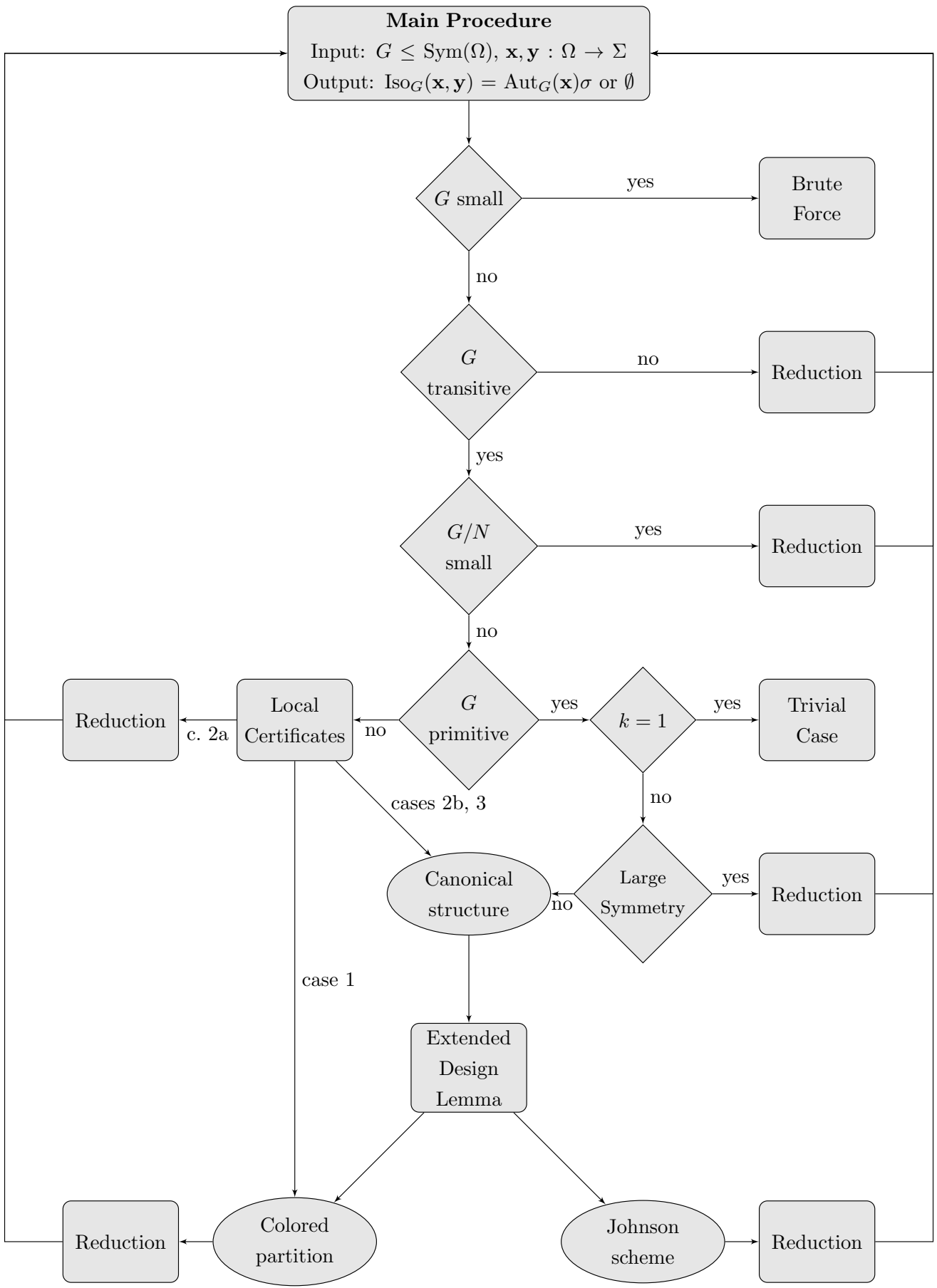

As this thesis focuses on one particular case of Babai's algorithm, it is worth obtaining an impression of the overall algorithm that decides the String Isomorphism Problem. We will refer to it as the \emph{main procedure}. As outlined earlier, it requires three arguments: the strings $\str x, \str y: \Omega \to \Sigma$ and a description of a permutation group $G \leq \Sym(\Omega)$, i.e.\@ a set of its generators. Principally, our result will be either that $\Iso_G(\str x, \str y)$ is empty or that 
\[ \Iso_G(\str x, \str y) = \Aut_G(\str x) \sigma\]
for $\sigma \in \Iso_G(\str x, \str y)$. In the latter case we aim at computing generators for $\Aut_G(\str x)$ and a suitable $\sigma$. The following steps are undertaken to decide the problem. Details beyond the case of imprimitivity will be omitted. For these aspects, the reader is pointed to \cite[§3.1]{helfgott} and \cite[§3.3]{babai}.

\section{First steps}
\label{sec:first-steps}
We can first exclude some trivial cases. If $G \leq \Aut(\str x)$, then $\Iso_G(\str x, \str y) = G$ if $\str x = \str y$ or empty if not. We recall that $G$ is without loss of generality described by polynomially many generators, cf.\@ \cref{sec:schreier-sims}. Thus the condition is testable in polynomial time. If furthermore $\abs{G} < C$ for some absolute constant $C$, then we compute $\Iso_G(\str x, \str y)$ with brute force in constant time. Many theorems require some absolute lower bounds on $\abs G$. These restrictions should not bother us because we are always interested in asymptotic behaviors only. Sufficiently small subproblems will be solved brutally.

If $G$ acts intransitively on $\Omega$, strong Luks reduction can be used to recur on the orbits, that is on shorter strings, cf.\@ \cref{proc:strong-luks}.
Now suppose that $G$ is transitive. We aim at passing to a primitive action which allows us to use a group theoretic result by Cameron-Maróti. Using polynomial time algorithms described in \cref{sec:orbitsblocks} we compute a minimal system of blocks $\mathcal B$. Let $n = \abs{\mathcal B}$. The induced action on the blocks presents itself as an epimorphism $G \to G'$ where $G' \leq \Sym(\mathcal B)$. Let $N$ denote the kernel of this map. $N$ stabilizes the blocks, i.e.\@ $B^g = B$ for each $B \in \mathcal B$ and all $g \in N$. The group $G/N \cong G'$ acts primitively on $\mathcal B$. Under these circumstances, Cameron-Maróti implies one of the three following cases:

\begin{theorem}[Cameron-Maróti, {\cite[Theorem~1.1]{maroti}, \cite[Theorem~3.2.1]{babai}}] \label{thm:cameron}
	Let $G' \leq \Sym_n$ be a primitive permutation group. Then one of the following holds:
	\begin{enumerate}
		\item $G'$ is a \emph{Cameron group}. That is, there exist $m, k, r \in \mathbb{N}$ such that $n = \binom{m}{k}^r$. The group $G'$ is a subgroup of $\Sym_m^{(k)} \wr \Sym_r$ with the primitive product action\footnote{See \cite[§§2.6, 2.7]{dixon} for definitions of the wreath product and its primitive action.} on $\binom{\Gamma}{k}^r$ for $\Gamma = \{1,\dots, m\}$. $G'$ admits the normal subgroup $\left(\Alt_m^{(k)}\right)^r$. Furthermore, $\gindex{G'}{\left(\Alt_m^{(k)}\right)^r} \leq n$. \label{item:cameron1}
		\item $G'$ is a Mathieu group with $n < 25$. \label{item:cameron2}
		\item $\abs{G'} < n^{1+\log_2 n}$. \label{item:cameron3}
	\end{enumerate}
\end{theorem}

As explained above, we are not interested in \cref{item:cameron2}. For distinguishing \cref{item:cameron1,item:cameron3} we compute $\gindex{G}{N} = \abs{G'}$ in polynomial time using Schreier-Sims. In case of \cref{item:cameron3} we perform strong Luks reduction to reduce from $G$ to $N$, cf.\@ \cref{proc:strong-luks}. We obtain $n \cdot n^{1+\log_2 n}$ subproblems of length $\leq \abs{\Omega}/n$.

In the last case, we find a subgroup with admissibly small index and well understood action: We know that $G/N$ contains a normal subgroup $M$ which acts on $\mathcal B$ as $\left( \Alt_m^{(k)} \right)^r$. $M$ must be understood as contained in the base $\left(\Sym_m^{(k)} \right)^r$ of the wreath product. Hence, it acts on $\binom{\Gamma}{k}^r$ as $a^\sigma(i) = a(i)^{\sigma(i)}$ for $a \in \binom{\Gamma}{k}^r, 1 \leq i \leq r, \sigma \in M$, where we understand $M$ as the group of all maps $\{1,\dots,r\} \to \Alt_m^{(k)}$ with pointwise multiplication. Clearly, this action admits a coarser system $\mathcal B'$ of $\binom{m}{k}$ blocks of size $\binom{m}{k}^{r-1}$, i.e.\@ those that arise when fixing the first entry in the tuples of $\binom{\Gamma}{k}^r$. On these blocks, $M$ acts as Johnson group $\Alt_m^{(k)}$ inducing a homomorphism $M \to \Alt(\Gamma)$.

We have to transform this abstract structural knowledge into tangible information. \cite{bls} provides us with the tools to map out the structure of the Cameron group $G/N$ and the blocks of its action. This takes polynomial time. Let $\pi: G \to G/N$ denote the projection of $G$ onto $G/N$ and set $M' = \pi^{-1}(M)$. Then $M'$ acts on $\Omega$ admitting the blocks $\mathcal B'$ on which it acts as Johnson group. Let $N' \leq M'$ denote the stabilizer of $\mathcal{B}'$.

We apply \cref{proc:iota} with input $M'/N' \leq \Sym(\mathcal{B}')$ and $m$, $k$, obtaining a bijection $\iota': \mathcal{B}' \to \binom{\Gamma'}{k}$ and an isomorphism $\varphi: M'/N' \to \Alt(\Gamma')$ where $\Gamma'$ is a set of $m$ elements constructed by that procedure. $\iota'$ can be naturally extended to a surjection $\iota: \Omega \to \binom{\Gamma'}{k}$ with the property that each $\gamma \in \binom{\Gamma'}{k}$ has precisely $\abs \Omega \binom{m}{k}^{-1}$ preimages.

Using weak Luks reduction we pass from $G$ to $M'$. The multiplicative costs are linear since $\gindex{G}{M'} = \gindex{G}{\pi^{-1}(M)} = \frac{\abs G}{\abs M \abs{\ker \pi}} = \gindex{G/N}{M} \leq n$. Call $M'$ from now on $G$ and let $\Gamma \leftarrow \Gamma'$.
We then hold an epimorphism $\varphi: G \to \Alt(\Gamma)$. The stabilizer of the blocks in $\mathcal{B}'$ is precisely the kernel of the map $\varphi$. Thus, the epimorphism $\varphi$ is injective if and only if the action of $G$ on $\Omega$ is primitive. The case of $\varphi$ having a non-trivial kernel is therefore called the \emph{imprimitive case}. It is the main subject of this thesis.

We rename $N \leftarrow \ker \varphi = N'$, obtaining $\varphi: G/N \to \Alt(\Gamma)$, an isomorphism.
Moreover, we can exclude the case $\abs \Gamma \ll \log \abs \Omega$. If this holds, we have $\gindex{G}{N} = \frac12\abs\Gamma! < \abs\Gamma^{\abs\Gamma} \ll \abs\Omega^{O(\log \abs\Omega)}$ and can afford a strong Luks reduction from $G$ to $N$. Since $N$ is a stabilizer of $\binom{m}{k}$ blocks, we obtain $n \abs{\Omega}^{O(\log \abs\Omega)} \ll \abs{\Omega}^{O(\log \abs\Omega)}$ subproblems for strings of length $\leq \abs\Omega / \binom{m}{k} \leq \abs\Omega/2$.

Let us summarize the setting which will occupy us for the rest of this thesis: Being given two strings $\str x, \str y: \Omega \to \Sigma$ as inputs, we have reduced the problem to a group $G \leq \Sym(\Omega)$ for which we hold a description of an epimorphism $\varphi: G \to \Alt(\Gamma)$. During the reduction the original input strings have been manipulated, in particular shifted by permutations. We will refer to the updated strings as $\str x$ and $\str y$. The set $\Gamma$ is explicitly known, we even possess a surjection $\iota: \Omega \to \binom{\Gamma}{k}$ respecting $\varphi$. We can suppose that $\abs \Gamma \gg \log \abs \Omega$ and that $G$ acts transitively on $\Omega$.

\section{The primitive case and the case of large symmetry}
\label{sec:primitive}
The subject of this thesis is the imprimitive case which is treated using local certificates. In this case, $G$ acts imprimitively on $\Omega$, i.e.\@ admitting non-trivial blocks. Before looking at this case, we want to give an overview of the other parts of the algorithm, in particular of the primitive case. $G$ acts primitively on $\Omega$, $\varphi: G \to \Alt(\Gamma)$ is an isomorphism and $\iota: \Omega \to \binom{\Gamma}{k}$ is a bijection. 

If $k = 1$, we are in the very comfortable situation in which $G = \Alt(\Omega) = \Alt(\Gamma)$. We can compute $\Iso_G(\str x, \str y)$ rapidly. The two strings are isomorphic if and only if their letters occur with the same multiplicities. This can be verified in polynomial time. In order to describe $\Aut_G(\str x)$, we color the letters in $\Sigma$ according to their multiplicities in $\str x$. $\Aut_G(\str x)$ consists then of all permutations that swap letters of the same color. We can now describe $\Aut_G(\str x)$ in polynomial time. The strategy is described in the following example. Finally, we require a permutation $\sigma \in \Alt(\Omega)$ mapping $\str x$ to $\str y$. We obtain $\sigma$ by computing the same coloring on $\str y$ and imposing the additional condition that the permutation must belong to $\Alt(\Omega)$, as well in polynomial time. Our result is then
\[
	\Iso_G(\str x, \str y) = \Aut_G(\str x) \sigma.
\]
\begin{example}
We want to construct $\Aut_G(\str x)$ explicitly for $\str x = \mathsf{hippopotomonstrosesquippedaliophobia}$, i.e.\@ $\Omega = \{1,\dots, 36\}$, $\Sigma=\{\mathsf a,\dots,\mathsf z\}$. The coloring is induced by the multiplicities:
\begin{center}
\begin{tabular}{l|cccccc}
color/multiplicity & 7 & 6 & 4 & 3 & 2 & 1 \\ \hline
letters & \textsf{o} & \textsf{p} & \textsf{i} & \textsf{s} & \textsf{a, e, h, t} & \textsf{b, d, l, m, n, q, r, u}
\end{tabular}
\end{center}
For the colors 7, 6, 4, 3 and 1 we obtain rather simple contributions to $\Aut_G(\str x)$ because here no blocks need to be considered. For example for color 7, we add generators for $\Sym(\{5, 7, 9, 11, 16, 30, 33\})$. Color class 2 is more complicated. On the upper level, the four letters can be permuted while on the lower level the positions carrying the same letter can be swapped. The contribution of this color class is therefore isomorphic to $\Sym_2 \wr \Sym_4$. We add permutations for the lower level, e.g.\@ $(27\ 36)$ for $\mathsf{a}$, and permutations for the upper level, e.g.\@ $(18\ 27)(25\ 36)$ corresponding to a swap of $\mathsf{a}$ and $\mathsf{e}$. After we have computed generators for all these building blocks, we may have added odd permutations. However, we can apply Schreier-Sims, cf.\@ \cref{sec:schreier-sims}, to make sure that the resulting group is indeed a subgroup of $\Alt(\Omega)$.
\end{example}

Before considering the more complex case, we note that $k \leq \log_2 \abs\Omega$. This follows from $\abs\Omega = \binom{\abs\Gamma}{k} \geq (\abs\Gamma/k)^k \geq 2^k$, since by construction $1\leq k \leq \abs\Gamma/2$.

Let us now look at the case $k > 1$ and $G$ primitive. We want to rule out the case of large symmetry since it cannot be treated by the Design Lemma, cf.\@ \cref{proc:design}. The Design Lemma processes $k$-ary relational structures with little symmetry (or large symmetry defect), cf.\@ \cref{def:relstruct,def:symrelstruct}. We want to build such structures, one for each of the input strings. Let
\[
	\mathfrak{X}(\str x) \coloneqq \left(\Gamma, (R_\alpha)_{\alpha \in \Sigma}\right), \quad 
	R_\alpha \coloneqq \left\{ (\gamma_1, \dots, \gamma_k) \in \Gamma^k\ \middle|\ \str x\left(\iota^{-1}\left( \{\gamma_1, \dots, \gamma_k\}\right)\right) = \alpha\right\} \text{ for } \alpha \in \Sigma
\]
and $\mathfrak{X}(\str y)$ respectively. Both are $k$-ary partition structures. An edge in $\Gamma^k$ carries a color corresponding to the letter at the position in the string which is associated to the edge in virtue of $\iota$. This definition is possible since $\iota: \Omega \to \binom{\Gamma}{k}$ is a bijection. The structures canonically depend on the two strings. Hence, whenever we witness that $\mathfrak{X}(\str x)$ and $\mathfrak{X}(\str y)$ have differing canonical properties, we can refute isomorphicity. We can map out the two structures in $\abs\Gamma^k \ll \abs\Omega^{O(\log \abs\Omega)}$ many steps.

In order to distinguish cases, we have to compute the twin classes of the relational structures. Two elements $\gamma_1, \gamma_2 \in \Gamma$ are twins in $\mathfrak{X}(\str x)$, if the transposition $(\gamma_1\ \gamma_2) \in \Aut(\mathfrak{X}(\str x))$. This equivalence relation can be written alternatively\footnote{We cannot pull out $(\gamma_1\ \gamma_2)$ of $\iota$ using \cref{eq:permiso} because this transposition is clearly not contained in $G^\varphi = \Alt(\Gamma)$.} as
\[
	\gamma_1 \sim_{\str x} \gamma_2 \iff \forall j \in \binom{\Gamma}{k}:\ \str x\left(\iota^{-1}(j)\right) = \str x \left(  \iota^{-1}\left(j^{(\gamma_1\ \gamma_2)}\right) \right),
\]
Here, the transposition acts naturally on sets, i.e.\@ $(\gamma_1\ \gamma_2) \in \Sym^{(k)}(\Gamma)$. The condition can be verified for all pairs $\gamma_1,\gamma_2$ in polynomial time. Since there are $\binom{\abs\Gamma}{2} \ll \abs\Omega^2$ such pairs, we obtain the entire twin class structure in polynomial time. Having computed the twin classes, we can easily determine the symmetry defect. Let $T \subseteq \Gamma$ such that $\Gamma \setminus T$ is a largest twin class by cardinality. The (relative) symmetry defect is defined as $\abs{T}/\abs{\Gamma}$. Again, the twin classes of $\mathfrak{X}(\str x)$ and $\mathfrak{X}(\str y)$ may not be identical as sets. However, if they differ in size we can refute isomorphicity.

Suppose that we are in the case of large symmetry, i.e.\@ both structures have symmetry defect $\leq 1/2$. Let $C_{\str x}$ and $C_{\str y}$ denote the large twin classes satisfying $\abs{C_{\str x}} = \abs{C_{\str y}} > \abs\Gamma/2$. Due to their size they are unique and thus canonical. We can without loss of generality assume that $\abs{C_{\str x}} = \abs{C_{\str y}}$. 

The next step is to \emph{align} the two dominating twin classes. This strategy will be used often throughout the entire algorithm. 
By enumerating both classes, we can construct a $\sigma \in \Alt(\Gamma)$. We lift this permutation along $\varphi$ to a $\tau = \varphi^{-1}(\sigma) \in G$ using Schreier-Sims. Replacing $\str y$ by $\str y^\tau$ is what Babai calls alignment, cf.\@ \cite[§14.1]{babai}. Shifting $\str y $ to $\str y^\tau$ does not curtail our capability of determining $\Iso_G(\str x, \str y)$, given that by \cref{lemma:isoaut},
\[\Iso_G(\str x, \str y) = \Iso_{G\tau^{-1}}(\str x, \str y) = \Iso_G\left( \str x, \str y^\tau \right) \tau^{-1}.\]

After having completed the alignment, we can assume that $C \coloneqq C_{\str x} = C_{\str y}$. Applying \cref{proc:pullback}, we let the partition $\Gamma = (\Gamma \setminus C) \mathbin{\dot\cup} C$ induce a partition $\Omega_0 \mathbin{\dot\cup} \cdots \mathbin{\dot\cup} \Omega_k$ of $\Omega$. The resulting partition satisfies $\abs{\Omega_i} \leq \abs \Omega/2$ for $0 < i \leq k$ since $\abs{\Omega_0} = \abs{\iota^{-1} \binom{C}{k}} > \abs\Omega/2$. Every isomorphism $\sigma$ from $\str x$ to $\str y$ must preserve the former partition, i.e.\@ $C^{\varphi(\sigma)} = C$. Hence, we have reduced the problem such that $\Iso_G(\str x, \str y) = \Iso_H(\str x, \str y)$ for $H = \varphi^{-1}(\Alt(\Gamma)_C)$.

We can easily compute $H$ despite that it is the preimage of a setwise stabilizer. The computation merely involves taking the preimage of five generators applying Schreier-Sims: By \cref{eq:generators-alt}, $\Alt(C)$ and $\Alt(\Gamma \setminus C)$ are generated each by two elements, thus $\Alt(C) \times \Alt(\Gamma \setminus C)$ requires four generators. The fifth is of the form $\tau \coloneqq (a_1\ a_2)(b_1\ b_2)$ where $a_1,a_2 \in C$ and $b_1,b_2 \in \Gamma \setminus C$ and adds elements that are products of two odd permutations taken from the two alternating groups. We can pass to an even smaller group, namely $H' = \varphi^{-1}(\Alt(C) \times \Alt(\Gamma \setminus C))$. Here, we must, of course, read $\Alt(C) \times \Alt(\Gamma \setminus C)$ as a subgroup of $\Alt(\Gamma)$. The quotient $H/H'$ contains apart from the identity only $H'\tau$. Applying weak Luks reduction, cf.\@ \cref{proc:weak-luks}, the task reduces to the two problems of determining $\Iso_{H'}(\str x, \str y)$ and $\Iso_{H'}\left(\str x, \str y^{\tau^{-1}}\right)$. Even in the latter case $C$ remains a twin class because $\tau$ leaves this set invariant. We treat both instances of the subproblem similarly. Let $\str y$ denote the second string in any case.

Since $C$ is a twin class for $\str x$ and $\str y$, every permutation in $\varphi^{-1}(\Alt(C))$ leaves $\str x$ and $\str y$ unchanged. By construction, $C$ corresponds to $\Omega_0$. Hence, $\Iso_{H'}(\str x, \str y)$ is empty if $\restrict{\str x}{\Omega_0} \neq \restrict{\str y}{\Omega_0}$. Let us thus suppose the contrary. In this case, $\Iso_{H'}(\str x, \str y) = \Iso_{\restrict{H'}{\Omega'}}\left(\restrict{\str x}{\Omega'}, \restrict{\str y}{\Omega'}\right)$ for $\Omega' = \Omega \setminus \Omega_0$ holds. We have reduced the problem significantly. $\restrict{H'}{\Omega'}$ acts on $\Omega'$. \Cref{proc:pullback} ensures that $\Omega_j$ is the set of elements $\omega \in \Omega$ with the property that $\abs{\iota(\omega) \cap C} = k-j$. This value is an invariant under the action of $\restrict{H'}{\Omega'}$ because the $\varphi$-image of this group fixes $C$. Therefore, the orbits of $\restrict{H'}{\Omega'}$ acting on $\Omega'$ must refine the partition $\Omega' = \Omega_1 \mathbin{\dot\cup} \cdots \mathbin{\dot\cup} \Omega_k$. Each orbit is hence of length $\leq \abs\Omega/2$. We recur on the orbits by applying strong Luks reduction, cf.\@ \cref{proc:strong-luks}. This yields $k \ll \log\abs\Omega$ subproblems for strings of length $\leq \abs\Omega/2$ and total length $\leq \abs\Omega$.

We have treated the case of large symmetry. Let us now assume that $\mathfrak{X}(\str x)$ and $\mathfrak{X}(\str y)$ have symmetry defect $> 1/2$. On both structures we apply the Design Lemma, cf.\@ \cref{proc:design}, obtaining canonically colored partitions or canonically embedded non-trivial Johnson schemes. We have prepared everything to finish the reduction with \cref{sec:align}.

\chapter{Local certificates}
\label{sec:localcert}
We will now discuss the crucial part of Babai's algorithm. The situation is the same as in the previous section. We possess an epimorphism $\varphi: G \to \Alt(\Gamma)$ and a surjection $\iota: \Omega \to \Sym(\Omega)$. We can suppose that $G$ acts imprimitvely.
In order to exploit the correspondence between $\Omega$ and $\Gamma$ we introduce the following slightly abusive notions:
\begin{notation} \label{not:phi-stabilizer}
For a homomorphism $\varphi: G \to \Alt(\Gamma)$ and a set $T \subseteq \Gamma$ we define the stabilizers
$$G_T = \left\{ g \in G\ \middle|\ T^{\varphi(g)} = T \right\}, \quad G_{(T)} = \left\{ g \in G\ \middle|\ \forall t \in T: t^{\varphi(g)} = t \right\}.$$
\end{notation}

The strategy is to turn local information into knowledge about the global structure of the set of isomorphisms $\Iso_G(\str x, \str y)$. The procedure works as follows: First, we look only at one of the strings, say $\str{x}$, and try to encase its automorphism group $\Aut_G(\str{x})$ from above and from below. Approximating it from above involves finding canonical structures in the string which must be preserved by any automorphism, while for bounding the group from below constructing explicit automorphisms will be crucial. We attack this problem locally. For moderately sized test sets $T \subseteq \Gamma$ we verify either that the image under $\varphi$ of the stabilized automorphism group $\Aut_{G_T}(\str{x})$ contains $\Alt(T)$ or that this image is contained in a proper subgroup of $\Alt(T)$. The first outcome will be called \emph{certificate of fullness} while we refer to the latter as a \emph{certificate of non-fullness}. After collecting certificates for all $T \subseteq \Gamma$ of a certain size, this information is aggregated into knowledge about the global structure of $\Aut_G(\str{x})$. Now, we proceed similarly with the second string $\str{y}$. \Cref{sec:localcert1} is about the construction of local certificates. Comparing local certificates will be the subject of \cref{sec:localcert2}. In \cref{sec:agg-cert} we will explain how these tools allow us to detect canonical structures which are the basis for an efficient reduction of the problem described in \cref{sec:align}.

Formally, the two types of certificates we wish for are defined as follows:

\begin{defn}
	Let $\str x: \Omega \to \Sigma$, $G \leq \Sym(\Omega)$, $\varphi:G \to \Sym(\Gamma)$. Let $T \subseteq \Gamma$.
	\begin{enumerate}
		\item A \emph{certificate of non-fullness} for $T$ is a tuple $(W, M(T))$ where $W \subseteq \Omega$ the window, $ M(T) \leq \Sym(T)$, $M(T) \neq \Alt(T)$ and $\restrict{\varphi\left(\Aut_{G_T}^W(\mathbf{x})\right)}{T} \leq M(T)$.%
		\footnote{For $T \subseteq \Gamma$ and $ H \leq G_T$, the set of permutations in the image of $H$ is denoted as $\restrict{\varphi(H)}{T} = \left\{ \restrict{\varphi(h)}{T} \mid h \in H \right\}$.}
		\item A \emph{certificate of fullness} for $T$ is group $K(T) \leq \Aut_{G_T}(\mathbf{x})$ and $\left. \varphi(K(T))\right|_T = \Alt(T)$.
	\end{enumerate}
\end{defn}

Note that the type of certificate which $T$ admits depends on the group $G$ and the string $\str x$. Thus, when considering different strings and the same test set $T$ we may or may not arrive at the same type of certificate.
The window measures how much of the input is considered. The strategy will be to iteratively enlarge the window by looking at specific elements in the domain $\Omega$:

\begin{defn}
	Let $G \leq \Sym(\Omega)$, $\varphi:G \to \Sym(\Gamma)$. An element $x \in \Omega$ is called \emph{affected} with respect to $\varphi$ and a subgroup $H \leq G$ if $\varphi(H_x)$ does not contain $\Alt(\Gamma)$. Given a tuple $(H, \varphi)$, $\Aff(H, \varphi) \subseteq \Omega$ denotes the set of all elements in $\Omega$ affected by $(H, \varphi)$.\footnote{This definition is consistent with \cite[Definition~4.2]{helfgott} and adopted from \cite[Definition~10.1.4]{babai}.}
\end{defn}

We quote one of Babai's results which justifies the correctness of the subsequent algorithm as proven in \cite[Theorem~8.3.5, Corollary~8.3.7]{babai}, cf.\@ \cite[Proposition~4.4]{helfgott}.
\begin{theorem} Let $G \leq \Sym(\Omega)$, $\varphi: G \to \Alt_k$ an epimorphism and $U = \Omega \setminus \Aff(G, \varphi)$ the set of unaffected elements. Then the following hold: \label{thm:afforb}
	\begin{enumerate}
		\item \emph{(Unaffected Stabilizer Theorem)} Suppose $k > \max\left\{8, 2+\log_2 n_0\right\}$ where $n_0$ is the length of the largest $G$-orbit. Then $\left(G_{(U)}\right)^\varphi = \Alt_k$.
		\item \emph{(Affected Orbit Lemma)} Assume $k > 5$. If $\Delta$ is a $G$-orbit containing some affected elements, then each orbit of $ker \varphi$ contained in $\Delta$ is of length $\leq \abs \Delta/k$.
	\end{enumerate}
\end{theorem}
The Unaffected Stabilizer Theorem and the Affected Orbit Lemma are results genuinely originating in the theory of permutation groups. The proof was simplified by P.~P.~Pálfy, according to \cite{babai}, and depends on the Classification of Finite Simple Group (CFSG), although Pyber \cite{pyber} proved the claims CFSG-independently for larger $k$. In contrary to \cref{proc:design}, their justification is rather short but far from being trivial. The assumption $k > \max\left\{8, 2+\log_2 n_0\right\}$ is tight as Babai outlines in \cite[Remark~8.2.5]{babai}.

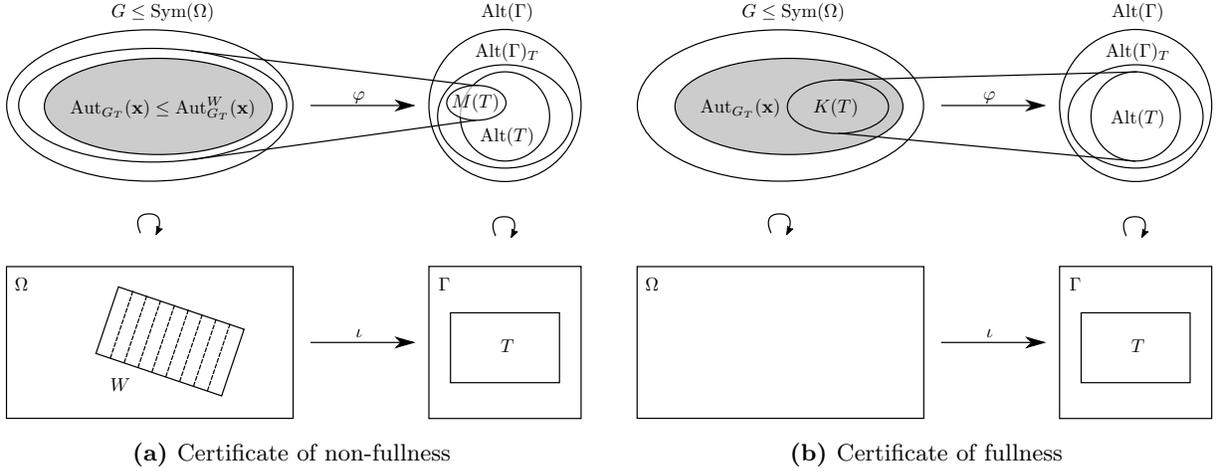
\begin{figure}
\centering
\begin{subfigure}{.45\textwidth}
\resizebox{\textwidth}{!}{
\begin{tikzpicture}
% General
\node[inner sep=0pt, anchor=north west] (russell) at (0,0) {\includegraphics[width=11cm]{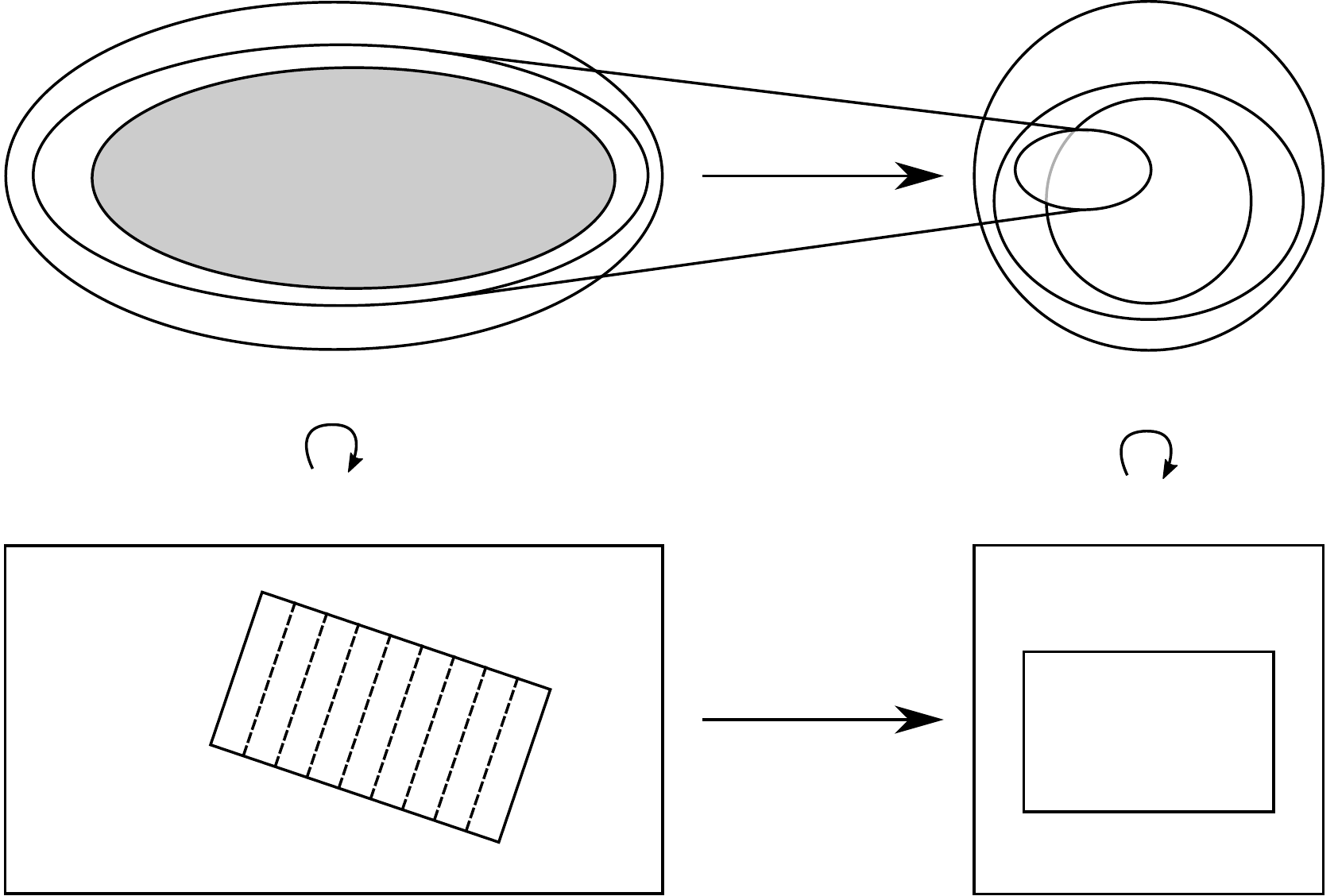}};
\node[inner sep=0pt] at (3,.3) {$G \leq \Sym(\Omega)$};
\node[inner sep=0pt] at (9.5,.3) {$\Alt(\Gamma)$};
\node[inner sep=0pt, anchor=north west] at (.2,-4.7) {$\Omega$};
\node[inner sep=0pt, anchor=north west] at (8.2,-4.7) {$\Gamma$};
\node[inner sep=0pt] at (9.5,-6) {$T$};
\node[inner sep=0pt] at (6.7,-1.3) {$\varphi$};
\node[inner sep=0pt] at (6.7,-5.75) {$\iota$};
\node[inner sep=0pt] at (9.5,-.45) {$\Alt(\Gamma)_T$};
% Special
\node[inner sep=0pt, anchor=north west] at (2,-6.6) {$W$};
\node[inner sep=0pt] at (3,-1.5) {$\Aut_{G_T}(\str x) \leq \Aut_{G_T}^W(\str x)$};
\node[inner sep=0pt] at (9.5,-2.05) {$\Alt(T)$};
\node[inner sep=0pt] at (8.9,-1.4) {$M(T)$};
\end{tikzpicture}
}
\caption{Certificate of non-fullness \label{fig:nonfull}}
\end{subfigure} \quad
\begin{subfigure}{.45\textwidth}
\resizebox{\textwidth}{!}{
\begin{tikzpicture}
% General
\node[inner sep=0pt, anchor=north west] (russell) at (0,0) {\includegraphics[width=11cm]{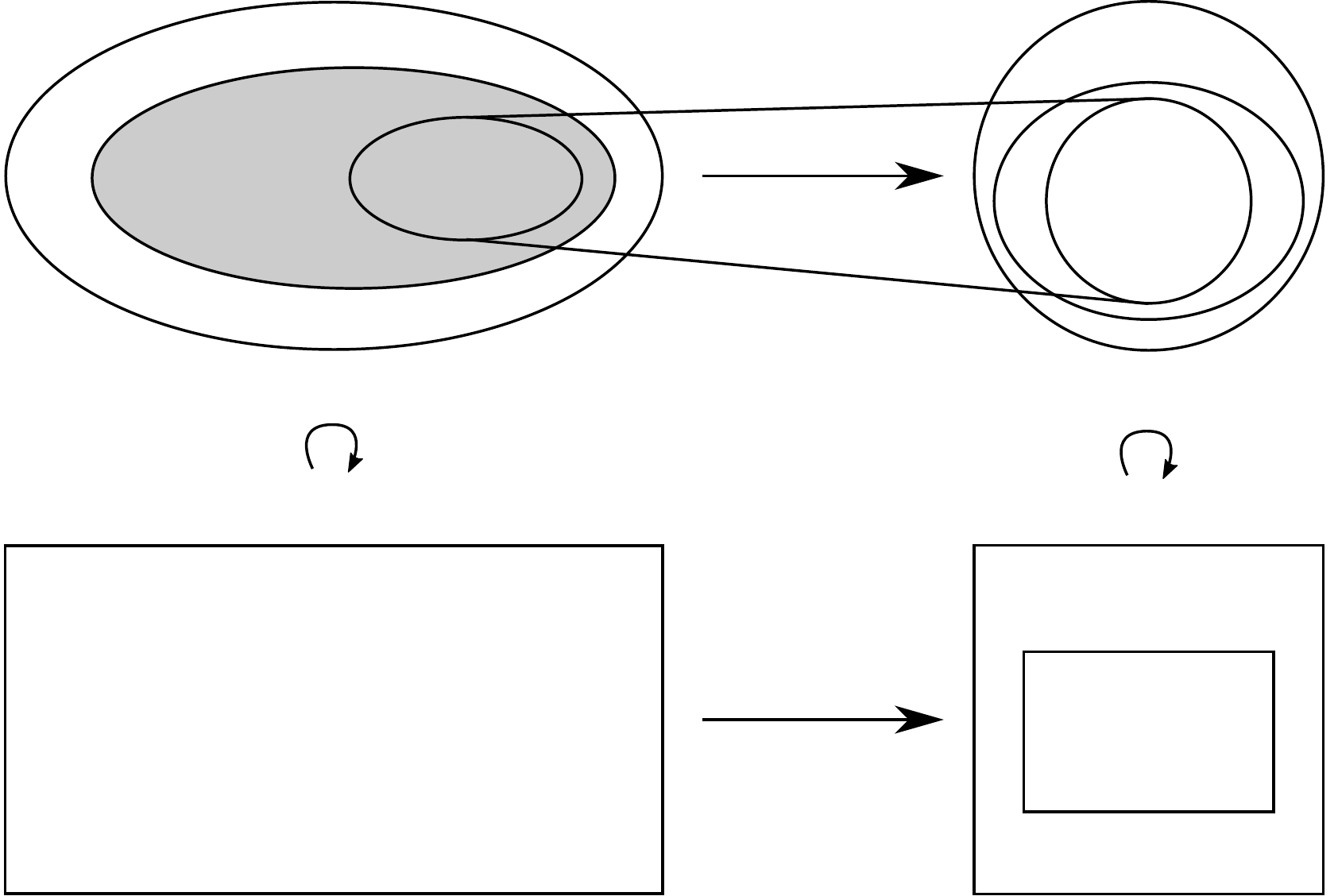}};
\node[inner sep=0pt] at (3,.3) {$G \leq \Sym(\Omega)$};
\node[inner sep=0pt] at (9.5,.3) {$\Alt(\Gamma)$};
\node[inner sep=0pt, anchor=north west] at (.2,-4.7) {$\Omega$};
\node[inner sep=0pt, anchor=north west] at (8.2,-4.7) {$\Gamma$};
\node[inner sep=0pt] at (9.5,-6) {$T$};
\node[inner sep=0pt] at (6.7,-1.3) {$\varphi$};
\node[inner sep=0pt] at (6.7,-5.75) {$\iota$};
\node[inner sep=0pt] at (9.5,-.45) {$\Alt(\Gamma)_T$};
% Special
\node[inner sep=0pt] at (9.5,-1.7) {$\Alt(T)$};
\node[inner sep=0pt] at (2,-1.5) {$\Aut_{G_T}(\str x)$};
\node[inner sep=0pt] at (3.8,-1.5) {$K(T)$};
\end{tikzpicture}
}
\caption{Certificate of fullness}
\end{subfigure}
\caption[Schematic overview of local certificates]{Schematic overview of local certificates. The group of interest in gray is $\Aut_{G_T}(\str x)$ for the test set $T$. It is either encased by $\Aut_{G_T}^W(\str x)$ from above (non-full) or by $K(T)$ from below (full). Crucial is whether the restriction to $T$ of the encasing group's $\varphi$-image is $\Alt(T)$, a giant. $G$ and the smaller $\Alt(\Gamma)$ are linked via $\varphi$ and $\iota: \Omega \to \binom{\Gamma}{l}$. For simplicity, we do not make a strict distinction between $\varphi$ and $g \mapsto \restrict{g^\varphi}{T}$.}
\end{figure}

\section{Construction of local certificates}
\label{sec:localcert1}
In this section we introduce an algorithm which verifies whether a test set $T \subseteq \Gamma$ is full or non-full. As explained before, this property does not depend on $T$ alone but also on $G$ and $\str x$. Throughout the procedure we consider more and more of the input string. Babai visualizes this strategy as \enquote{growing the beard}, cf.\@ \cref{fig:nonfull}. We will realize this iterative enlargement of the window by looking at affected elements.

In contrary to \cite[§6.1.1]{helfgott}, the algorithm will be presented in less natural language following \cite[§10.1]{babai}. For convenience, $A(W)$ will stand for $\Aut_{G_T}^W(\str x) \leq G_T$. As $W = \emptyset$ imposes an vacuous condition on $A(W)$, the group equals $G_T$ at the beginning. Throughout the procedure $W$ is enlarged which implies that $A(W)$ shrinks.
For our test set $T$ we set $k \coloneqq \abs T$ and suppose that $\max\left\{8,2+\log_2 \abs \Omega\right\} < k \ll \log \abs \Omega$. Thus, $T$ is of logarithmic size in $\abs \Omega$.

\begin{procedure}{Local Certificates}
	{test set $T \in \binom{\Gamma}{k}$, group $G \leq \Sym(\Omega)$, epimorphism $\varphi: G \to \Alt(\Gamma)$, string $\str x: \Omega \to \Sigma$}
	{either a certificate of fullness $K(T)$ or a certificate of non-fullness $(W, M(T))$}
	{localcert}
	\comment{Initialization}
	\item $W \leftarrow \emptyset$
	\item $A(W) \leftarrow G_T = \varphi^{-1}(\Alt(\Gamma)_T)$ \label{line:localcert-gt}
	\comment{Iteration: while $T$ is potentially full and non-full, i.e.\@ the image of $A(W)$ contains $\Alt(T)$ and there are still affected elements which can be added to $W$}
	\item \textbf{while} $\restrict{\varphi(A(W))}{T} \geq \Alt(T)$ and $\Aff(A(W), \varphi) \not\subseteq W$: \label{line:local-cert-while}
	\comment{\cindent Enlarge the window and recompute $A(W)$ accordingly}	
	\item \cindent $W^+ \leftarrow \Aff(A(W), \varphi)$ \label{line:localcert-aff}
	\item \cindent $N \leftarrow \ker \restrict{\varphi}{A(W)} = A(W)_{(T)}$, \textit{where\footnote{As $A(W) \leq G_T$, we can restrict $\varphi$ in a well-defined manner to a map $A(W) \to \Alt(T)$.} $\restrict{\varphi}{A(W)}: A(W) \to \Alt(T)$} \label{line:localcert-ker}
	\item \cindent $A(W^+) \leftarrow \emptyset$
	\item \cindent compute a set of right coset representatives $S$ for $N$ in $A(W)$
	\item \cindent \textbf{for} $\sigma \in S$:
	\item \cindent \cindent $A(W^+) \leftarrow A(W^+) \cup \Iso_N^{W^+}\left( \str x, \str{x}^{\sigma^{-1}} \right) \sigma$, \textit{by recurring on orbits}
	\item \cindent \textbf{end}
	\item \cindent $W \leftarrow W^+, A(W) \leftarrow A(W^+)$ \label{line:localcert-recompute-end}
	\item \textbf{end}
	\item \textbf{if} $\restrict{\varphi(A(W))}{T} \geq \Alt(T)$:
	\item \cindent $K(T) \leftarrow A(W)_{(\Omega \setminus W)}$
	\item \cindent \textbf{return} certificate of fullness $K(T)$
	\item \textbf{else}:
	\item \cindent $M(T) \leftarrow \restrict{\varphi(A(W))}{T}$
	\item \cindent \textbf{return} certificate of non-fullness $(W, M(T))$
	\item \textbf{end}
\end{procedure}

Our goal is to verify whether $\restrict{\varphi(\Aut_{G_T}(\str x))}{T}$ is or is not equal to $\Alt(T)$. In every step we enlarge the window by the elements affected by $\Aut_{G_T}^W(\str x)$. Two scenarios trigger a termination. Either the window stops growing or $\restrict{\varphi(A(W))}{T}$ no longer contains the alternating group on $T$. \Cref{thm:afforb} is crucial: If as in the first case $W$ contains all $(A(W),\varphi)$-affected elements and $\restrict{\varphi}{A(W)}: A(W) \to \Alt(T)$ is surjective, then $\restrict{\varphi\left(A(W)_{(\Omega \setminus W)}\right)}{T} = \Alt(T)$. Furthermore, $A(W)_{(\Omega \setminus W)} \leq \Aut_{G_T}(\str x)$ as the points in the complement of $T$ are fixed. Consequently, we have evaluated the action of a subgroup on the entire string without looking at the entirety of the input. In the other case, i.e.\@ when the restricted map is not surjective, we computed a certificate of non-fullness which verifies that the automorphism group of interest is not equal to the alternating group. Importantly, the assignments of the certificates to $\str x $ are canonical.

The procedure contains a couple of non-trivial operations whose execution times must be carefully analyzed.
In line~\ref{line:localcert-gt} we have to determine $G_T$ for a given $T$. Reverting to \cref{sec:schreier-sims,sec:primitive}, we do so by computing the preimage of $\Alt(\Gamma)_T$ under $\varphi$ after choosing five generators for the latter group. By the surjectivity of $\varphi$, $\frac{\abs{\Gamma}!}{2} \leq \abs{G} \leq \abs{\Omega}!$ and thus $\abs{\Gamma} \ll \abs{\Omega}$, which guarantees a polynomial running time.

The set of affected elements $\Aff(A(W), \varphi)$ in line~\ref{line:localcert-aff} is computed by iterating over all $x \in \Omega$, checking whether $\restrict{\varphi(A(W)_x)}{T} = \Alt(T)$. Both the computation of $A(W)_x$ and the subsequent verification, are done using a modified version of Schreier-Sims. The conditions in line~\ref{line:local-cert-while} are straightforward to check once we have computed $\Aff(A(W), \varphi)$.

For line~\ref{line:localcert-ker} we recall that we hold a description of $A(W)$ in terms of the generators of $G$. Thus, we can compute the kernel of $\restrict{\varphi}{A(W)}$ by taking the preimage of the trivial group in polynomial time. We may have to modify the generators of $A(W)$ in order to reflect the restriction of the image from $\Alt(\Gamma)$ to $\Alt(T)$, but by Schreier-Sims the generating set is not too big.

The purpose of the for-loop is to compute $A(W^+)$. Prior to iterating, we start with computing a set $S$ of right coset representatives for $N$ in $A(W)$, cf.\@ \cref{sec:schreier-sims}. It is clear that $N = A(W)_{(T)}$ admits polynomial membership testing.
Since the quotient $A(W)/N$ is isomorphic to $\Alt(T)$, we find $k!/2$ such representatives. Thus, computing them took quasi-polynomial time.
As $A(W^+) \subseteq  A(W)$ and $A(W) = \bigcup_{\sigma \in S} N\sigma$, we have by \cref{lemma:isoaut} that 
\begin{equation} \label{eq:chain-aut}
	A(W^+) = \Aut_{A(W)}^{W^+} (\str x) = \Aut_{\bigcup N\sigma}^{W^+} (\str x) = \bigcup_\sigma \Aut_{N\sigma}^{W^+}(\str x) = \bigcup_\sigma \Iso_N^{W^+}\left(\str x, \str{x}^{\sigma^{-1}}\right) \sigma.
\end{equation}
We can compute this union as in \cref{eq:coset-union}.

Since all elements in $W^+$ are by definition $(A(W),\varphi)$-affected, they are all $(N,\varphi)$-affected. Thus, by \cref{thm:afforb}, all orbits of $N$ that are contained in $W^+$ are of size $\leq \abs{W^+}/k \leq \abs{\Omega}/k$. We exploit this intransitivity to further reduce the subproblems.

For every $\sigma$, we perform strong Luks reduction: In polynomial time we compute a partition of $W^+$ into the orbits under the action of $N$. This partition is $N$-invariant and acted upon trivially by $N$. Calling \cref{proc:strong-luks} we reduce to one subproblem for each orbit.

After passing through all $\leq \abs \Omega$ orbits and all $\sigma$ we will have called the main procedure $\leq \abs{\Omega}(k!/2)$ times on strings of length $\leq \abs{\Omega}/k$. The additive costs are $\abs{\Omega}^{O(\log \abs{\Omega})}$ in every iteration. As $W \subseteq \Omega$ is constantly growing, the while-loop's body will be executed $\leq \abs{\Omega}$ times. Thus, in total we recur to $\leq \abs{\Omega}^2(k!/2)$ subproblems of size $\leq \abs{\Omega}/k$. As $k \ll \log \abs\Omega$, this complies with the desired bound as we will see in \cref{sec:complexity}. All steps considered, the additive costs account for $\abs{\Omega}^{O(\log \abs{\Omega})}$ operations.

\section{Comparing certificates}
\label{sec:localcert2}
In practice, that is in the application in \cref{sec:agg-cert}, we are not only interested in local certificates for one string but in the correspondence of the certificates for both input strings. Hence, we require a procedure which allows us to efficiently compare local certificates. By comparing two certificates we mean verifying whether there exist string isomorphisms sending one test set pointwise to the other and respecting certain parts of the input strings.

The setting is as follows: We are given two strings $\str x, \str y: \Omega \to \Sigma$, two test sets $T, T' \subseteq \Gamma$ of the same size $k = \abs T = \abs{T'}$. Again, we will look at increasing parts of the strings. The following notation will be useful.

\begin{notation}[Truncation of strings]
Let $\str x: \Omega \to \Sigma$ a string, $W \subseteq \Omega$ the window. Then the \emph{truncated string} $\str x^W: \Omega \to \Sigma \mathbin{\dot{\cup}} \{\text{glaucous}\}$ is defined as 
\[
\str x^S(i) \coloneqq \begin{cases} \str x(i), & i \in S, \\ \text{glaucous}, & i \in \Omega \setminus S, \end{cases}
\] 
where glaucous is a symbol alien to the alphabet $\Sigma$.
\end{notation}

\begin{notation}
For a homomorphism $\varphi: G \to \Alt(\Gamma)$ and two sets $T, T' \subseteq \Gamma$ whose elements are ordered, i.e.\@ $T = \{t_1, \dots, t_k\}, T' = \{t'_1, \dots, t'_k\}$, we define
\[
G_{T,T'} \coloneqq \left\{ g \in G\ \middle|\ T^{\varphi(g)} = T' \right\}, \quad 
G_{(T, T')} \coloneqq \left\{ g \in G\ \middle|\ \forall i = 1, \dots, k: t_i^{\varphi(g)} = t'_i \right\}.
\]
It is clear, that these stabilizers can be written as $G_{T,T'} = G_T \sigma$ and $G_{(T,T')} = G_{(T)} \tau$ for $\sigma \in G_{T,T'}$ and $\tau \in G_{(T,T')}$.
\end{notation}

The objective of this section is to compute $\Iso_{G_{T,T'}}\left(\str x^{W(T)}, \str y^{W(T')}\right)$ where $W(T)$ and $W(T')$ are the windows returned by \cref{proc:localcert} for inputs $T$ and $T'$ respectively. In other words, we are interested in the elements of $G$ whose $\varphi$-images send $T$ to $T'$ and which respect the input strings $\str x$ and $\str y$ at least on the given windows. 
We obtain the result by running \cref{proc:localcert} in parallel for both input strings. While enlarging the windows $W(T)$ and $W(T')$ we keep track of $\Iso_{G_{T,T'}}\left( \str x^W, \str y^{W'} \right)$ which we will call $Q(W,W')$.

In the beginning, both windows are empty. Thus $Q(W,W') = G_{T,T'}$. The more the windows are enlarged the more $Q(W,W')$ shrinks. Again, the test sets must satisfy $\max\left\{8,2+\log_2 \abs \Omega\right\} < k \ll \log \abs \Omega$ where $k = \abs T = \abs{T'}$.

\begin{procedure}
	{Comparing Local Certificates}
	{test sets $T, T' \in \binom{\Gamma}{k}$, group $G \leq \Sym(\Omega)$, epimorphism $G \to \Alt(\Gamma)$, strings $\str x, \str y: \Omega \to \Sigma$}
	{windows $W(T), W(T') \subseteq \Omega$, the group $\Iso_{G_{T,T'}}\left(\str x^{W(T)}, \str y^{W(T')}\right)$}
	{localcert2}
	\comment{Initialization}
	\item $W \leftarrow \emptyset, W' \leftarrow \emptyset$
	\item $A(W) \leftarrow G_T = \varphi^{-1}(\Alt(\Gamma)_T), A(W') \leftarrow G_{T'} = \varphi^{-1}(\Alt(\Gamma)_{T'})$
	\item $Q(W, W') \leftarrow G_{T,T'} = A(W) \sigma$, for any $\sigma \in G_{T,T'}$ \label{line:localcert2-coset}
	\item \textbf{while} $\restrict{\varphi(A(W))}{T}~\geq~\Alt(T)$ and $\Aff(A(W), \varphi)~\not\subseteq~W$ and $\restrict{\varphi(A(W'))}{T'}~\geq~\Alt(T')$ and $\Aff(A(W'), \varphi)~\not\subseteq~W'$:
	\comment{\cindent Enlarge the windows and recompute $A(W), A(W')$ and $Q(W,W')$ accordingly}	
	\item \cindent $W^+ \leftarrow \Aff(A(W), \varphi)$, $W'^+ \leftarrow \Aff(A(W'), \varphi)$
	\item \cindent $N \leftarrow \ker \restrict{\varphi}{A(W)} = A(W)_{(T)}$
	\item \cindent $A(W^+) \leftarrow \emptyset, Q^+ \leftarrow \emptyset$
	\item \cindent fix $\pi_0 \in Q(W,W')$ \label{line:localcert2-fix-pi}
	\item \cindent compute a set of right coset representatives $S$ for $N$ in $A(W)$
	\item \cindent \textbf{for} $\sigma \in S$:
	\item \cindent \cindent $A(W^+) \leftarrow A(W^+) \cup \Iso_N^{W^+}\left( \str x, \str{x}^{\sigma^{-1}} \right) \sigma$ \label{line:localcert2-recompute1}
	\item \cindent \cindent $\pi \leftarrow \sigma \pi_0$
	\item \cindent \cindent $Q^+ \leftarrow Q^+ \cup \Iso_N \left( \str x^{W^+}, \left({\str y}^{W'^+}\right)^{\pi^{-1}} \right) \pi$
	\item \cindent \textbf{end}
	\item \cindent $W \leftarrow W^+, A(W) \leftarrow A(W^+)$
	\item \cindent complete recomputing $W'$ and $A(W')$ as in \cref{proc:localcert} lines~\ref{line:localcert-aff}-\ref{line:localcert-recompute-end}.	\label{line:localcert2-recompute2}
	\item \cindent \textbf{if} $Q^+ = \emptyset$: reject isomorphicity, \textbf{exit}
	\item \cindent $Q(W,W') \leftarrow Q^+$
	\item \textbf{end}
	\item \textbf{return} windows $W$ as $W(T)$, $W'$ as $W(T')$ and $Q(W,W')$
\end{procedure}

Recomputing $Q(W,W')$ requires some attention. Following \cref{eq:chain-aut}, we reuse the coset representatives for $N = \ker \restrict{\varphi}{A(W)}$ in $A(W)$ which together with the fixed element $\pi_0 \in Q(W,W')$ serve as representatives for the cosets $N\sigma \pi_0$ of $N$ in $A(W)\pi_0 = Q(W,W')$. For $W \subseteq W^+, W' \subseteq W'^+$, obviously $Q(W^+, W'^+) \leq Q(W,W') \leq G_{T,T'}$. Hence, abbreviating $\pi \coloneqq \sigma \pi_0$, we have
\begin{align}
Q(W^+, W'^+) 
&= \Iso_{G_{T,T'}}\left( \str x^{W^+}, \str y^{W'^+} \right)
= \Iso_{Q(W, W')}\left( \str x^{W^+}, \str y^{W'^+} \right)
= \bigcup_{\sigma} \Iso_{N\sigma\pi_0} \left( \str x^{W^+}, \str y^{W'^+} \right) \nonumber \\
&= \bigcup_{\pi} \Iso_{N}  \left( \str x^{W^+}, \left(\str y^{W'^+} \right)^{\pi^{-1}}\right) \pi
\end{align}
As in the previous case, all orbits of $N$ in $W^+$ are of size $\leq \abs{W^+}/k \leq \abs \Omega / k$. We resume with Luks reduction on the $\leq \abs \Omega$ orbits. As $\gindex{A(W)}{N} = k!/2$, we have to make additional $\abs{\Omega}^2 (k!/2)$ calls to the main routine for strings of length $\leq \abs \Omega / k$.

Finding an element $\sigma$ in line~\ref{line:localcert2-coset} satisfying $G_{T,T'} = G_T \sigma$ is possible in polynomial time by Schreier-Sims, cf.\@ \cref{sec:schreier-sims}, as any $\rho \in \Alt(\Gamma)$ mapping $T^\rho = T'$ has a non empty preimage $\varphi^{-1}(\{\rho\})$ from which $\sigma$ can be chosen. Such a $\rho$ can be easily constructed by enumerating the elements of $T$ and $T'$.

Besides the treatment of $Q(W,W')$ the recomputation of $A(W)$ and $A(W')$ in lines~\ref{line:localcert2-recompute1} and \ref{line:localcert2-recompute2} are expensive. Summing up, we have recurred to $3\abs{\Omega}^2 (k!/2)$ subproblems of size $\leq \abs{\Omega}/k$. The additive costs account for $\abs{\Omega}^{O(\log \abs \Omega)}$ operations. This complies with our desired bound as we will see in \cref{sec:complexity}.

From a computational point of view, sets are much harder to handle than indexed lists. Hence, we are not only interested in $\Iso_{G_{T,T'}}\left(\str x^{W(T)}, \str y^{W(T')}\right) = Q$ but in $\Iso_{G_{(T,T')}}\left(\str x^{W(T)}, \str y^{W(T')}\right)$, i.e.\@ the set of isomorphism taken from $G$ which map the tuple $T = \{t_1,\dots, t_k\}$ to $T' = \{t'_1,\dots, t'_k\}$ in the right order. After obtaining the former from \cref{proc:localcert2}, it is passed to a modified version of Schreier-Sims which in virtue of $\varphi$ provides generators for $Q_i = \left\{\sigma \in Q\ \middle|\ \forall l = 1, \dots, i: t_l^{\varphi(\sigma)} = t'_l \right\}$ for $i \leq k-1$. As $Q \leq G_{T,T'}$, $\gindex{Q_i}{Q_{i+1}} \leq k$, which assures a polynomial execution time. Finally, $Q_{k-1} = \Iso_{G_{(T,T')}}\left(\str x^{W(T)}, \str y^{W(T')}\right)$. This justifies the following procedure.

\begin{proc}[Comparing Local Certificates for Tuples] \label{proc:localcert3}
	\Input test sets $T, T' \in \Gamma^k$, group $G \leq \Sym(\Omega)$, epimorphism $G \to \Alt(\Gamma)$, strings $\str x, \str y: \Omega \to \Sigma$ \\
	\Output windows $W(T), W(T') \subseteq \Omega$, the group $\Iso_{G_{(T,T')}}\left(\str x^{W(T)}, \str y^{W(T')}\right)$ \\
	\Complexity reduction to $3\abs{\Omega}^2 (k!/2)$ subproblems of size $\leq \abs{\Omega}/k$, additive costs $\abs{\Omega}^{O(\log \abs \Omega)}$
\end{proc}

\section{Aggregation of certificates}
\label{sec:agg-cert}
In this section we will combine the tools provided in the two preceding sections in order to reduce the problem of determining $\Iso_G(\str x, \str y)$ to subproblems that can be solved by other algorithmic building blocks. We are again equipped with an epimorphism $\varphi: G \to \Alt(\Gamma)$. Choose $k \in \mathbb{N}$ such that $\max\left\{8,2+\log_2 \abs \Omega\right\} < k < \abs{\Gamma}/10$ and $k \ll \log\abs\Omega$, for example $k \coloneqq \ceil{8+ \log_2\abs\Omega}$. Then $\abs\Omega/k \leq \abs\Omega/2$.

After computing and comparing local certificates we will have obtained one of the following results:
\begin{description}
	\item[Case 1] canonical colored $1/2$-partitions on $\Gamma$ for the two strings $\str x$ and $\str y$,
	\item[Case 2a] reduction to $\leq 3\abs \Omega$ many subproblems for strings of length $\leq \abs \Omega/2$,
	\item[Case 2b] canonically embedded binary relational structures on $\Gamma$ with symmetry defect $\geq 1/2$ for the two strings $\str x$ and $\str y$ and $\leq \abs\Omega$ subproblems of length $\leq \abs \Omega/2$ at additional multiplicative costs,
	\item[Case 3] canonically embedded $k$-ary relational structures on $\Gamma$ with symmetry defect $> 4/5$ for the two strings $\str x$ and $\str y$ and $\leq \abs\Omega$ subproblems subproblem of length $\leq \abs \Omega/2$.
\end{description}

We will leave the overall cost estimate to \cref{sec:complexity}. Ludi incipiant:

First we run \cref{proc:localcert} for both input strings $\str x $ and $\str y$ and all test sets $T \subseteq \Gamma$ of size $k$. For every such test set and both of the strings we obtain either a certificate of fullness or a certificate of non-fullness. Let $F(\str x) \leq \Aut_G(\str x)$ denote the group generated by the certificates of fullness $K(T)$ associated to $\str x$, $F(\str y)$ respectively. Let $S(\str x) \subseteq \Gamma$ denote the support of $F(\str x)^\varphi$, i.e.\@ the set of all elements of $\Gamma$ that are not fixed by all elements of the image of $F(\str x)$ under $\varphi$. In total, we thus make $2\binom{\abs{\Gamma}}{k} \leq 2 \frac{\abs\Gamma^k}{k!} \leq 2 \abs\Gamma^k \ll \abs\Omega^{O(\log \abs\Omega)}$ calls to \cref{proc:localcert} resulting in $\leq \abs \Gamma^k \abs \Omega^2 \ll \abs\Omega^{O(\log \abs\Omega)}$ many subproblems of length $\leq \abs\Omega/2$ and additive costs $\abs\Omega^{O(\log \abs\Omega)}$.

From now on, we omit the reference to the string in our notation. Implicitly, we work on both strings, distinguishing cases. If we arrive with the two strings in different cases, we can reject isomorphicity. By the canonicity of the assignment of the certificates, $F(\str x)$ and $S(\str x)$ are canonical. We will write $F \coloneqq F(\str x)$ and $S \coloneqq S(\str x)$. In particular, we can assume that $F(\str x)$ and $F(\str y)$ are conjugate subgroups of $\Sym(\Omega)$. $S(\str x)$ and $S(\str y)$ are not necessary equal as sets but of the same size.

The numbering of the cases follows \cite[§6.2]{helfgott} and \cite[§13.2]{babai2}. For checking the case's preconditions, we compute the orbits of the action of $\varphi(F)$ on $\Gamma$ in polynomial time in $\abs \Gamma \ll \abs \Omega$.

\paragraph{Case 1} $\abs S \geq \abs \Gamma/2$, but no orbit of $\varphi(F)$ is of length $> \abs \Gamma/2$. \nopagebreak

We want to compute canonical colored $1/2$-partitions on $\Gamma$ depending on the orbit structure of $\varphi(F)$, cf.\@ \cref{def:colored-partition}. We will obtain two such partitions associated to $\str x$ and $\str y$ respectively since $\varphi(F(\str x))$ and $\varphi(F(\str y))$ are not identical but only conjugate as subgroups of $\Alt(\Gamma)$. Below the reference to the string will be omitted from the notation.

As the orbit structure is canonical, we can color the elements of $\Gamma$ according to the lengths of their respective orbits in a canonical way. This is surely possible in polynomial time. If no color class dominates, i.e.\@ no set of elements in $\Gamma$ with a certain color is of size $> \abs \Gamma/2$, we have found a canonical colored $1/2$-partition. If contrarily such a big color class exists, it must partition at least into two sets of the same size $\geq 2$, as there would otherwise be an orbit of length $> \abs \Gamma/2$ or a trivial orbit in contradiction to the assumption. Refining the color classes according to this partition yields a canonical colored $1/2$-partition as well.

\paragraph{Case 2} $\abs S \geq \abs \Gamma/2$ and there is an orbit $\Phi$ of $\varphi(F)$ in $\Gamma$ satisfying $\abs \Phi > \abs \Gamma/2$. \nopagebreak

Since there can only be one such orbit associated to each string, $\Phi$ is again canonical. Note that this does not mean that the string's dominating orbits are equal as sets. Let $\Phi_{\str x}$ and $\Phi_{\str y}$ denote the orbits of this kind associated to the strings $\str x$ and $\str y$ respectively. They must be of the same size, otherwise we have refuted isomorphicity. We align the two orbits as we did in \cref{sec:primitive}. First, we construct $\sigma \in \Alt(\Gamma)$ such that $\Phi_{\str x}^\sigma = \Phi_{\str y}$ by enumerating the orbits' elements. Now $\sigma$ can be lifted along the epimorphism $\varphi$ using Schreier-Sims to a $\tau \in G$ satisfying $\tau^\varphi = \sigma$. Replacing $\str y \leftarrow \str y^\tau$ finishes the alignment. We can hence assume that $\Phi_{\str x} = \Phi_{\str y}$, which we will call $\Phi$ again. 

Reverting to \cref{sec:schreier-sims}, the preconditions of the two subsequent subcases can be tested in polynomial time. 

\subparagraph{Case 2a} $\Alt(\Phi) \leq \restrict{\varphi(F)}{\Phi}$. \nopagebreak

In this case we have found a feature of large symmetry in both strings. Recall, that we hold a surjection $\iota: \Omega \to \binom{\Gamma}{l}$. Invoking \cref{proc:pullback}, we let the partition $(\Gamma \setminus \Phi) \mathbin{\dot\cup} \Phi$ of $\Gamma$ induce a partition $\Omega_0 \mathbin{\dot\cup} \cdots \mathbin{\dot\cup} \Omega_l$ of $\Omega$. It is guaranteed that $\Omega_0 = \iota^{-1}\binom{\Phi}{l}$ and hence $\abs{\Omega_j} \leq \abs \Omega/2$ for all $0 < j \leq l$ by the precondition of case~2.

Due to the orbit's canonicity, any possible isomorphism from $\str x$ and $\str y$ must leave $\Phi$ as a set invariant, i.e.\@ a $\sigma \in G$ can only satisfy $\str x = \str y^\sigma$ if $\Phi^{\varphi(\sigma)} = \Phi$. Therefore we have reduced the problem of determining $\Iso_G(\str x, \str y)$ to the problem of computing $\Iso_H(\str x, \str y)$ where $H = \varphi^{-1}\left(\Alt(\Gamma)_\Phi\right)$. Instead of brutally computing a description for $H$, we scrutinize its structure and perform a more sophisticated reduction. Let $K = \varphi^{-1}\left(\Alt(\Gamma)_{(\Phi)} \right)$ and $\sigma_1, \sigma_2 \in G$ be preimages under $\varphi$ of two generators of $\Alt(\Phi) \leq \restrict{\Alt(\Gamma)}{\Phi}$, cf.\@ \cref{eq:generators-alt}. Recalling \cref{sec:schreier-sims}, we compute $K$ as a preimage of a pointwise stabilizer in polynomial time. As $\Alt(\Phi) \leq \restrict{\varphi(F)}{\Phi} \leq \restrict{\Aut_{G}(\str x)}{\Phi}$, the two sets $\Aut_{K\sigma_i}(\str x)$ for $i=1,2$ are non-empty.

By the correspondence of $\varphi$ and $\iota$, $K$ acts on $\Omega$ but fixes $\Omega_0$. The orbits of this action refine the partition of $\Omega$ into the $\Omega_j$ because the number of elements from $\Phi$ in an element of $\binom{\Gamma}{l}$ is invariant under the action of $K$. Thus all orbits of $K$ are at most of length $\abs \Omega/2$. Consequently, we can compute the two $\Aut_{K\sigma_i}(\str x) = \Iso_K\left(\str x, \str x^{\sigma_i^{-1}} \right) \sigma_i$ by calling the main procedure for strings of length $\leq \abs \Omega/2$ and total length $\leq 2\abs \Omega$. The number of calls is $\leq 2l \leq 2\abs \Omega$ and the costs for the reduction are negligible. As $\Alt(\Gamma)_{\Phi}/\Alt(\Gamma)_{(\Phi)} \cong \Alt(\Phi)$, the two $\Aut_{K\sigma_i}(\str x)$ generate $\Aut_H(\str x)$. By the same argument,
\[
\Iso_H(\str x, \str y) = \Aut_H(\str x) \Iso_K(\str x, \str y),
\]
iff $\Iso_K(\str x, \str y)$ is non-empty. The latter is the case if and only if $\Iso_H(\str x, \str y)$ is non-empty because by assumption $\Alt(\Phi) \leq \restrict{\varphi(F)}{\Phi}$. As above, the shortness of $K$'s orbits allows us to compute $\Iso_K(\str x, \str y)$ by calling the main procedure $\leq l \leq \abs\Omega$ times on strings of length $\leq  \abs \Omega/2$ and total length $\leq \abs \Omega$. This completes the reduction.

\subparagraph{Case 2b} $\Alt(\Phi) \not\leq \restrict{\varphi(F)}{\Phi}$. \nopagebreak

Let $d \geq 1$ denote the maximal integer such that $\restrict{\varphi(F)}{\Phi}$ acts $d$-transitively on $\Phi$, i.e.\@ the induced action of $\restrict{\varphi(F)}{\Phi}$ on the set of tuples of length $d$ drawn from $\Phi$ with distinct elements is transitive. \cite[Theorem~13.1.1, Remark~13.1.2]{babai} following \cite{wielandt} implies that $d \leq 5$, or without dependence on the Classification of Finite Simple Groups $d \ll \log \abs \Gamma$. We want to ensure that $\restrict{\varphi(F)}{\Phi}$  is transitive but not doubly transitive, i.e. $d=1$. If this is the case already, the next paragraph can be skipped. We determine $d$ by checking iteratively for $i = 1,\dots, (d+1)$ whether $\restrict{\varphi(F)}{\Phi}$ is $i$-transitive. This takes $\binom{\abs\Gamma}{d+1}^{O(1)} \ll \abs\Omega^{O(1)}$ elementary operations, cf.\@ \cref{sec:orbitsblocks}.

We choose any set $T \in \binom{\Phi}{d-1}$ and individualize its elements. Heuristically individualization means to assign new unique colors to certain elements. Any isomorphism must preserve these colors which of course limits the number of possible isomorphisms drastically. On the other hand, any further treatment will depend on the choice of the individualized elements. Thus, the algorithm branches. The challenge is to keep the size of the subproblems and the intensity of individualization balanced. Applied to our case this means that we perform weak Luks reduction to the subgroup $G_{(T)}$, cf.\@ \cref{proc:weak-luks,not:phi-stabilizer}. The number of subproblems is
\[
	\gindex{G}{G_{(T)}} = \gindex{G}{\varphi^{-1}\left(\Alt(\Gamma)_{(T)}\right)} = \frac{\abs G}{\abs{\ker \varphi} \abs{\Alt(\Gamma\setminus T)}}
	= \frac{\abs{G/N}}{\abs{\Alt(\Gamma\setminus T)}}
	= \frac{\abs \Gamma!}{(\abs \Gamma-d+1)!}
	\leq \abs{\Gamma}^{d-1} \ll \abs \Gamma^{O(1)}.
\]

Clearly, for $k > 1$, a group is $k$-transitive if and only if it is $(k-1)$-transitive and each of the $(k-1)$-point stabilizers are transitive on the set of the remaining points, cf.\@ \cite[p.\@~211]{dixon}. For $\Phi' = \Phi \setminus T$, we conclude that $\restrict{\varphi(F)_{(T)}}{\Phi'}$ acts transitively but is not doubly transitive. We replace $\varphi$ with $g \mapsto \restrict{g^\varphi}{\Gamma \setminus T}$, rename $\Phi \leftarrow \Phi'$ and consider from now on $\restrict{\varphi(F)}{\Phi}$, a group satisfying $d=1$ in the notation from above. Surely, $\varphi: G \to \Alt(\Gamma \setminus T)$ is still an epimorphism.

We will construct two binary rational structures with large symmetry defect. One structure will be associated to $\str x$, the other to $\str y$. Until now, the treatment for both strings was the same. In particular, because $F(\str x)$ and $F(\str y)$ are conjugate and $\Phi$ has been aligned. What follows now, depends on the input string.

Let $\mathfrak{X} = (\Phi; R_1, \dots, R_r)$ be the orbital configuration of $\restrict{\varphi(F)}{\Phi}$ acting on $\Phi$, cf.\@ \cref{lemma:orbital-config}. It is non-trivial and homogeneous. Each orbital is at least of length $\abs{\Phi}$. Therefore $3 \leq r \leq \abs \Phi$. Choose a constituent graph $X$ other than the diagonal. Then, by \cref{lemma:biregular}, $X$ is a biregular non-trivial digraph.

We proceed with individualizing $X$. Note that $X$ is not a choice of points that any isomorphism should fix, as above when individualizing $T$, but a more complex structure. Despite that, the principle remains the same: We fixed $X$ for one of the input strings. Now the corresponding structure for the other string can be $X$ or one of the other off-diagonal constituent graphs. We have to branch into $r-1 \leq \abs \Gamma -1 \ll \abs \Omega$ subproblems, one for each of these choices. However, these multiplicative costs are acceptable. Having done this, we can assume that $X$ is a canonical binary relational structure on $\Phi$ shared by the two input strings. We apply \cref{cor:2413}. $X$ has a relative symmetry defect of $\geq 1/2$ on $\Phi$. 

We require a canonical structure with large symmetry defect on $\Gamma$. Problematically, $X$ covers only $\Phi$. But as in the preceding case, any isomorphism from $\str x$ to $\str y$ must preserve $\Phi$. Thus we can on one hand restrict $\varphi$ to $g \mapsto \restrict{g^\varphi}{\Phi}$ and return $X$ as canonical binary relational structure with large symmetry defect. In this case $\Phi$ becomes our new $\Gamma$. On the the hand, we determine a partition of $\Omega$ corresponding to $\Gamma = \Phi \mathbin{\dot\cup} (\Gamma\setminus \Phi)$ as in case~2a applying \cref{proc:pullback} and call the main procedure. Since $\abs{\Gamma \setminus \Phi} \leq \abs \Gamma/2$, the $\leq \abs\Omega$ subproblems are of length $\leq \abs\Omega/2$.

\paragraph{Case 3} $\abs S < \abs \Gamma/2$. \nopagebreak

We want to turn local asymmetry into global irregularity by constructing a $k$-ary relational structure on $\Gamma$ with large symmetry defect on $\Gamma \setminus S$. First, we align $S(\str x)$ and $S(\str y)$ in polynomial time as at the beginning of case 2. We can now assume that $S = S(\str x) = S(\str y)$ and restrict $\varphi$ to $g \mapsto \restrict{g^\varphi}{\Gamma \setminus S}$. On $S$ we call the main procedure as we did in case~2b. This accounts for $\leq\abs\Omega$ subproblems for strings of length $\leq \abs \Omega/2$.

We assign colors to the elements of $(\Gamma \setminus S)^k$. In order to guarantee canonicity, we let these colors originate in the local certificate data. For all $(\str u, \str v) \in \left\{(\str x, \str x), (\str x, \str y), (\str y, \str y)\right\}$ and $T,T' \in (\Gamma \setminus S)^k$ we call \cref{proc:localcert3} obtaining groups $\Iso_{G_{(T,T')}}\left(\str u^{W(T)}, \str v^{W(T')}\right)$. In total, these are no more that $3 \abs{\Gamma \setminus S}^k \ll \abs \Omega^{O(\log \abs \Omega)}$ calls. The colors that we infer correspond to the equivalence classes of the following relation:
\begin{equation} \label{eq:rel-iso}
	(\str u, T) \sim (\str v, T') \iff \Iso_{G_{(T,T')}}\left(\str u^{W(T)}, \str v^{W(T')}\right) \neq \emptyset.
\end{equation}

It is worth checking that this is indeed an equivalence relation. Transitivity and symmetry are clear. For reflexivity, recall that $T \subseteq \Gamma \setminus S$ where $S$ is the set of elements that are not fixed by $F \leq \Aut_G(\str x)$, a non-trivial group generated by the certificates of fullness. Thus $T$ is fixed by $F$ and $\emptyset \neq F = F_{(T)} \leq \Aut_{G_{(T)}}(\str x) \leq \Aut_{G_{(T)}}\left(\str x^{W(T)}\right)$. Indeed, $(\str x, T) \sim (\str x, T)$. Because the test sets are not full, not all $(\str x, T)$ for $T \in \binom{\Gamma \setminus S}{k}$ belong to the same equivalence class.

As in case~2b, we want to compute two canonical structures $\mathfrak{X}(\str x)$ and $\mathfrak{X}(\str y)$ associated to the two input strings. We thus infer two colorings from the equivalence relation. In the one associated to $\str x$ the set $T \in (\Gamma \setminus S)^k$ gets the color of $(\str x, T)$, in the coloring corresponding to $\str y$ the set $T$ carries the color of $(\str y, T)$. From now on, we will omit the string from our notation keeping in mind that we compute not one but two canonical structures.

The coloring can be extended canonically to $\binom{\Gamma \setminus S}{k}$; the color of a vector with multiply occurring entries will be a mediocre gray. Computing these color classes is not a problem. From \cref{proc:localcert3} we get the right-hand side of \cref{eq:rel-iso}. For determining the partition of $(\Gamma \setminus S)^k$ into the color classes we create $\nu = \abs{(\Gamma\setminus S)^k}$ entries in a union--find data structure, make $O(\nu^2)$ comparisons based on the knowledge from \cref{eq:rel-iso}, make $\leq \nu-1$ unions and finally request $\nu$ times the colors of each entry, i.e.\@ the set representatives in the data structure. In total, this takes $\ll \nu^2 + \nu \alpha(\nu, \nu) \ll \abs{\Omega}^{O(\log \abs \Omega)}$ elementary operations, where $\alpha$ is the (practically absolutely bounded) functional inverse of Ackermann's function, cf.\@ \cite{tarjan}. Consequently, the costs for the calls to \cref{proc:localcert3} far outweigh the extra expenses for the computation of the relational structure. We will call the described $k$-ary relational structure $\mathfrak{X} = (\Gamma \setminus S, c)$. % Finally, we extend the relational structure to a structure $\mathfrak{X}$ on $\Gamma$ coloring all vertices in $\Gamma^k \setminus (\Gamma \setminus S)^k$ with the new color emerald.

It remains to show that $\mathfrak{X}$ has large symmetry defect, cf.\@ \cref{def:symrelstruct}. Assume that $\mathfrak{X}$ admits a twin class $C$ with $\geq k$ elements. Then $C$ contains a set $T$ of size $k$. Since $T$ does not contain any duplicates, it is not gray. Because $C$ is a twin class, all possible orderings of $T$ must have the same color: Any transposition $(x\ y)$ of elements $x,y\in T$ is contained in $\Aut(\mathfrak{X})$ and may not push an ordering of $T$ out of its color class. Being able to find a permutation in $\Aut_{G_{(T)}}^{W(T)}(\str x)$ for every ordering of $T$ means that the ordering of $T$ is irrelevant. We have $\Alt(T) = \restrict{\varphi\left(\Aut_{G_{T}}^{W(T)}(\str x)\right)}{T}$ in contradiction to the fact that $T$, as a subset of $\Gamma \setminus S$, is not full. Consequently, all twin classes of $\mathfrak{X}$ are of size $< k$. The symmetry defect of $\mathfrak{X}$ is therefore by assumption $\geq \frac{\abs{\Gamma\setminus S} - k}{\abs{\Gamma \setminus S}} > 1- \frac{\abs\Gamma /10}{\abs \Gamma /2} = 4/5$. The structure that we return is $\mathfrak{X}$.

\section{Effect of the discovery of canonical structures}
\label{sec:align}

This section deals primarily with the outcomes of \cref{sec:agg-cert} and secondarily with those of \cref{sec:primitive}. We are given as before an epimorphism $\varphi: G \to \Alt(\Gamma)$ and a surjection $\iota: \Omega \to \binom{\Gamma}{k}$. In cases~1, 2b and 3 we have obtained canonical structures $\mathfrak{X}(\str x)$ and $\mathfrak{X}(\str y)$ associated with the two input strings $\str x$ and $\str y$. Here, $\Gamma$ refers to the subsets on which we have found the structures in the previous section. These structures can either be canonical colored $1/2$-partitions or canonically embedded $\kappa$-ary relational structures with relative symmetry defect $\geq 1/2$ for $\kappa \in \{2,k\}$. Of course, $\mathfrak{X}(\str x)$ and $\mathfrak{X}(\str y)$ are structures of the same kind as we would have had refuted isomorphicity otherwise. We follow Babai's align procedure, cf.\@ \cite[§11.1]{babai} and \cite[§5.3]{helfgott}.

In case of a $\kappa$-ary relational structure with large symmetry defect we call the the Extended Design Lemma, cf.\@ \cref{proc:design}, obtaining either a canonical colored $1/2$-partition  on $\Gamma$ or a canonically embedded non-trivial Johnson scheme on a subset $W \subseteq \Gamma$ such that $\abs W \geq \abs \Gamma/2$. Note that since the $W$ depends on the input structure, we have actually found subsets $W_{\str x}$ and $W_{\str y}$ carrying Johnson schemes. These sets do not have to be identical.

Hence, the structures are either colored partitions or Johnson schemes. These objects have the advantage that their automorphism group can be computed rapidly. That we will look for automorphisms in $G^\varphi$ only does not complicate the computation, since $G^\varphi = \Alt(\Gamma)$.

We aim at aligning the two structures $\mathfrak{X}(\str x)$ and $\mathfrak{X}(\str y)$, that is we want to shift $\str y$ to $\str y^{\sigma^{-1}}$ for some permutation $\sigma \in G$ such that $\mathfrak{X}(\str x) = \mathfrak{X}\left(\str y^{\sigma^{-1}}\right)$. Then the set of isomorphisms that we want to compute has to satisfy
\begin{equation} \label{eq:align-result}
	\Iso_G(\str x, \str y) = \Iso_{G_1}\left(\str x, \str y^{\sigma^{-1}}\right) \sigma \quad \text{for} \quad G_1^\varphi = \Aut(\mathfrak{X}(\str x)).
\end{equation}
This clearly results in a significant reduction of the problem size. Our goal is to compute suitable $G_1$ and $\sigma$. We distinguish two cases:

\paragraph{Johnson schemes} Suppose that $\mathfrak{X}(\str x) = \mathfrak{J}(m_1,t_1)$ and $\mathfrak{X}(\str y) = \mathfrak{J}(m_2,t_2)$ are non-trivial Johnson schemes on subsets $W_{\str x}, W_{\str y} \subseteq \Gamma$. The schemes are $G^\varphi$-isomorphic if $m_1 = m_2 \eqqcolon m$ and $t_1 = t_2 \eqqcolon t$. If this is not the case, we can refute isomorphicity. From \cref{lemma:johnsonaut} we know that $\Aut(\mathfrak{X}(\str y)) \cong \Aut(\mathfrak{X}(\str x)) \cong \Sym_m^{(t)}$. Since the schemes are non-trivial, $t \geq 2$. What justifies the reduction is that for $\str z \in \{\str x, \str y\}$ the action of $\Aut(\mathfrak{X}(\str z))$ on $W_{\str z}$ is entirely described by an action on $m \asymp \binom{m}{2}^{1/2} \ll \binom{m}{t}^{1/2} = \sqrt{\abs{W_{\str z}}} \leq \sqrt{\abs \Gamma}$ elements. We can hope for a reduction to a subgroup acting on a domain of size $m \ll \sqrt{\abs \Gamma}$.

We apply \cref{proc:iota} for $\Sym(W_{\str x})$, $W_{\str x}$, $m$ and $t$ as well as for the objects corresponding to $\str y$. In this way, we obtain two bijections $\iota_{\str x}: W_{\str x} \to \binom{\Lambda_{\str x}}{t}$ and $\iota_{\str y}: W_{\str y} \to \binom{\Lambda_{\str y}}{t}$ for sets $\Lambda_{\str x}$ and $\Lambda_{\str y}$. We establish any bijection between $\Lambda_{\str x}$ and $\Lambda_{\str y}$. This allows us to compute a $\tau \in \Alt(\Gamma)$ such that $W_{\str x}^\tau = W_{\str y}$. We lift $\tau$ along $\varphi$ using Schreier-Sims to a $\sigma \in G$ such that $\sigma^\varphi = \tau$ and hence $\mathfrak{X}(\str x) = \mathfrak{X}\left(\str y^{\sigma^{-1}}\right)$. \Cref{proc:iota} also provided an isomorphism $\psi_{\str x} : \Sym(W_{\str x}) \to \Sym(\Lambda_{\str x})$. We compute in polynomial time $G_1 = \varphi^{-1}(\Aut(\mathfrak{X}(\str x))) = \varphi^{-1}(\psi_{\str x}^{-1}(\Sym(\Lambda_{\str x}))$. Finally, we update $\varphi$ and $\iota$. The epimorphism is $\varphi: G_1 \to \Alt(\Lambda_{\str x})$, the composition of $\psi_{\str x}$ and a suitable restriction of $\varphi$. The new surjection is $\iota: \Omega \to \binom{\Lambda_{\str x}}{t}$. The reduction will prove beneficial, since by the above argument, $\abs{\Lambda_{\str x}} \ll \sqrt{\abs \Gamma}$. Updating $\Gamma \leftarrow \Lambda_{\str x}$ finalizes this polynomial time reduction.

\paragraph{Colored partitions} Two colored partitions are isomorphic if their color classes and sets are of the same size. This condition can be easily verified. By enumerating the color classes and blocks we construct a $\tau \in \Alt(\Gamma)$ which satisfies $\mathfrak{X}(\str x)^\tau = \mathfrak{X}(\str y)$. Lifting $\tau$ to $\sigma \in \varphi^{-1}(\{\tau\})$ completes the alignment of the two structures. Having mapped out the color classes of $\mathfrak{X}(\str x)$, we rapidly compute $\Aut(\mathfrak{X}(\str x))$ and its $\varphi$-preimage $G_1$. 
We have provided everything for \cref{eq:align-result}. Let us nevertheless see how the reduction continues. Let from now on $\str y \leftarrow \str y^{\sigma^{-1}}$ and $G \leftarrow G_1$. We want to further exploit the color classes in order to reduce the size of the subproblems significantly.

Let $\Delta_1, \dots, \Delta_t$ denote the color classes of $\mathfrak{X}(\str x) = \mathfrak{X}(\str y)$. They form a partition of $\Gamma$. Invoking \cref{proc:pullback}, we obtain a partition $\Omega_1, \dots, \Omega_s$ of $\Omega$. Our strategy is to iteratively process the partitions using the chain rule, cf.\@ \cref{lemma:isoaut}. Considering an additional $\Omega_i$ imposes a stronger condition on the isomorphisms between $\str x$ and $\str y$. For each $1 \leq i \leq s$ we obtain a separate subproblem. For being able to recur, we provide new auxiliary sets $\Gamma'$ and corresponding epimorphisms. With $G_0 \coloneqq G$, the iteration has the following form:
\begin{equation}
\Iso_G^{\Omega_1 \cup \dots \cup \Omega_{i+1}} (\str x, \str y) 
	= \Iso_G^{\Omega_1 \cup \dots \cup \Omega_{i}} (\str x, \str y) \cap  \Iso_G^{\Omega_{i+1}} (\str x, \str y)
	= \Iso_{G_i}^{\Omega_{i+1}} (\str x, \str y), \quad G_i \coloneqq 	\Iso_G^{\Omega_1 \cup \dots \cup \Omega_{i}} (\str x, \str y).
\end{equation}

If there exists a dominant color class, i.e.\@ an $i$ such that $\abs{\Delta_i} > \abs \Gamma/2$, we start with processing $\Delta_i$. It corresponds to one $\Omega_j = \iota^{-1}\binom{\Delta_i}{k}$ satisfying $\abs{\Omega_j} > \abs\Omega/2$. By \cref{def:colored-partition}, $\Delta_i$ decomposes into $\geq 2$ equally large blocks of size $l \geq 2$. Let $\Gamma'$ denote the set of blocks. Any permutations in $\Alt(\Delta_i) = \restrict{\Alt(\Gamma)_{\Delta_i}}{\Delta_i}$ must now respect this partition. This means that it in fact acts as $\Alt(\Gamma')$ on the blocks. Since $\abs{\Gamma'} = \abs{\Delta_i}/l \leq \abs{\Gamma}/2$, we have reduced the problem significantly replacing $\Alt(\Gamma)$ by $\Alt(\Gamma')$. The updated epimorphism $\varphi: G \to \Alt(\Gamma')$ is inferred from the action on the blocks. We process the other color classes as if there was no dominant class.

Let $\Omega_i$ be the set of interest. So far, we have computed $G_{i-1}$. We can assume that $\Omega_i$ is not dominant because we would have otherwise treated it as above. Hence, $\abs{\Omega_i} \leq \abs{\Omega}/2$. The computation of $G_{i} = \Iso^{\Omega_i}_{G_{i-1}}(\str x,\str y)$ results therefore in a single call to the main procedure for strings of length $\leq \abs\Omega/2$.

The reduction is now complete. We will see in the next chapter that we indeed took only quasi-polynomial time.

\chapter{Time complexity}
\label{sec:complexity}

Throughout the algorithm we encountered subroutines of different complexities. Simple operations such as the computation of orbits took polynomial time while the most costly intermediate results required $\abs\Omega^{O(\log \abs\Omega)}$ many steps. Neither of these two groups of operations endanger our desired bound on the execution time. However, we performed a wide range of reductions, cf.\@ \cref{fig:overview}, which require a careful analysis as they incurred multiplicative costs.

We will measure the complexity of the algorithm and its constituents in elementary operations. Recalling \cref{sec:schreier-sims}, we do not need to make a difference between genuinely elementary operations and most group operations. The strategy for obtaining an overall complexity estimate is to look at the complexities of parts of the algorithm. For this purpose we introduce the following quantities:

\begin{defn} \label{def:complex} 
We consider instances of problems for the inputs $\str x, \str y: \Omega \to \Sigma$ and $G \leq \Sym(\Omega)$. Let $n = \abs\Omega$. 
\begin{enumerate}
\item Let $T(x)$ for $x \geq 1$ denote the maximum cost of the main procedure, cf.\@ \cref{sec:overview}, i.e.\@ the number of operations for computing $\Iso_G(\str x, \str y)$ for arbitrary strings of length $n \leq x$.
\item Let $T_{\text{trans}}(x)$ for $x \geq 1$ denote the maximum costs of the main procedure when restricted to transitive input groups with $n \leq x$.
\item Let $T_{\text{Jh}}(x)$ for $x \geq 1$ denote the maximum costs of the constituents of the main procedure that deal with groups acting as Johnson groups, cf.\@ \cref{sec:primitive}. In this case we are equipped with a set $\Gamma$ of size $m$ and an epimorphism $\varphi: G \to \Alt(\Gamma)$ accompanied by a surjection $\iota: \Omega \to \binom{\Gamma}{k}$. We suppose $2 \leq m \leq x$ and $n \leq x$. 
\item Let $T_{\text{Jh}}(x,y)$ for $x \geq y \geq 5$ denote the maximum costs of those constituents of the main procedure fulfilling the conditions for $T_{\text{Jh}}(x)$ and additionally $\log x \ll m \leq y$ for $m = \abs\Gamma$. This is the setting of \cref{sec:agg-cert}. \label{item:complex-jh}
\item Let $T_{\text{struct}}(x,y)$ for $x \geq y \geq 5$ denote the maximum costs of \cref{sec:align} excluding \cref{proc:design}. This is the special case of $T_{\text{Jh}}(x,y)$ when we have found a Johnson scheme or a colored partition.
\item Let $T(x,y)$ for $x \geq y \geq 5$ denote the maximum costs of the main procedure with the restriction that any possibly contained subproblem of Johnson type complies with the conditions in \cref{item:complex-jh} for $x$ and $y$.
\end{enumerate}
\end{defn}

With these quantities we estimate the complexity of the corresponding subroutines. For the main procedure we want to obtain a bound of the form $T(x) \ll \exp\left(\log(x)^{O(1)}\right)$. Each reduction gives us an upper bound for the complexity of the original problem in terms of the complexity of the subproblems it is reduced to. For example, at the beginning of the main procedure we pass from a general group to a transitive group by recurring on orbits. This yields an upper bound for $T$ in terms of $T_{\text{trans}}$, cf.\@ \cref{eq:complex2}.

After some rather simple reduction, we have in \cref{sec:primitive} arrived at the core of the algorithm: the treatment of Johnson groups. Consequently, estimates on $T_{\text{Jh}}$ will be crucial. We distinguish two types of reductions. Either we are able to look at shorter strings passing from $T_{\text{Jh}}(x,y)$ to say $T_{\text{Jh}}(x/2,y)$ or we can reduce the size of the Johnson group that is acting on blocks passing from $T_{\text{Jh}}(x,y)$ to $T_{\text{Jh}}(x,y/2)$, say. In the most favorable situation we can do both simultaneously. The multiplicative costs incurred by these reductions are critical.

A precondition for the treatment of Johnson groups using local certificates is that $\log \abs\Omega \ll \abs\Gamma$, cf.\@ \cref{thm:afforb}. However, if the contrary holds, we can solve the problem brutally by strong Luks reduction, cf.\@ \cref{eq:complex4}.
We note that $T_{\text{Jh}}(x) = T_{\text{Jh}}(x,y)$ and $T(x) = T(x,x)$. 
Overtly, all functions in \cref{def:complex} can be assumed to be monotonically non-decreasing in all their parameters. %  and $T(x) =T(x,x)$.

\paragraph{Estimates for the first steps} We will start with the cases distinguished in \cref{sec:first-steps}: That we can solve the problem brutally for sufficiently small strings means that there exists an absolute constant $C > 0$ such that
\begin{equation} \label{eq:complex1}
\forall x \leq C: \ T(x) = O(1).
\end{equation}
Next we consider the case in which $G$ acts intransitively. We perform strong Luks reduction on the orbits, cf.\@ \cref{proc:strong-luks}, recurring on transitive groups. For a string of length $\floor{x}$ the possible partitions into orbits are described by all choices of positive integers $n_i$ such that $\sum n_i = \floor{x}$. We obtain,
\begin{equation} \label{eq:complex2}
T(x) \leq \max\left\{ \sum T_{\text{trans}}(n_i) + x^{O(1)} \right\} 
\end{equation}
where the maximum is taken over all choices of integers $n_i$. The polynomial contribution $x^{O(1)}$ must be added because we have to map out the orbits. The next step was to apply \cref{thm:cameron}. The case of $G/N$ being a Mathieu group (\cref{item:cameron2}) is covered by \cref{eq:complex1}. The other two cases contribute
\begin{equation} \label{eq:complex3}
T_{\text{trans}}(x) \leq \max_{2 \leq m \leq x}\left\{m^{2+\log_2 m}\left( T(x/m) + x^{O(1)} \right),\ m \left(T_{\text{Jh}}(x) + x^{O(1)} \right)\right\}
\end{equation}
with $m = \abs{\mathcal{B}}$ in the notation of \cref{sec:first-steps}. \Cref{item:cameron3} in Cameron-Maróti causes a reduction to $m^{2+\log_2 m}$ subproblems of size $\leq x/m$. If \cref{item:cameron1} holds, we perform weak Luks reduction to an index $m$ subgroup of order $> m^{\log_2 m}$ acting as Johnson group on the coarser blocks $\mathcal{B}'$. The number of blocks is $\abs{\mathcal{B}'} \leq \abs{\mathcal{B}}\leq \abs\Omega \leq x$ in compliance with the definition of $T_{\text{Jh}}(x)$. We further excluded the case $\abs\Gamma \ll \log \abs\Omega$ by recurring on $\abs\Omega^{O(\log \abs\Omega)}$ many subproblems of size $\leq \abs\Omega/2$. Hence,
\begin{equation} \label{eq:complex4}
T_{\text{Jh}}(x) \leq x^{O(\log x)}\left( T(x/2) + 1 \right).
\end{equation}

We proceed with the analysis of the steps in \cref{sec:primitive}. We are in the case of $G$ acting as Johnson group. The bounds that we obtain will be on $T_{\text{Jh}}(x)$. On the trivial case $k=1$ we spend $x^{O(1)}$ steps. After having excluded the trivial case, we compute in polynomial time canonical structures. This leads to either a Luks reduction (large symmetry) or the application of the Design Lemma to a $k$-ary structure with $k \ll \log x$, cf.\@ \cref{proc:design}. We obtain,
\begin{equation} \label{eq:complex5}
T_{\text{Jh}}(x) \leq \max \left\{ O(\log x) \left(T(x/2) + x^{O(1)}\right),\ x^{O(\log x)} \left(T_{\text{struct}}(x,x)+1\right) \right\}.
\end{equation}

\paragraph{Estimates for the case of local certificates} We estimate $T_{\text{Jh}}(x,y)$. However, no reduction on $y$ happens in this part of the algorithm. We can hence trivially estimate $y \leq x$. \Cref{sec:agg-cert} begins with the computation of local certificates. We apply \cref{proc:localcert} in total $x^{O(\log x)}$ times. This incurs $x^{O(\log x)}$ many subproblems of length $\leq x/2$ at additive costs of $x^{O(\log x)}$ because we chose the parameter $k$ such that $\abs\Omega/k \leq \abs\Omega/2$. In case~1 we find a colored partition suitable for \cref{sec:align} in polynomial time. 
Case~2a results in a reduction to linearly many general problems for strings with length $\leq x/2$. The two individualizations in case~2b incur multiplicative costs of $x^{O(1)}$. Afterwards we obtain a canonical binary relational structure and $\leq x$ subproblems of length $\leq x/2$. In case~3 we have to solve a problem of the same size before we apply \cref{proc:localcert3}. Again by the choice of the parameter $k$, \cref{proc:localcert3} has the same complexity as \cref{proc:localcert} noting that the higher factor is absorbed by the implicit constants. The outcome in case~3 is a $k$-ary canonical relational structure with $k \ll \log x$. The administration of the union--find structure is relatively expensive and accounts for an additional quasi-polynomial summand. 

The relational structures computed in cases 2b and 3 need to be treated by the Design Lemma, cf.\@ \cref{proc:design}, before they can be processed by \cref{sec:align}. The necessary individualizations account for a factor of $x^{O(\log x)}$. Polynomial costs are absorbed by the initial costs. Consequently, we get as estimate for $T_{\text{Jh}}(x,y)$,
\begin{center}
\begin{tabular}{ll}
Initial & $\leq x^{O(\log x)}(T(x/2) + 1)$ \\
Case 1  & $\leq T_{\text{struct}}(x,y)$ \\
Case 2a & $\leq 3x T(x/2)$ \\
Case 2b & $\leq x^{O(1)} \left(x^{O(\log x)} \left(T_{\text{struct}}(x,y)+1\right) + xT(x/2) + 1 \right)$ \\
Case 3  & $\leq xT(x/2) +  x^{O(\log x)}(T(x/2) + 1) + x^{O(\log x)} + x^{O(\log x)}\left(T_{\text{struct}}(x,y)+1\right)$
\end{tabular}
\end{center}
The overall estimate for $T_{\text{Jh}}(x,y)$ is the sum of the initial costs and the maximum of the costs for the treatment of the four cases. In total, we obtain
\begin{equation} \label{eq:complex6}
T_{\text{Jh}}(x,y) \leq x^{O(\log x)} \left(T_{\text{struct}}(x,y) +  T(x/2)+1 \right)
\end{equation}
Even if we incorporate \cref{eq:complex4,eq:complex5}, this bound does not enlarge qualitatively. \Cref{eq:complex6} is therefore the overall estimate for $T_{\text{Jh}}(x,y)$ taking all cases of \cref{sec:primitive,sec:agg-cert} into account.

\paragraph{Estimates for the reduction following canonical structures} It remains to estimate $T_{\text{struct}}(x,y)$, i.e.\@ the steps in \cref{sec:align}. We exclude the Design Lemma, cf.\@ \cref{proc:design}, and suppose that the given structures are either canonical colored partitions or Johnson schemes. In the latter case, we obtain a single subproblem of Johnson type, i.e.\@ a subproblem that is treated by \cref{sec:primitive,sec:localcert}. The new set $\Gamma$ is of size $\sqrt{y} \leq y/2$ assuming that $y$ is sufficiently large.
In the former case, we process each of the $\leq x$ color classes separately. If none of the color classes is dominant, we obtain $\leq x$ subproblems of size $\leq x/2$. In the contrary case, the procedure yields additionally a subproblem of Johnson type. In total, we have
\begin{align}
T_{\text{struct}}(x,y) &\leq \max\left\{T_{\text{Jh}}(x,y/2),\ T_{\text{Jh}}(x,y/2) + xT(x/2),\ xT(x/2) \right\} + x^{O(1)} \nonumber \\
&\leq T_{\text{Jh}}(x,y/2) + xT(x/2) + x^{O(1)} \label{eq:complex7}
\end{align}

\paragraph{The bottom line} We can now combine \cref{eq:complex6,eq:complex7} obtaining an overall estimate on $T_{\text{Jh}}$ independent of $T_{\text{struct}}$:
\begin{equation} \label{eq:complex8}
T_{\text{Jh}}(x,y) \leq x^{O(\log x)} \left( T_{\text{Jh}}(x,y/2) + T(x/2) + 1 \right)
\end{equation}
Having this bound at hand, we can derive an overall estimate for the entire algorithm recalling \cref{eq:complex2,eq:complex3}. For simplicity, we introduce the at first glance artificial quantity $T(x,y)$, cf.\@ \cref{def:complex}. Clearly, $T_{\text{Jh}}(x,y) \leq T(x,y)$ holds.
\begin{align}
T(x,y) & \leq x^{O(\log x)} \left(T(x/2, y) +1 \right) + x T_{\text{Jh}}(x, y) \nonumber \\
& \leq x^{O(\log x)}\left( T(x/2, y) + T_{\text{Jh}}(x,y/2)  + 1 \right) \nonumber \\
& \leq x^{O(\log x)}\left( T(x/2, y) + T(x,y/2) + 1 \right) \label{eq:complex13}
\end{align}
In other words, we can reduce the problem to quasi-polynomially many subproblems of half length and of halved Johnson parameter. It remains to solve this recursion. In contrary to other estimates, it is now crucial to scrutinize the implicit constant hidden in the $O$-term. Since it must be independent of $x$ and $y$, we cannot use it to absorb high order terms when applying the recursion. \Cref{thm:complex} provides a solution of \cref{eq:complex13}. We verify the preconditions. The base case for small $x$ is given by \cref{eq:complex1}. For the case of $y$ small and $x$ arbitrary we revert to \cref{eq:complex4}: The recursion $T(x) \leq x^{O(\log x)}(T(x/2)+1)$ implies the generous bound $T(x) \leq x^{O\left(\log x\right)^2} = \exp\left(O(\log x)^3\right)$ under the assumption of \cref{eq:complex1} after $O(\log_2 x)$ many substitutions. The constants $C_0$ and $C_1$ are extracted from the $O$-terms.

\begin{theorem} \label{thm:complex}
Let $T: \mathbb{R}_{\geq 0} \times \mathbb{R}_{\geq 0} \to \mathbb{R}_{\geq 0}$ be monotonically non-decreasing. Suppose that $T$ satisfies for all $x,y \in \mathbb{R}_{\geq 0}$ the recursion
\begin{equation} \label{eq:recrec}
T(x,y) \leq x^{C_0\log x} \left( T(x/2, y) + T(x,y/2) +1 \right)
\end{equation}
for a constant $C_0 > 0$. Furthermore, suppose that there exists a constant $C_1 > 0$ and $x_0, y_0 \in \mathbb{N}$, $x_0,y_0 \geq 2$ such that 
\begin{equation} \label{eq:recbound}
T(x,y) \leq \exp \left( C_1 \left( \log x + \log y \right)^3 \right)
\end{equation}
for all $x \leq x_0$, $y \in \mathbb{R}_{\geq 0}$ and for all $y \leq y_0$, $x \in \mathbb{R}_{\geq 0}$. Then there exists a constant $C > 0$ depending only on $C_0$ and $C_1$ such that \cref{eq:recbound} holds for all $x, y \in \mathbb{R}_{\geq 0}$ with $C$ instead of $C_1$.
\end{theorem}
\begin{proof}
It is crucial that the constant $C$ must be independent of the parameters $x$ and $y$. Thus, when showing the claim inductively, we must ensure that we can choose $C$ uniformly. Of course, $C$ must be greater than $C_1$. By monotonicity we do not need to distinguish integer and non-integer parameters. The proof is based on a double induction: Choose integers $x \geq x_0$, $y \geq y_0$ arbitrarily. We assume that \cref{eq:recbound} holds for all $x', y'$ satisfying $x' \leq x -1$ and $y' \leq y$ or $x' \leq x$ and $y' \leq y -1$. It is to show that \cref{eq:recbound} holds for $x$ and $y$. Let $\lambda(x,y)$ denote $\log x+\log y$. The only available tools are the recursion in \cref{eq:recrec} and the induction hypothesis:
\begin{align}
T(x, y) & \leq  x^{C_0\log x} \left( T(x/2, y) + T(x,y/2) +1 \right) \nonumber \\
& \leq 3 x^{C_0\log x} \exp \left( C \left( \lambda(x,y) - \log 2\right)^3 \right) \nonumber \\
& \leq \exp \left( C \lambda(x,y)^3 + (C_0 - 3C \log 2)\lambda(x,y)^2 + 3C (\log 2)^2\lambda(x,y) - C(\log 2)^3 + \log 3\right) \label{eq:complexind}
\end{align}

Weakening the assumption, we set $C_0 \coloneqq \max\{C_0, 1\}$. This will simplify what follows. We choose $C \coloneqq \max\left\{C_1, C_0/\log{2} \right\} > C_0/(3\log 2)$ independently of $x$ and $y$. Now the quadratic polynomial in $\lambda(x,y)$ in \cref{eq:complexind} tends to $-\infty$ for $x,y\to \infty$. Hence, there must exist a lower bound $\xi$ for $x$ and $y$ such that this polynomial is negative. Problematically, this bound depends on $C_0$ imposing an a priori nonviable condition on $x_0$ and $y_0$. However, we can elementarily compute this lower bound as
\[
\xi \coloneqq \exp \left( \frac{3 C_0 \log{2} + \sqrt{{\left(\log{ 2}\right)^{2}}\, {{C_0}^{2}}+8 C_0 \log{3}} }{8 C_0} \right).
\]
This decreases monotonically approaching $\sqrt{2}$ for $C_0 \to \infty$. Under the assumption $C_0 \geq 1$, we obtain $\xi < 2$. Thus, supposing \cref{eq:recbound} for $x_0, y_0 \geq 2$ suffices. This is an absolute condition on $x_0$ and $y_0$. We can now estimate the polynomial in \cref{eq:complexind} by zero from above completing the induction.
\end{proof}
By definition, $y \leq x$. We hence conclude that the overall execution time is bounded by
\[
T(x) \leq \exp \left(C \left(\log x\right)^3 \right)
\]
for a constant $C > 0$. Babai's algorithm for deciding the String Isomorphism Problem is of quasi-polynomial complexity. As a corollary, we obtain the desired bound on the complexity of the Graph Isomorphism Problem.

\chapter*{Acknowledgements}
\addcontentsline{toc}{chapter}{Acknowledgements}
I wholeheartedly thank my supervisors Prof.~Dr.~Harald Andrés Helfgott and Prof.~Dr.~Stephan Waack for the opportunity of writing this thesis, their suggestions, comments and feedback. Moreover, I thank Daniele Dona for his helpful explanations. Finally, I thank Johannes Hochwart, Olivia Howe and Katarina Hahn for their corrections and comments.

\bibliography{literature}
\bibliographystyle{amsalpha}
\end{document}